\documentclass[11pt,a4paper]{article}

\usepackage{titlesec}
\usepackage{fancyhdr}
\usepackage{a4wide}
\usepackage{graphicx}
\usepackage{float}
\usepackage{amssymb}
\usepackage{amsmath}
\usepackage{amsfonts}
\usepackage{dsfont}
\usepackage{amsthm}
\usepackage{color}
\usepackage{mathrsfs}
\usepackage{array}
\usepackage{eucal}
\usepackage{tikz}
\usepackage[T1]{fontenc}
\usepackage{inputenc}
\usepackage[english]{babel}
\usepackage{lmodern}
\usepackage{hyperref}
\usepackage{geometry}
\usepackage{changepage}
\usepackage{bm}
\changepage{0pt}{}{}{}{}{0pt}{}{0pt}{6pt}
\usepackage[numbers]{natbib}
\usepackage{bbm}
\usepackage{pgfplots}

\usepackage{hyperref}
\hypersetup{
pdfpagemode=none,
pdftoolbar=true,        % show Acrobat's toolbar?
pdfmenubar=true,        % show Acrobats menu?
pdffitwindow=false,     % window fit to page when opened
pdfstartview={Fit},    % fits the width of the page to the window
pdftitle={On the breathing of spectral bands in periodic quantum waveguides with inflating resonators},    % title
pdfauthor={L. Chesnel, S.A. Nazarov},     % author
pdfsubject={},  % subject of the document
pdfcreator={L. Chesnel, S.A. Nazarov},   % creator of the document
pdfproducer={}, % producer of the document
pdfkeywords={}, % list of keywords
pdfnewwindow=true,      % links in new window
colorlinks=true,       % false: boxed links; true: colored links
linkcolor=magenta,          % color of internal links
citecolor=red,        % color of links to bibliography
filecolor=cyan,      % color of file links
urlcolor=blue           % color of external links
}

\newcommand{\dsp}{\displaystyle}

\newcommand{\eps}{\varepsilon}
\newcommand{\om}{\omega}
\newcommand{\Om}{\Omega}
\newcommand{\mrm}[1]{\mathrm{#1}}

\newcommand{\Cplx}{\mathbb{C}}
\newcommand{\N}{\mathbb{N}}
\newcommand{\R}{\mathbb{R}}
\newcommand{\Z}{\mathbb{Z}}

\newcommand{\mL}{\mrm{L}}
\newcommand{\mH}{\mrm{H}}

\newcommand{\mX}{\mrm{X}}

\renewcommand{\ker}{\mrm{ker}}

\renewcommand{\dim}{\mrm{dim}}

\newtheorem{theorem}{Theorem}[section]
\newtheorem{lemma}[theorem]{Lemma}
\newtheorem{remark}[theorem]{Remark}
\newtheorem{definition}[theorem]{Definition}

\newtheorem{proposition}[theorem]{Proposition}

\begin{document}

~\vspace{-0.6cm}
\begin{center}

{\sc \bf\fontsize{20}{20}\selectfont 
On the breathing of spectral bands in periodic \\[8pt]
quantum waveguides with inflating resonators}
\end{center}

\begin{center}
\textsc{Lucas Chesnel}$^1$, \textsc{Sergei A. Nazarov}$^{2}$\\[16pt]
\begin{minipage}{0.95\textwidth}
{\small
$^1$ Inria, Ensta Paris, Institut Polytechnique de Paris, 828 Boulevard des Mar\'echaux, 91762 Palaiseau, France;\\
$^2$ Institute of Problems of Mechanical Engineering RAS, V.O., Bolshoi pr., 61, St. Petersburg, 199178, Russia;\\
E-mails: \texttt{lucas.chesnel@inria.fr}, \texttt{srgnazarov@yahoo.co.uk} \\[-14pt]
\begin{center}
(\today)
\end{center}
}
\end{minipage}

\vspace{0.2cm}

\begin{minipage}{0.95\textwidth}
\noindent\textbf{Abstract.} We are interested in the lower part of the spectrum of the Dirichlet Laplacian $A^\eps$ in a thin waveguide $\Pi^\eps$ obtained by repeating periodically a pattern, itself constructed by scaling an inner field geometry $\Om$ by a small factor $\eps>0$. The Floquet-Bloch theory ensures that the spectrum of $A^\eps$ has a band-gap structure. Due to the Dirichlet boundary conditions, these bands all move to $+\infty$ as $O(\eps^{-2})$ when $\eps\to0^+$. Concerning their widths, applying techniques of dimension reduction, we show that the results depend on the dimension of the so-called space of almost standing waves in $\Om$ that we denote by $\mX_\dagger$. Generically, i.e. for most $\Om$, there holds $\mX_\dagger=\{0\}$ and the lower part of the spectrum of $A^\eps$ is very sparse, made of bands of length at most $O(\eps)$ as $\eps\to0^+$. For certain $\Om$ however, we have $\dim\,\mX_\dagger=1$ and then there are bands of length $O(1)$ which allow for wave propagation in $\Pi^\eps$. The main originality of this work lies in the study of the behaviour of the spectral bands when perturbing $\Om$ around a particular $\Om_\star$ where $\dim\,\mX_\dagger=1$. We show a breathing phenomenon for the spectrum of $A^\eps$: when inflating $\Om$ around $\Om_\star$, the spectral bands rapidly expand before shrinking. In the process, a band dives below the normalized threshold $\pi^2/\eps^2$, stops breathing and becomes extremely short as $\Om$ continues to inflate.\\[4pt]
\end{minipage}
\end{center}

\section{Introduction}

\begin{figure}[!ht]
\centering
\begin{tikzpicture}[scale=1]
\path[draw=black,fill=gray!30,smooth,rounded corners=0mm](0,-0.3)--(0.9,0.2)--(1,0.9)--(-0.7,1.4)--(-1.1,-0.1)--cycle;
\draw[fill=gray!30,draw=none](-2,-1/2) rectangle (2,1/2);
\draw[black] (-2,-1/2)--(2,-1/2);
\draw[black] (-2,1/2)--(-0.94,1/2);
\draw[black] (2,1/2)--(0.92,1/2);
\draw[black,<->] (1.8+0.6,-1/2)--(1.8+0.6,1/2);
\node at (1.8+0.8,0) {\small $1 $};
\node at (0,0) {\small $\Om $};
\end{tikzpicture}\qquad\qquad
\begin{tikzpicture}[scale=0.6]
\path[draw=black,fill=gray!30,smooth,rounded corners=0mm](0,-0.3)--(0.9,0.2)--(1,0.9)--(-0.7,1.4)--(-1.1,-0.1)--cycle;
\draw[fill=gray!30,draw=none](-5,-1/2) rectangle (5,1/2);
\draw[black] (-0.95,1/2)--(-5,1/2)--(-5,-1/2)--(5,-1/2)--(5,1/2)--(0.92,1/2);
\draw[black,<->] (3,-1/2)--(3,1/2);
\draw[black,<->] (-5,-0.8)--(5,-0.8);
\node at (0,-1.4) {\small $1$};
\draw[black,line width=0.7mm] (-5,-1/2)--(-5,1/2);
\draw[black,line width=0.7mm] (5,-1/2)--(5,1/2);
\node at (5+0.8,-0.1) {\small $\partial\om^\eps_+$};
\node at (-5-0.8,-0.1) {\small $\partial\om^\eps_-$};
\node at (3+0.4,0) {\small $\eps $};
\node at (0,0) {\small $\om^\eps $};
\end{tikzpicture}\\[20pt]
\begin{tikzpicture}[scale=0.22]
\path[draw=black,fill=gray!30,smooth,rounded corners=0mm](0,-0.3)--(0.9,0.2)--(1,0.9)--(-0.7,1.4)--(-1.1,-0.1)--cycle;
\draw[fill=gray!30,draw=none](-5,-1/2) rectangle (5,1/2);
\draw[black] (-0.95,1/2)--(-5,1/2);
\draw[black] (-5,-1/2)--(5,-1/2);
\draw[black] (5,1/2)--(0.92,1/2);
\begin{scope}[xshift=19.95cm]
\path[draw=black,fill=gray!30,smooth,rounded corners=0mm](0,-0.3)--(0.9,0.2)--(1,0.9)--(-0.7,1.4)--(-1.1,-0.1)--cycle;
\draw[fill=gray!30,draw=none](-5,-1/2) rectangle (5,1/2);
\draw[black] (-0.95,1/2)--(-5,1/2);
\draw[black] (-5,-1/2)--(5,-1/2);
\draw[black] (5,1/2)--(0.92,1/2);
\node at (7,0) {\small $\Pi^\eps $};
\end{scope}
\begin{scope}[xshift=9.95cm]
\path[draw=black,fill=gray!30,smooth,rounded corners=0mm](0,-0.3)--(0.9,0.2)--(1,0.9)--(-0.7,1.4)--(-1.1,-0.1)--cycle;
\draw[fill=gray!30,draw=none](-5,-1/2) rectangle (5,1/2);
\draw[black] (-0.95,1/2)--(-5,1/2);
\draw[black] (-5,-1/2)--(5,-1/2);
\draw[black] (5,1/2)--(0.92,1/2);
\end{scope}
\begin{scope}[xshift=-9.95cm]
\path[draw=black,fill=gray!30,smooth,rounded corners=0mm](0,-0.3)--(0.9,0.2)--(1,0.9)--(-0.7,1.4)--(-1.1,-0.1)--cycle;
\draw[fill=gray!30,draw=none](-5,-1/2) rectangle (5,1/2);
\draw[black] (-0.95,1/2)--(-5,1/2);
\draw[black] (-5,-1/2)--(5,-1/2);
\draw[black] (5,1/2)--(0.92,1/2);
\end{scope}
\begin{scope}[xshift=-19.95cm]
\path[draw=black,fill=gray!30,smooth,rounded corners=0mm](0,-0.3)--(0.9,0.2)--(1,0.9)--(-0.7,1.4)--(-1.1,-0.1)--cycle;
\draw[fill=gray!30,draw=none](-5,-1/2) rectangle (5,1/2);
\draw[black] (-0.95,1/2)--(-5,1/2);
\draw[black] (-5,-1/2)--(5,-1/2);
\draw[black] (5,1/2)--(0.92,1/2);
\end{scope}
\begin{scope}[xshift=-29.95cm]
\path[draw=black,fill=gray!30,smooth,rounded corners=0mm](0,-0.3)--(0.9,0.2)--(1,0.9)--(-0.7,1.4)--(-1.1,-0.1)--cycle;
\draw[fill=gray!30,draw=none](-5,-1/2) rectangle (5,1/2);
\draw[black] (-0.95,1/2)--(-5,1/2);
\draw[black] (-5,-1/2)--(5,-1/2);
\draw[black] (5,1/2)--(0.92,1/2);
\end{scope}
\end{tikzpicture}
\caption{Geometries of $\Om$ (top left), $\om^\eps$ (top right) and $\Pi^\eps$ (bottom). \label{Geometry}} 
\end{figure}
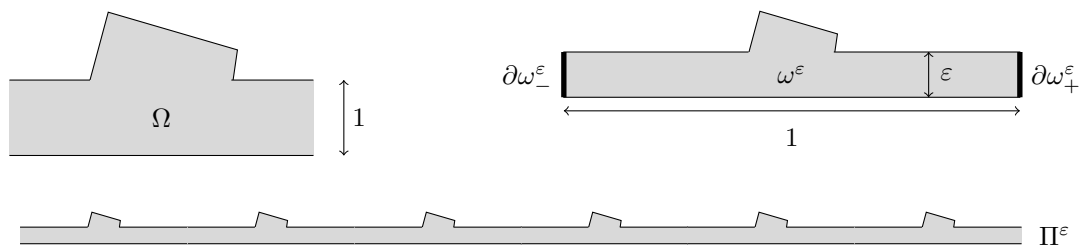

Motivated in particular by the extraordinary properties of graphene, which could make it possible to design revolutionary devices \cite{GeNo07,Geim09,NGPNG09}, major efforts have been made to understand wave propagation phenomena in quantum waveguides. Mathematically, this leads to study the spectrum of the Laplace operator with Dirichlet boundary conditions in domains made of thin ligaments forming unbounded periodic lattices. In the literature, different kind of geometries have been considered and we refer the reader to \cite{Poyn84,Kuch02,KuPo07,Grie08INI,KuchINI,ExKo15} for review studies. In the present article, we work in a waveguide $\Pi^\eps\subset\R^2$ obtained by repeating periodically a pattern $\om^\eps$, itself constructed by shrinking an inner field geometry $\Om$ by a factor $\eps>0$ small (see Figure \ref{Geometry}). We denote by $A^\eps$ the corresponding Dirichlet Laplacian in $\Pi^\eps$. From the Floquet-Bloch-Gelfand theory \cite{Gelf50,Kuch82,Skri87,Kuch93}, we know that the spectrum of $A^\eps$ has a band-gap structure, the bands being generated by the eigenvalues of a spectral problem set on the periodicity cell $\om^\eps$ with quasi-periodic boundary conditions involving the Floquet-Bloch parameter.\\
\newline
The first goal of this article is to recall how to study the features of the spectrum of $A^\eps$ as the ligaments become thinner and thinner, i.e. as $\eps\to0^+$. To proceed, following D. Grieser in \cite{Grie08} (see also \cite{Post05,MoVa07,CaEx07,MoVa08}), we use techniques of dimension reduction to derive 1D models for the spectral problem in $\om^\eps$. Then by exploiting these 1D models depending on the Floquet-Bloch parameter, for which explicit computations can be done, we get information on the spectral bands. The results that we obtain can be summarized as follows.\\ 
\newline
First, due to the Dirichlet boundary conditions, all the bands move to $+\infty$ as $O(\eps^{-2})$ when $\eps\to0^+$. Concerning the width of the bands, the results directly depend on the properties of the Dirichlet Laplacian $A^\Om$ in the near field geometry $\Om$. Note that $\Om$ is an unbounded waveguide with two infinite outlets of height one. The first spectral bands of $A^\eps$ are extremely short, of length $O(e^{-\delta/\eps})$, $\delta>0$, and their number coincides with the number of eigenvalues of $A^{\Om}$ smaller or equal to the lower bound of the continuous spectrum ($\pi^2$ in our case).\\
\newline
The width of the next spectral bands depends on the geometry of $\Om$ and more precisely on the dimension of some space $\mX_\dagger$ of almost standing waves in $\Om$. More precisely, $\mX_\dagger$ is defined as the space of bounded functions which are not exponentially decaying at infinity and which satisfy the homogeneous Helmholtz equation in $\Om$ with a spectral parameter equal to the lower bound of the continuous spectrum of $A^{\Om}$.\\
\newline
Generically, i.e. for most $\Om$, we have $\mX_\dagger=\{0\}$ and in that situation, the length of the next spectral bands of $A^\eps$ is $O(\eps)$. As a consequence, for most spectral parameters, waves can not propagate in $\Pi^\eps$. For certain particular inner field geometries $\Om$ however, there holds $\dim\,\mX_\dagger=1$. In that case, the length of the next spectral bands of $A^\eps$ in the corresponding periodic domains $\Pi^\eps$ is $O(1)$ and so they are much more possibilities for waves to propagate.\\
\newline
The second goal of this work, which makes the main originality of this article, is to study the behaviour of the spectral bands of $A^\eps$ when perturbing $\Om$ around a particular $\Om_\star$ where $\dim\,\mX_\dagger=1$. To precise the problem, assume that $\Om$ is defined via a geometrical parameter $\rho$ such that $\dim\,\mX_\dagger=1$ for $\rho=\rho_\star$ and $\mX_\dagger=\{0\}$ for $\rho-\rho_\star\ne0$ small. From these near field geometries, we can define a family of thin periodic waveguides that we denote by $\Pi^{\rho,\eps}$. In $\Pi^{\rho_\star,\eps}$, the bands of the spectrum of $A^\eps$ are rather large whereas in $\Pi^{\rho,\eps}$ for $\rho\ne\rho_\star$, they are short. How to describe the transition? To answer this question, we will propose a 1D model for the spectrum incorporating the dependence with respect to the parameter $\rho$. It will allow us to highlight a phenomenon of breathing of the spectrum: when inflating $\Om$ around $\Om_\star$, the spectral bands of $A^\eps$ suddenly expand and then shrink. Additionally, in the process a band dives below the normalized threshold $\pi^2/\eps^2$, stops breathing, ``dies'' and becomes extremely extremely short as the near field geometry $\Omega$ continues to inflate.\\
\newline
Let us mention that the spectrum of the Dirichlet Laplacian was studied in different 3D periodic waveguides in \cite{BaMaNa17,ChNa23} (see also \cite{Pank17}). More precisely, in \cite{BaMaNa17} the authors deal with geometries for which the near field domain $\Om$ is a cruciform junction of two cylinders whose cross section coincides with the unit disc. In \cite{ChNa23}, $\Om$ is the union of three bars which have a square cross section. In these works, it is shown that the corresponding $\Om$ are such that $\mX_\dagger=\{0\}$. However by varying some parameters like the angle between the cylinders or the width of certain bars, or adding a resonator, one could find situations where $\dim\,\mX_\dagger=1$. Then perturbing $\Om$ around these particular configurations, one would observe the same phenomenon of breathing of the spectrum as described in the present article. Indeed what we describe below in Section \ref{SectionModel} is generic and not specific to dimension two.\\
\newline
Finally, let us indicate that these phenomena do not appear when considering the spectrum of the Laplace operator with Neumann boundary conditions (BCs). The reason is that for this problem, the bottom of the continuous spectrum of the near field operator $A^\Om$ is zero and we always have $\dim\,\mX_\dagger=1$ because the constants solve the homogeneous Laplace equation with Neumann BCs. As a consequence, independently of the geometry of $\Om$, the good limit model is a system of ordinary differential equations on the 1D ligaments obtained when taking $\eps\to0^+$ supplemented by Kirchoff transmission conditions at the nodes of the graph (continuity of the field and zero outgoing flux, i.e. the sum of the derivatives of the field along the outgoing directions at the node vanishes). This is the so-called L. Pauling's model \cite{Paul36} which is used in \cite{KuPo07} and justified in \cite{KuZe03,Grie08,Post12}.\\
\newline
The outline is as follows. In Section \ref{SectionSetting} we start by defining the geometry and introduce the notation. Section \ref{SectionNFPb} is dedicated to the presentation of the properties of the near field operator $A^\Om$ which plays a key role in the analysis. There we also state the first important result of this work, namely Theorem \ref{MainThmPerio}, which describes the features of the spectrum of $A^\eps$ as $\eps$ tends to zero. In Sections \ref{SectionFirstBands}, \ref{SectionAsymptoHigher} we prove the different items of Theorem \ref{MainThmPerio}. Then in Section \ref{SectionModel} we establish  Theorem \ref{MainThmPerioMigration}, the second important result of this article, where we show the breathing phenomenon of the spectral bands of $A^\eps$ when perturbing the inner field geometry around a particular $\Om_\star$ where $\dim\,\mX_\dagger=1$. In Section \ref{SectionNFgeom}, we explain how to construct examples of $\Om_\star$ where $\dim\,\mX_\dagger=1$. In Section \ref{SectionGapsUsualTR}, we consider the particular case where $\Om_\star$ is such that $\dim\,\mX_\dagger=1$ and classical Kirchhoff transmissions conditions must be imposed in the first model describing the bands above $\pi^2/\eps^2$. Due to this latter property, we  find that the spectrum of this model coincides with $[\pi^2/\eps^2;+\infty)$. By taking into account terms of higher order in $\eps$, we show however that in general the spectrum of the exact operator $A^\eps$ has some short gaps. Finally in Section \ref{SectionNumerics}, we give numerical illustrations of the different results.  

\section{Setting}\label{SectionSetting}

Let $\Om\subset\R^2$ be a waveguide which coincides with the strip $\R\times(-1/2;1/2)$ outside of a bounded region (Figure \ref{Geometry} top left). For $\eps>0$, we consider the unit cell
\begin{equation}\label{defCell}
\om^\eps:=\{z:=(x,y)\in\R^2\,|\,z/\eps\in\Om\mbox{ and }|x| < 1/2\}
\end{equation}
and set $\partial\om^\eps_{\pm}:=\{\pm1/2\}\times(-1/2;1/2)$ (Figure \ref{Geometry} top right). Finally we define the periodic waveguide
\begin{equation}\label{defPerio}
\Pi^\eps:=\{z\in \R^2\,|\,(x-m,y)\in \om^\eps\cup\partial\om^\eps_{+}\mbox{ for some }m\in\Z:=\{0,\pm1,\pm2,\dots\}\}
\end{equation}
(Figure \ref{Geometry} bottom). We assume that $\Om$, $\om^\eps$ and $\Pi^\eps$ are connected with Lipschitz boundaries. In $\Pi^\eps$, we consider the spectral problem for the Dirichlet Laplacian\begin{equation}\label{PbSpectral}
\begin{array}{|rcll}
-\Delta u^\eps&=&\lambda^\eps\,u^\eps &\mbox{ in }\Pi^\eps\\[3pt]
u^\eps&=&0&\mbox{ on } \partial\Pi^\eps.
\end{array}
\end{equation}
Note that since we work with Dirichlet boundary conditions, the assumption of Lipschitz regularity of the boundary could be relaxed. The variational form associated with this problem writes
\begin{equation}\label{variationPb}
\int_{\Pi^\eps}\nabla u^\eps\cdot\nabla v^\eps\,dz=\lambda^\eps\int_{\Pi^\eps} u^\eps v^\eps\,dz,\qquad\forall v^\eps\in\mH^1_0(\Pi^\eps).
\end{equation}
Here $\mH^1_0(\Pi^\eps)$ stands for the Sobolev space of functions of $\mH^1(\Pi^\eps)$ which vanish on the boundary $\partial\Pi^\eps$. Classically (see e.g \cite[\S10.1]{BiSo87}), the variational problem (\ref{variationPb}) gives rise to an unbounded, positive definite, selfadjoint operator $A^\eps$ in the Hilbert space $\mL^2(\Pi^\eps)$, with domain $\mathcal{D}(A^\eps)\subset\mH^1_0(\Pi^\eps)$. Note that this operator is sometimes called the quantum graph Laplacian \cite{BeKu13,Post12}. Since $\Pi^\eps$ is unbounded, the embedding $\mH^1_0(\Pi^\eps)\subset \mL^2(\Pi^\eps)$ is not compact and $A^\eps$ has a non empty continuous component $\sigma_c(A^\eps)$ (\cite[Thm. 10.1.5]{BiSo87}). Actually, due to the periodicity, we have $\sigma_c(A^\eps)=\sigma(A^\eps)$. The Gelfand transform (see the surveys \cite{Kuch82,Naza99a} and books \cite{Kuch93,NaPl94}) 
\[
u^{\eps}(z)\mapsto U^\eps(z,\eta)=\cfrac{1}{(2\pi)^{1/2}}\sum_{j\in\Z} e^{i\eta j}u^{\eps}(x+j,y),\quad\eta\in\R,
\]
converts Problem (\ref{PbSpectral}) into a spectral problem set in the periodicity cell $\om^\eps$ (see Figure \ref{Geometry} top right). This problem, which involves quasi-periodicity boundary conditions at $\partial\om^\eps_\pm$, writes
\begin{equation}\label{PbSpectralCell}
\begin{array}{|rcll}
-\Delta U^\eps(z,\eta)&=&\Lambda^\eps\,U^\eps(z,\eta) & z\in\om^\eps\\[3pt]
U^\eps(z,\eta)&=&0&z\in\partial\om^\eps\cap\partial\Pi^\eps\\[3pt]
U^\eps(-1/2,y,\eta) & = & e^{i\eta}U^\eps(+1/2,y,\eta) & y\in(-\eps/2;\eps/2)\\[3pt]
\partial_xU^\eps(-1/2,y,\eta) & = & e^{i\eta}\partial_xU^\eps(+1/2,y,\eta) & y\in(-\eps/2;\eps/2).
\end{array}
\end{equation}
Problem (\ref{PbSpectralCell}) is formally selfadjoint for any value of the Floquet parameter $\eta\in\R$. Additionally the transformation $\eta\mapsto\eta+2\pi$ leaves invariant the quasiperiodicity conditions. Therefore it is sufficient to study (\ref{PbSpectralCell}) for $\eta\in[0;2\pi)$. For any $\eta\in[0;2\pi)$, the spectrum of (\ref{PbSpectralCell}) is discrete, made of the unbounded sequence of normal eigenvalues
\[
0<\Lambda_1^\eps(\eta)\le \Lambda_2^\eps(\eta)\le \dots \le \Lambda_p^\eps(\eta)\le \dots
\]
where the $\Lambda_p^\eps(\eta)$ are counted according to their multiplicity. The functions 
\[
\eta\mapsto \Lambda_p^\eps(\eta)
\]
are continuous (\cite[Chap. 9]{Kato95}) so that the sets 
\begin{equation}\label{SpectralBands}
\Upsilon^\eps_p=\{\Lambda_p^\eps(\eta),\,\eta\in[0;2\pi)\},
\end{equation}
the so-called spectral bands, are connected compact in $[0;+\infty)$. Finally, according to the theory (see again \cite{Kuch82,Naza99a,Kuch93,NaPl94}), the spectrum of the operator $A^{\eps}$ has the form
\[
\sigma(A^\eps)=\bigcup_{p\in\N^\ast}\Upsilon^\eps_p
\]
where $\N^\ast:=\{1,2,\dots\}$. Thus, we see that to study the behaviour of the spectrum of $A^\eps$ with respect to $\eps\to0^+$, we have to clarify the dependence of the $\Upsilon^\eps_p$ in $\eps$.

\section{Near field problem and first main result}\label{SectionNFPb}
As already mentioned in the introduction, the analysis developed for example in \cite{Grie08,Naza14} shows that the asymptotic behaviour of the $\Upsilon^\eps_p$ with respect to $\eps$ depends on the features of the Dirichlet Laplacian in $\Om$, the geometry obtained when zooming in the periodicity cell $\om^\eps$ (see Figure \ref{Geometry} top left). This leads us to study the spectral problem 
\begin{equation}\label{PbSpectralZoom}
\begin{array}{|rcll}
-\Delta u&=&\lambda\,u &\mbox{ in }\Om\\[3pt]
u&=&0&\mbox{ on } \partial\Om
\end{array}
\end{equation}
which is now independent of $\eps$. We denote by $A^\Om$ the unbounded, positive definite, selfadjoint operator naturally associated with this problem defined in the Hilbert space $\mL^2(\Om)$, with domain $\mathcal{D}(A^\Om)\subset\mH^1_0(\Om)$. Its continuous spectrum $\sigma_c(A^\Om)$ occupies the ray $[\pi^2;+\infty)$ and the threshold point $\lambda_\dagger:=\pi^2$ is the first eigenvalue of the Dirichlet Laplacian in the cross section of the unbounded branches of $\Om$. According to the geometry of $\Om$, the operator $A^\Om$ may have or not discrete spectrum. To set ideas, assume that $A^\Om$ has exactly $N_{\bullet}\in\N:=\{0,1,2,\dots\}$ eigenvalues (counted with multiplicity) in its discrete spectrum, that we denote  by
\begin{equation}\label{DefMu1}
0<\mu_1<\mu_2\le \mu_3 \le \dots \le \mu_{N_{\bullet}}<\pi^2.
\end{equation}
The fact that $\mu_1$ is simple is a classical consequence of the Krein-Rutman theorem (see e.g. \cite[Thm. 1.2.5]{Henr06}). We will see below that the properties of (\ref{PbSpectralZoom}) with $\lambda$ coinciding with the threshold of the continuous spectrum of $A^\Om$ plays an important role in the asymptotics of certain of the $\Upsilon^\eps_p$ as $\eps\to0^+$. This leads us to study the problem 
\begin{equation}\label{NearFieldPb}
\begin{array}{|rcll}
\Delta W+\pi^2W&=&0&\mbox{ in }\Om\\
W&=&0 &\mbox{ on }\partial\Om.
\end{array}
\end{equation}
Let us describe the features of Problem (\ref{NearFieldPb}). To proceed, first define the waves $w_0$, $w_1$ and the linear wave packets $w_{\pm}^{\mrm{out}}$, $w_{\pm}^{\mrm{in}}$ such that
\begin{equation}\label{defModeswinout}
w_0(z)=\varphi(y),\qquad w_1=|x|\varphi(y),\qquad 
\begin{array}{|lcl}
w^{\mrm{out}}(z) &=& w_1(z)- iw_0(z) =(|x|- i)\varphi(y)\\[5pt]
w^{\mrm{in}}(z) &=& w_1(z)+ iw_0(z) =(|x|+ i)\varphi(y)
\end{array}
\end{equation}
with
\begin{equation}\label{defUdagger}
\varphi(y)=\sqrt{2}\cos(\pi y).
\end{equation}
Here the notation ``in/out'' stands for ``incoming/outgoing''. Introduce some cut-off functions $\psi_\pm\in\mathscr{C}^{\infty}(\R)$ such that $\psi_\pm(x)=1$ for $\pm x>1$ and $\psi_\pm(x)=0$ for $\pm x\le 1/2$. The theory presented for example in \cite[Chap. 5]{NaPl94} guarantees that Problem (\ref{NearFieldPb}) admits solutions of the form
\begin{equation}\label{ScaSol}
\begin{array}{rcl}
v_+ & =& \psi_-s_{+-}w^{\mrm{out}}+\psi_+(w^{\mrm{in}}+s_{++}w^{\mrm{out}})+\tilde{v}_+ \\[5pt]
v_- & =& \psi_-(w^{\mrm{in}}+s_{--}w^{\mrm{out}})+\psi_+s_{-+}w^{\mrm{out}}+\tilde{v}_-
\end{array}
\end{equation}
where $s_{\pm\pm}$, $s_{\pm\mp}$ are complex numbers and the $\tilde{v}_\pm$ decay exponentially as $|x|\to+\infty$. The coefficients $s_{\pm\pm}$, $s_{\pm\mp}$ form the so-called threshold scattering matrix
\begin{equation}\label{DefTRScaMa}
\mathbb{S}:=\left(
\begin{array}{cc}
s_{++} & s_{+-} \\[2pt]
s_{-+} & s_{--}
\end{array}
\right).
\end{equation}
It is known that $\mathbb{S}$ is symmetric ($\mathbb{S}=\mathbb{S}^{\top}$) but not hermitian in general, and unitary ($\mathbb{S}\overline{\mathbb{S}}^{\top}=\mrm{Id}$). Here and below $\mrm{Id}$ is the identity matrix of $\Cplx^{2\times2}$. On the other hand, it may happen that (\ref{NearFieldPb}) also admits solutions which are exponentially decaying at infinity (trapped modes). In this situation, $v_\pm$ are not uniquely defined. However these trapped modes do not modify the threshold scattering matrix because of the decay and therefore $\mathbb{S}$ is always uniquely defined. Following \cite{MoVa08,Naza17bis,KoNS16,Naza20Threshold}, now we introduce some vocabulary. Define the vector spaces
\begin{equation}\label{DefSpaceX}
\begin{array}{lcl}
\mX&\hspace{-0.2cm}:=&\hspace{-0.2cm}\dsp\{U=\sum_{\pm}\psi_\pm (C^1_\pm w_1+C^0_\pm w_0)+\tilde{U}\mbox{ satisfying  }(\ref{NearFieldPb})\mbox{ with }C^1_\pm,\,C^0_\pm\in\Cplx\mbox{ and }\tilde{U}\in\mH^1_0(\Om)\};\\[14pt]
\mX_{\mrm{bo}}&\hspace{-0.2cm}:=&\hspace{-0.2cm}\dsp\{U=\sum_{\pm}\psi_\pm C^0_\pm w_0+\tilde{U}\mbox{ satisfying  }(\ref{NearFieldPb})\mbox{ with }C^0_\pm\in\Cplx\mbox{ and }\tilde{U}\in\mH^1_0(\Om)\};\\[14pt]
\mX_{\mrm{tr}}&\hspace{-0.2cm}:=&\hspace{-0.2cm}\dsp\{U\in\mH^1_0(\Om)\mbox{ satisfying  }(\ref{NearFieldPb})\}.
\end{array}\hspace{-0.3cm}
\end{equation}
Here $\mX_{\mrm{bo}}$ is the space of bounded solutions of (\ref{NearFieldPb}) while $\mX_{\mrm{tr}}$ stands for the space of trapped modes. 

\begin{definition}
We say that $A^{\Om}$ has a Threshold Resonance (TR) if $\mX_{\mrm{bo}}\ne\{0\}$, i.e. if (\ref{NearFieldPb}) admits a non zero bounded solution.
\end{definition}

\begin{definition}
We say that $A^{\Om}$ has a proper TR if $\mX_{\mrm{bo}}/\mX_{\mrm{tr}}\ne\{0\}$, i.e. if (\ref{NearFieldPb}) admits a bounded solution with does not decay at infinity. If $\mX_{\mrm{bo}}=\mX_{\mrm{tr}}\ne\{0\}$, we say that the TR is improper. 
\end{definition}

\begin{definition}\label{DefXdagger}
Set $\mX_{\dagger}:=\mX_{\mrm{bo}}/\mX_{\mrm{tr}}$. This quotient space is sometimes called the space of almost standing waves of (\ref{NearFieldPb}).
\end{definition}

\noindent In the following, the dimension of $\mX_{\dagger}$ will play an important role. The following proposition gives a characterization of this quantity. Its proof is similar to the one of \cite[Thm.\,7.1]{Naza16} (see also \cite{Naza14}).
\begin{proposition}\label{PropoCharacterisation}
We have $\dim\,\mX_{\dagger}=\dim(\ker\,(\mathbb{S}+\mrm{Id}))$.
\end{proposition}
\noindent Let us describe the different situations that one can meet.\\[4pt]
$\bullet$ The dimension of $\mX_{\mrm{tr}}$ is independent of the other quantities. Moreover, the existence of trapped modes for (\ref{NearFieldPb}) (hence at the threshold of the continuous spectrum of $A^{\Om}$) is a rare phenomenon which appears only for specific geometries which are unstable with respect to small perturbations.\\
\newline
$\bullet$ From the definition of $\mX_{\mrm{bo}}$, clearly we have $\dim\,\mX_{\dagger}\le2$. On the other hand, since $\mathbb{S}$ is a unitary matrix, its two eigenvalues are located on the unit circle. Proposition \ref{PropoCharacterisation} ensures that Problem (\ref{NearFieldPb}) has a proper TR when one of these eigenvalues is located at $-1$. For this reason, the existence of TR is not the common situation. In other words, in general there holds $\mX_{\mrm{bo}}=\{0\}$. However let us mention that the reference strip $\Om=\R\times(-1/2;1/2)$ offers a very simple example of geometry where $\mX_{\mrm{bo}}\ne\{0\}$. More precisely, in that case there holds $\mX_{\mrm{bo}}=\mrm{span}(w_0)$, $\mX_{\mrm{tr}}=\{0\}$, and so $\dim\,\mX_{\dagger}=1$ (see Remark \ref{RmkRefStrip} for more details). It is an open question to prove the existence of waveguides where $\dim\,\mX_{\dagger}=2$ but this does not seem impossible. For an example of geometry in 2D where this happens for the Neumann Laplacian at a positive threshold, we refer the reader to \cite{NazaDoubleTR}.\\
\newline
Once this notation has been introduced, we can give the first important result of this article.
\begin{theorem}\label{MainThmPerio}
For $p\in\N^\ast$, let $\Upsilon^\eps_p=[a^\eps_{p-};a^\eps_{p+}]$, with $a^\eps_{p-}\le a^\eps_{p+}$, be the spectral band in (\ref{SpectralBands}). 
There are some (real) constants $c_{p-}<c_{p+}$, $C_{p}>0$, $\delta_p>0$ and $\eps_p>0$ such that we have \\[6pt]
$\mbox{For }p=1,\dots,N_\bullet:$
\begin{equation}\label{result_thm_1}
\Big|a^\eps_{p\pm}-\Big(\eps^{-2}\mu_p+\eps^{-2}e^{-\sqrt{\pi^2-\mu_p}/\eps}c_{p\pm}\Big)\Big| \le C_p\,e^{-(1+\delta_p)\sqrt{\pi^2-\mu_p}/\eps},\qquad\forall\eps\in(0;\eps_p];\\[4pt]
\end{equation}
$\mbox{For }p=N_\bullet+1,\dots,N_{\bullet}+N_{\dagger}:$
\begin{equation}\label{result_thm_2}
\Big|a^\eps_{p\pm}-\Big(\eps^{-2}\pi^2+\eps^{-2}e^{-\pi\sqrt{3}/\eps}c_{p\pm}\Big)\Big| \le C_p\,e^{-(1+\delta_p)\pi\sqrt{3}/\eps},\qquad\forall\eps\in(0;\eps_p];\\[4pt]
\end{equation}
$\mbox{For }p=N_{\bullet}+N_{\dagger}+m,\,m\in\N^{\ast}:$
\begin{equation}\label{result_thm_3}
\begin{array}{ll}
\phantom{ii}i)\mbox{ if }\mX_{\dagger}=\{0\}, &\qquad \Big|a^\eps_{p\pm}-\Big(\eps^{-2}\pi^2+m^2\pi^2+\eps c_{p\pm}\Big)\Big|\le C_{p}\,\eps^{1+\delta_p},\qquad\forall\eps\in(0;\eps_p];\\[10pt]
\phantom{i}ii)\mbox{ if }\dim\,\mX_{\dagger}=1, &  \qquad\Big|a^\eps_{p\pm}-\Big(\eps^{-2}\pi^2+c_{p\pm}\Big)\Big|\le C_{p}\,\eps^{1+\delta_p},\qquad\forall\eps\in(0;\eps_p];\\[10pt]
iii)\mbox{ if }\dim\,\mX_{\dagger}=2, & \qquad\Big|a^\eps_{p\pm}-\Big(\eps^{-2}\pi^2+(m-1)^2\pi^2+\eps c_{p\pm}\Big)\Big|\le C_{p}\,\eps^{1+\delta_p},\quad\forall\eps\in(0;\eps_p].
\end{array}
\end{equation}
Here the $\mu_p$ are the ones introduced in (\ref{DefMu1}).
\end{theorem}
\noindent Let us comment the different results provided by this theorem. First, when $\eps$ tends to zero, the whole spectrum of $A^\eps$ goes to $+\infty$ as $\eps^{-2}$. This is due to the Dirichlet condition imposed on the boundary combined with the fact that the elements of the lattice have diameter in $O(\eps)$. Besides, the first $N_{\bullet}+N_{\dagger}$ spectral bands of $A^\eps$ become extremely short, in $O(e^{-c/\eps})$ for some $c>0$ which depends on the band. Concerning the next spectral bands $\Upsilon^\eps_p$, $p=N_{\bullet}+N_{\dagger}+m$ with $m\in\N^\ast$, the behaviour depends on the dimension of $\mX_{\dagger}$. When the latter is zero or two (cases $i)$ and  $iii)$), the spectral bands are of length $O(\eps)$. Moreover, between $\Upsilon^\eps_p$ and $\Upsilon^\eps_{p+1}$, there is a gap, that is, a segment of spectral parameters $\lambda^\eps$ such that waves cannot propagate, whose length tends to $(2m+1)\pi^2$ (resp. $(2m-1)\pi^2$) in case $i)$ (resp. $iii)$). In other words, for these two cases, the propagation of waves in the thin lattice $\Pi^\eps$ is hampered and occurs only for very narrow (closed) intervals of frequencies. When the dimension of $\mX_{\dagger}$ is one (case $ii)$), the situation is very different. Indeed, asymptotically the spectral band $\Upsilon^\eps_p$ is of length $c_{p+}-c_{p-}$, with in general $c_{p+}>c_{p-}$. As a consequence, waves can propagate in $\Pi^\eps$ for much larger intervals of frequencies than in cases $i)$ and $iii)$.

\section{Asymptotic behaviour of the first $N_{\bullet}$ spectral bands}\label{SectionFirstBands}

In this section, we recall how to obtain a model leading to an expansion for the spectral band $\Upsilon^\eps_p$, $p\in\{1,\dots,N_{\bullet}\}$, appearing in (\ref{SpectralBands}). By definition, we have 
\begin{equation}\label{Def_FirstBande}
\Upsilon^\eps_p=\{\Lambda_p^\eps(\eta),\,\eta\in[0;2\pi)\}
\end{equation}
where $\Lambda_p^\eps(\eta)$ is the first eigenvalue of (\ref{PbSpectralCell}). Therefore our goal is to obtain an asymptotic expansion of $\Lambda_p^\eps(\eta)$ with respect to $\eps$ as $\eps\to0^+$.\\
\newline
Pick $\eta\in[0;2\pi)$. Let $u^\eps(\cdot,\eta)$ be an eigenfunction associated with $\Lambda_p^\eps(\eta)$. As a first approximation when $\eps\to0^+$, it is natural to consider the expansions
\begin{equation}\label{ExpansionTR0}
\Lambda_p^\eps(\eta)=\eps^{-2}\mu_p+\dots,\qquad u^\eps(z,\eta)=v(z/\eps)+\dots
\end{equation}
where $\mu_p\in(0;\pi^2)$ stands for an eigenvalue of the discrete spectrum of the operator $A^{\Om}$ introduced in (\ref{DefMu1}) and $v\in\mH^1_0(\Om)$ is a corresponding eigenfunction normalized in $\mL^2(\Om)$. Indeed, inserting the pair $(\eps^{-2}\mu_p,v(\cdot/\eps))$ in Problem (\ref{PbSpectralCell}) only leaves a small discrepancy on the faces $\partial\om^{\eps}_\pm$ defined after (\ref{defCell}) because $v$ is exponentially decaying at infinity. Let us write more precisely the decomposition of $v$ at infinity for further usage. Using Fourier decomposition, we get 
\begin{equation}\label{ExpanEigenVec}
v(z)=K_\pm\,e^{-\beta^p_1 |x|}\,\varphi(y)+O(e^{-\beta^p_2 |x|})\qquad\mbox{ as } \pm x\to+\infty.
\end{equation}
Here $K_\pm\in\R$, $\beta^p_1:=\sqrt{\pi^2-\mu_p}$, $\beta^p_2:=\sqrt{4\pi^2-\mu_p}$ and $\varphi$ appears in (\ref{defUdagger}). The first model (\ref{ExpansionTR0}) is simple but does not comprise the dependence in $\eta$. To improve it, consider the more refined ans\"atze
\begin{equation}\label{ExpansionTR1}
\Lambda_p^\eps(\eta)=\eps^{-2}\mu_p+\eps^{-2}e^{-\beta^p_1/\eps}M(\eta)+\dots,\qquad u^\eps(z,\eta)=v(z/\eps)+e^{-\beta^p_1/\eps}V(z/\eps,\eta)+\dots
\end{equation}
where the quantities $M(\eta)$ and $V(\cdot,\eta)$, which are complex valued, are to determine. Inserting (\ref{ExpansionTR1}) into (\ref{PbSpectralCell}), first we obtain that $V(\cdot,\eta)$ must satisfy
\begin{equation}\label{Pb_corrector_TR}
\begin{array}{|rccl}
-\Delta V(\cdot,\eta)-\mu_p V(\cdot,\eta)&=&M(\eta)v & \mbox{ in }\Om\\[3pt]
 V(\cdot,\eta) &=&0 &\mbox{ on }\partial\Om.
\end{array}
\end{equation}
Below, to simplify and because this is not the heart of the article, we shall assume that $\mu_p$ is a simple eigenvalue and the constants $K_\pm$ are non zero. When this is not the case, the analysis can be easily adapted (see Remark \ref{RemarkMultipleEig} below). In order to have a solution to (\ref{Pb_corrector_TR}), we must look for a $V(\cdot,\eta)$ which is growing at infinity. Due to (\ref{Pb_corrector_TR}), the simplest growth that we can allow is
\[
V(z,\eta)=B_\pm\,e^{\beta^p_1 |x|}\,\varphi(y)+O(e^{-\beta^p_1 |x|})\qquad\mbox{ as } \pm x\to+\infty
\]
where the $B_\pm$ are some constants. Then the quasi-periodic conditions satisfied by $u^\eps(\cdot,\eta)$ at $\partial\om^\eps_\pm$ (see (\ref{PbSpectralCell})) lead us to choose the $B_j$ such that
\[
\begin{array}{|l}
K_-+B_-=e^{i\eta}(K_++B_+)\\[2pt]
K_--B_-=e^{i\eta}(-K_++B_+)
\end{array}
\]
Solving this algebraic system, we obtain $B_+=K_-\,e^{-i\eta}$ and $B_-=K_+\,e^{+i\eta}$. Now since $\mu_p$ is a simple eigenvalue of $A^\Om$, multiplying (\ref{Pb_corrector_TR}) by $v$, integrating by parts in 
\[
\Om^R:=\{(x,y)\in\Om\,|\,|x|<R\}
\] 
and taking the limit $R\to+\infty$, we find that there is a solution if and only if the following compatibility condition 
\[
M(\eta)\|v\|^2_{\mL^2(\Om)}=-2\beta^p_1K_+K_-(e^{i\eta}+e^{-i\eta})\qquad\Leftrightarrow\qquad M(\eta)=-4\beta^p_1K_+K_-\cos\eta
\]
is satisfied. This defines the value of $M(\eta)$ in the expansion (\ref{ExpansionTR1}). From (\ref{Def_FirstBande}), we deduce that as $\eps$ tends to zero, the bounds of $\Upsilon^\eps_p=[a^\eps_{p-};a^\eps_{p+}]$ admit the asymptotics
\[
a^\eps_{p\pm}=\eps^{-2}\mu_{p}\pm\eps^{-2}e^{-\sqrt{\pi^2-\mu_p}/\eps}c_p+\dots
\]
with $c_p=4\beta^p_1K_+K_-$. This shows the result (\ref{result_thm_1}) of Theorem \ref{MainThmPerio}. Note that we decided to focus on a rather formal presentation above for the sake of conciseness. We emphasize that all these results can be completely justified by proving rigorous error. This has been realized in detail in \cite{Naza17,Naza18,Naza23} for similar problems and can be repeated here with obvious modifications.

\begin{remark}\label{RemarkMultipleEig}
When $\mu_p$ is a multiple eigenvalue of $A^\Om$, one can always construct a basis of the corresponding eigenspace for which at most one vector admits the expansion (\ref{ExpanEigenVec}) with $K_+K_-\ne0$. Let us mention that the spectral bands generated by the other vectors of the basis will be shorter than the one studied above when $\eps\to0^+$. 
\end{remark}

\section{Asymptotic behaviour of the higher spectral bands}\label{SectionAsymptoHigher}

We turn our attention to the asymptotics of the spectral bands 
\begin{equation}\label{Def_FirstBande_2}
\Upsilon^\eps_p=\{\Lambda_p^\eps(\eta),\,\eta\in[0;2\pi)\}
\end{equation}
for $p\ge N_{\bullet}$ when $\eps\to0^+$. First, if in $\Om$, there holds $\dim\,\mX_{\mrm{tr}}=N_{\dagger}\ge1$, then not only the first $N_{\bullet}$ spectral bands of $A^\eps$ become extremely tight as $\eps\to0^+$ but also the next $N_{\dagger}$ ones. More precisely, working as in Section \ref{SectionFirstBands} above, one can show that for $p\in\{N_{\bullet}+1,\dots,N_{\bullet}+N_{\dagger}\}$, the bounds of $\Upsilon^\eps_p=[a^\eps_{p-};a^\eps_{p+}]$ admit the asymptotics
\[
a^\eps_{p\pm}=\eps^{-2}\pi^2\pm\eps^{-2}e^{-\sqrt{3}\pi/\eps}c_p+\dots
\]
for some real constant $c_p$. We do not detail more that case, which gives the result (\ref{result_thm_2}) of Theorem \ref{MainThmPerio}, and now consider the asymptotics of the band $\Upsilon^\eps_p$ with $p\in\N$, $p>N_{\bullet}+N_{\dagger}$.\\
\newline
Pick $\eta\in[0;2\pi)$ and introduce $u^\eps(\cdot,\eta)$ an eigenfunction associated with $\Lambda_p^\eps(\eta)$. In the sequel, to simplify, we remove the subscript ${}_p$ and do not indicate the dependence on $\eta$. As a first approximation when $\eps\to0^+$, we consider the expansion
\begin{equation}\label{ExpansionTR0_2}
\Lambda^\eps=\eps^{-2}\pi^2+\nu+\dots,\qquad u^\eps(z)=v^\eps(z)+\dots
\end{equation}
with $v^\eps$ of the form
\begin{equation}\label{FField}
v^\eps(z)=\gamma^\pm(x)\,\varphi(y/\eps)\quad\mbox{ for }\pm x>0.
\end{equation}
Here the functions $\gamma^\pm$ are to determine and $\varphi$ is defined in (\ref{defUdagger}). Inserting (\ref{ExpansionTR0_2}) into Problem (\ref{PbSpectralCell}), we obtain 
\begin{equation}\label{farfield}
\begin{array}{|rcl}
\partial^2_x\gamma^++\nu\gamma^+&=&0\quad\mbox{ in }(0;1/2)\\[3pt]
\partial^2_x\gamma^-+\nu\gamma^-&=&0\quad\mbox{ in }(-1/2;0)\\[3pt]
\gamma^-(-1/2)&=&e^{i\eta}\gamma^+(+1/2)\\[3pt]
\partial_x\gamma^-(-1/2)&=&e^{i\eta}\partial_x\gamma^+(+1/2).
\end{array}
\end{equation}
To define the $\gamma^\pm$ uniquely, we need to complement (\ref{farfield}) with some conditions at the origin. To obtain them, we match the behaviour of the $\gamma^\pm$ with the one of some inner field expansion of $u^\eps$. More precisely, in a neighbourhood of the origin we look for an expansion of $u^\eps$ of the form 
\begin{equation}\label{NearField}
u^\eps(z)=W(z/\eps)+\dots
\end{equation}
with $W$ to determine. Inserting (\ref{NearField}) and (\ref{ExpansionTR0_2}) in (\ref{PbSpectralCell}), we find that $W$ must satisfy (\ref{NearFieldPb}), i.e. the near field problem at the threshold value of the continuous spectrum of $A^\Om$. Let us look for $W$ of the form 
\[
W=a_+v_++a_-v_-
\]
where $a_\pm\in\Cplx$ are constants to be set and $v_\pm$ are the functions introduced in (\ref{ScaSol}). For the $\gamma_\pm$ in (\ref{FField}), we have the Taylor series, as $x\to0$
\[
\gamma_\pm(x)=\gamma_\pm(0)+x\partial_x\gamma_\pm(0)+O(x^2).
\]
Matching the constant behaviour of the far field expansion (\ref{ExpansionTR0_2})--(\ref{FField}) with the one of the inner field expansion (\ref{NearField}) at order $\eps^0$, we find that we must impose
\begin{equation}\label{Matching1}
\begin{array}{|rcl}
\gamma^+(0)&\hspace{-0.2cm}=&\hspace{-0.2cm}a_+i(1-s_{++})-a_-is_{-+}\\[5pt]
\gamma^-(0)&\hspace{-0.2cm}=&\hspace{-0.2cm}-a_+is_{+-}+a_-i(1-s_{--})
\end{array}
\end{equation}
Introduce the notation $\boldsymbol{\gamma}(0):=(\gamma^+(0),\gamma^-(0))^\top$, $\boldsymbol{a}:=(a_+,a_-)^\top$. Then (\ref{Matching1}) rewrites as 
\begin{equation}\label{CompactForm1}
\boldsymbol{\gamma}(0)=i(\mrm{Id}-\mathbb{S})\boldsymbol{a}.
\end{equation}
Now let us match the $z$-behaviour of (\ref{ExpansionTR0_2})--(\ref{FField}) with the one of (\ref{NearField}) at order $\eps^0$. We get
\begin{equation}\label{Matching2}
\begin{array}{|rcl}
\eps\partial_x\gamma^+(0)&\hspace{-0.2cm}=&\hspace{-0.2cm}a_+(1+s_{++})+a_-s_{-+}\\[5pt]
-\eps\partial_x\gamma^-(0)&\hspace{-0.2cm}=&\hspace{-0.2cm}a_+s_{+-}+a_-(1+s_{--})
\end{array}
\end{equation}
which can be rewritten as 
\begin{equation}\label{CompactForm2}
\eps\partial_x\boldsymbol{\gamma}(0)=(\mrm{Id}+\mathbb{S})\boldsymbol{a}
\end{equation}
where $\partial_x\boldsymbol{\gamma}(0):=(\partial_x\gamma^+(0),-\partial_x\gamma^-(0))^\top$. At this stage, we have to divide the analysis according to the situation.

\subsection{Asymptotics of the spectral bands when $\mX_{\dagger}=\{0\}$}\label{ParaCasi}

In this section, we assume that $\Om$ is such that $\mX_{\dagger}=\{0\}$. Again we emphasize that this is the generic situation. Then according to Proposition \ref{PropoCharacterisation}, $\mrm{Id}+\mathbb{S}$ is invertible and relation (\ref{CompactForm2}) leads us to set $\boldsymbol{a}=0$ and so $W\equiv0$. Then, from (\ref{CompactForm1}), we find that we must impose $\boldsymbol{\gamma}(0)=0$, which gives the Dirichlet conditions 
\begin{equation}\label{farfield_BC}
\gamma^\pm(0)=0.
\end{equation}
to complement (\ref{farfield}). Solving the spectral problem (\ref{farfield}), (\ref{farfield_BC}), we obtain 
\begin{equation}\label{DefModel1D}
\nu=m^2\pi^2 \mbox{ for }m\in\N^\ast,\qquad  \begin{array}{|l}
\gamma^+(x)=\sin(m\pi x)\\[2pt]
\gamma^-(x)=-e^{i\eta}\sin(m\pi x).
\end{array}
\end{equation}
Note that $\nu$ is independent of $\eta\in[0;2\pi)$. This latter fact is not satisfactory because this does not allow us to assess the width of the spectral bands. In the sequel we wish to improve the model obtained above. Let us refine the expansion proposed in (\ref{ExpansionTR0_2}) and work with
\begin{equation}\label{ExpansionTR0_3}
\Lambda^\eps=\cfrac{\pi^2}{\eps^2}+m^2\pi^2+\eps\tilde{\nu}+\dots,\quad u^\eps(z)=(\gamma^\pm(x)+\eps\tilde{\gamma}^{\pm}(x))\,\varphi(y/\eps)+\dots\mbox{for }\pm x>0.
\end{equation}
Here $\tilde{\nu}$ as well as the $\tilde{\gamma}^{\pm}$ are to determine.  Inserting (\ref{ExpansionTR0_3}) into Problem (\ref{PbSpectralCell}) and extracting the terms of order $\eps$, we get 
\begin{equation}\label{farfield_3}
\begin{array}{|rcl}
\partial^2_x\tilde{\gamma}^++m^2\pi^2\tilde{\gamma}^+&=&-\tilde{\nu}\gamma^+\quad\mbox{ in }(0;1/2)\\[3pt]
\partial^2_x\tilde{\gamma}^-+m^2\pi^2\tilde{\gamma}^-&=&-\tilde{\nu}\gamma^-\quad\mbox{ in }(-1/2;0)\\[3pt]
\tilde{\gamma}^-(-1/2)&=&e^{i\eta}\tilde{\gamma}^+(+1/2)\\[3pt]
\partial_x\tilde{\gamma}^-(-1/2)&=&e^{i\eta}\partial_x\tilde{\gamma}^+(+1/2).
\end{array}
\end{equation}
To define properly the $\tilde{\gamma}^\pm$, we need to add to (\ref{farfield_3}) conditions at the origin. To identify them, again we match the far field expansion of $u^\eps$ with some inner field representation. In a neighbourhood of the origin, we look for an expansion of $u^\eps$ of the form 
\begin{equation}\label{NearField_3}
u^\eps(x)=\eps \tilde{W}(x/\eps)+\dots.
\end{equation}
Inserting (\ref{NearField_3}) and (\ref{ExpansionTR0_3}) in (\ref{PbSpectralCell}), we find that $\tilde{W}$ must satisfy (\ref{NearFieldPb}). We could look for $\tilde{W}$ as a combination of the linear waves $v_{\pm}$ appearing in (\ref{ScaSol}). However, it will be more convenient to work with the functions $\tilde{W}^\pm$ which satisfy (\ref{NearFieldPb}) and admit the expansions
\begin{equation}\label{DefPolMat1}
\tilde{W}^+(z)=
\begin{array}{|rl}
(x+M_{++})\,\varphi(y)+\dots&\mbox{ in }\Om\mbox{ as }x\to+\infty \\[2pt]
M_{+-}\,\varphi(y)+\dots&\mbox{ in }\Om\mbox{ as }x\to-\infty 
\end{array}\hspace{0.25cm}
\end{equation}
\begin{equation}\label{DefPolMat1Bis}
\tilde{W}^-(z)=
\begin{array}{|rl}
M_{-+}\,\varphi(y)+\dots&\mbox{ in }\Om\mbox{ as }x\to+\infty \\[2pt]
(|x|+M_{--})\,\varphi(y)+\dots&\mbox{ in }\Om\mbox{ as }x\to-\infty
\end{array}
\end{equation}
where $M_{\pm\pm}$, $M_{\pm\mp}$ are some constants. These $\tilde{W}^\pm$ can be constructed by combining the $v_{\pm}$. To proceed, one needs to use the fact that $\mrm{Id}+\mathbb{S}$ is invertible, which is a consequence of Proposition \ref{PropoCharacterisation} and of the assumption $\mX_{\dagger}=\{0\}$. In the calculus, we find, see e.g. relation (7.9) in \cite{Naza16}, that the polarization matrix
\begin{equation}\label{def_Pola_matrix}
\mathbb{M}:=\left(\begin{array}{cc}
M_{++} & M_{+-} \\[2pt]
M_{-+} & M_{--}
\end{array}\right),
\end{equation}
coincides with the Cayley transform of $\mathbb{S}$, i.e. we have
\begin{equation}\label{def_Pola_matrix}
\mathbb{M}=i(\mrm{Id}+\mathbb{S})^{-1}(\mrm{Id}-\mathbb{S}).
\end{equation}
This shows in particular that $\mathbb{M}$ is real and symmetric even in non symmetric geometries (see \cite[Chap. 5, Prop. 4.13]{NaPl94} and \cite{Naza17,Naza18,Naza23}).\\
\newline
For the functions $\gamma^{\pm}$ in (\ref{DefModel1D}), we have the Taylor expansion, as $x\to0$,
\begin{equation}\label{ExpanFarField}
\begin{array}{rcl}
\gamma^+(x)&=&0+m\pi x+\dots=\eps m\pi\,\cfrac{x}{\eps}+\dots,\\[5pt]
\gamma^-(x)&=&0-e^{i\eta}m\pi x+\dots=-\eps e^{i\eta}m\pi \,\cfrac{x}{\eps}+\dots.
\end{array}
\end{equation}
Comparing (\ref{ExpanFarField}) with (\ref{DefPolMat1}) leads us to choose $\tilde{W}$ in the expansion $u^\eps(x)=\eps \tilde{W}(x/\eps)+\dots$ (see (\ref{NearField_3})) such that
\[
\tilde{W}=m\pi \,(\tilde{W}^++e^{i\eta}\tilde{W}^-).
\]
This sets the constant behaviour of $\tilde{W}$ at infinity and we now match the later with the behaviour of the $\tilde{\gamma}^\pm$ at the origin to close system (\ref{farfield_3}). This step leads us to impose 
\begin{equation}\label{ConditionAtZero}
\tilde{\gamma}^+(0)=m\pi (M_{++}+e^{i\eta}M_{+-}); \qquad\quad \dsp \tilde{\gamma}^-(0)=m\pi (M_{-+}+e^{i\eta}M_{--}).
\end{equation}
Equations (\ref{farfield_3}), (\ref{ConditionAtZero}) form a boundary value problem for a system of ordinary differential equations. For this problem, there is a kernel and a co-cokernel. In order to have a solution, some compatibility conditions must be satisfied. Multiplying (\ref{farfield_3}) by $\overline{\gamma^\pm}$ and integrating by parts, we find that they are verified if 
\[
\overline{\partial_x\gamma^+(0)}\,\tilde{\gamma}^+(0)-\overline{\partial_x\gamma^-(0)}\,\tilde{\gamma}^-(0)=-\tilde{\nu}\left(\int_{-1/2}^{0}|\gamma^-|^2\,dx+\int_{0}^{1/2}|\gamma^+|^2\,dx\right)
\]
(note that we used that $\gamma^{\pm}(0)=0$ according to (\ref{farfield_BC})). Since $\overline{\partial_x\gamma^+(0)}=m\pi$ and $\overline{\partial_x\gamma^-(0)}=-m\pi e^{-i\eta}$, this gives
\[
2m^2\pi^2\Big((M_{++}+e^{i\eta}M_{+-})+e^{-i\eta}(M_{-+}+e^{i\eta}M_{--})\Big)=-\tilde{\nu}.
\]
By exploiting that the polarization matrix $\mathbb{M}$ defined in (\ref{def_Pola_matrix}) is real and symmetric, we get
\[
2m^2\pi^2(M_{++}+2\cos\eta\, M_{+-}+M_{--})=-\tilde{\nu}(\eta)
\]
(above we reintroduce the dependence with respect to $\eta$ for $\tilde{\nu}$). Finally, one can compute the solution to the system (\ref{farfield_3}), (\ref{ConditionAtZero}) to obtain the expressions of the $\tilde{\gamma}^\pm(\eta)$. This ends the definition of the terms appearing in the expansions (\ref{ExpansionTR0_3}). Let us comment and exploit these results.\\ 
\newline
$\star$ First, observe that the $\tilde{\nu}(\eta)$ are real because the coefficients of $\mathbb{M}$ are real.\\
\newline
$\star$ We have obtained 
\begin{equation}\label{ExpanConclu}
\Lambda^\eps(\eta)=\eps^{-2}\pi^2+m^2\pi^2+\eps\tilde{\nu}(\eta)+\dots\,.
\end{equation}
Note that in this expansion, the third term, contrary to the first two ones, depends on $\eta$ when $M_{+-}\ne0$. Since $\Upsilon^\eps_p=[a^\eps_{p-};a^\eps_{p+}]=\{\Lambda_{p}^\eps(\eta),\,\eta\in[0;2\pi)\}$, this analysis shows that we have the asymptotics
\[
a^\eps_{p\pm}=\eps^{-2}\pi^2+m^2\pi^2+\eps m^2c_{\pm}+\dots
\]
with 
\[
\begin{array}{rcl}
c_-&=&2\pi^2\min(-(M_{++}-2|M_{+-}|+M_{--}),-(M_{++}+2|M_{+-}|+M_{--}))\\[4pt]
c_+&=&2\pi^2\max(-(M_{++}-2|M_{+-}|+M_{--}),-(M_{++}+2|M_{+-}|+M_{--})).
\end{array}
\]
$\star$ Observe that we have $c_+\ne c_-$ if and only if the geometry of $\Om$ is such that $M_{+-}=M_{-+}\ne0$.\\
\newline
This proves the item $i)$ of (\ref{result_thm_3}). Finally, let us mention again that all the formal presentation above can be justified rigorously by a direct adaptation of the proofs of error estimates presented in \cite{Naza17,Naza18,Naza23}.

\subsection{Asymptotics of the spectral bands when $\dim\,\mX_{\dagger}=1$}\label{ParaDim1}

In this section, we assume that the space of almost standing waves is of dimension one (in particular $A^\Om$ has a proper TR). From Proposition \ref{PropoCharacterisation}, this is equivalent to  have $\dim(\ker\,(\mathbb{S}+\mrm{Id}))=1$. In accordance with the work of D. Grieser \cite{Grie08}, we will show that the conditions at zero to complement the system of ODEs (\ref{farfield}) are different from the Dirichlet ones found in the previous section.\\
\newline
Denote by $\tau_1=-1$, $\tau_2\ne-1$ the two eigenvalues of $\mathbb{S}$ and introduce $\boldsymbol{b}_1$, $\boldsymbol{b}_2$ two corresponding normalized eigenvector. Since $\mathbb{S}$ is unitary, there holds $(\boldsymbol{b}_1,\boldsymbol{b}_2)=0$ where $(\cdot,\cdot)$ denotes the standard inner product of $\mathbb{C}^2$.
\begin{lemma}\label{LemmaUnitary}
Since $\mathbb{S}$ is unitary and symmetric, we can choose real valued $\boldsymbol{b}_1$, $\boldsymbol{b}_2$. 
\end{lemma}
\begin{proof}
Let us work for example with $\boldsymbol{b}_1$. Exploiting that $\mathbb{S}$ is unitary and symmetric, we can write $\overline{\boldsymbol{b}_1}=\mathbb{S}\overline{\mathbb{S}}^{\top}\,\overline{\boldsymbol{b}_1}=\mathbb{S}\overline{\mathbb{S}}\,\overline{\boldsymbol{b}_1}$. Inserting the identity $\mathbb{S}\boldsymbol{b}_1=\tau_1\,\boldsymbol{b}_1$ in this relation, we obtain 
\[
\overline{\boldsymbol{b}_1}=\overline{\tau_1}\,\mathbb{S}\,\overline{\boldsymbol{b}_1}.
\] 
Using that $\tau_1=\overline{\tau_1}^{-1}$, because the eigenvalues of $\mathbb{S}$ are located on the unit circle in the complex plane, we infer that $\overline{\boldsymbol{b}_1}$ is also an eigenvector associated with $\tau_1$. This ensures that $\boldsymbol{b}_1$ can be chosen in $\R^2$.
\end{proof}
\begin{remark}
If $\mathbb{S}$ has an eigenvalue of multiplicity two, which is not the case considered in this section, one can also show that the corresponding eigenvectors can be chosen in $\R^2$. Note that since $\mathbb{S}$ is unitary, it is diagonalizable, and therefore cannot have a Jordan block of size two.
\end{remark}
\noindent Exploiting Lemma \ref{LemmaUnitary}, and changing $\boldsymbol{b}_1$ in $-\boldsymbol{b}_1$ if necessary, we can introduce $\theta\in[0;\pi)$ such that 
\begin{equation}\label{defTheta}
\boldsymbol{b}_1=(\cos\theta,\sin\theta)^{\top}.
\end{equation}
Then, due to the relation $(\boldsymbol{b}_1,\boldsymbol{b}_2)=0$, we deduce that we can take $\boldsymbol{b}_2=(-\sin\theta,\cos\theta)^{\top}$.\\
\newline
This said, let us exploit the identities (\ref{CompactForm1}), (\ref{CompactForm2}). Taking the inner product of (\ref{CompactForm2}) against $\boldsymbol{b}_1$, we find
\begin{equation}\label{OnExplicite}
\eps(\partial_x\boldsymbol{\gamma}(0),\boldsymbol{b}_1)=((\mrm{Id}+\mathbb{S})\boldsymbol{a},\boldsymbol{b}_1)=(\boldsymbol{a},(\mrm{Id}+\overline{\mathbb{S}}^{\top})\boldsymbol{b}_1)=(\boldsymbol{a},(\mrm{Id}+\overline{\mathbb{S}})\boldsymbol{b}_1)=0.
\end{equation}
Therefore we obtain
\begin{equation}\label{Trans1}
\cos\theta\,\partial_x\gamma^+(0)=\sin\theta\,\partial_x\gamma^-(0).
\end{equation}
This is a first transmission condition to complement (\ref{farfield}). On the other hand, (\ref{CompactForm2}) also implies that $(\mrm{Id}+\mathbb{S})\boldsymbol{a}=0\Leftrightarrow\mathbb{S}\boldsymbol{a}=-\boldsymbol{a}$ and so $\boldsymbol{a}\in\ker(\mrm{Id}+\mathbb{S})$, i.e. $\boldsymbol{a}=c\,\boldsymbol{b}_1$ for some $c\in\Cplx$. From (\ref{CompactForm1}), this gives $\boldsymbol{\gamma}(0)=2i\boldsymbol{a}=2ic\,\boldsymbol{b}_1$, which implies $(\boldsymbol{\gamma}(0),\boldsymbol{b}_2)=0$. This provides the second transmission condition at the origin
\begin{equation}\label{Trans2}
\sin\theta\,\gamma^+(0)=\cos\theta\,\gamma^-(0).
\end{equation}
Let us summarize the results that we have obtained. Set $I:=(-1/2;1/2)$, $I_+:=(0;1/2)$ and $I_-:=(-1/2;0)$. Gathering (\ref{farfield}), (\ref{Trans1}), (\ref{Trans2}), finally we have found that $\nu$, $\gamma^\pm$ satisfy the problem with weighted Kirchhoff  transmissions conditions at the origin
\begin{equation}\label{farfieldBis}
\begin{array}{|rclrcl}
\partial^2_x\gamma^\pm+\nu\gamma^\pm&=&0\quad\mbox{ in }I_\pm\\[3pt]
\gamma^-(-1/2)&=&e^{i\eta}\gamma^+(+1/2) & \sin\theta\,\gamma^+(0)&=&\cos\theta\,\gamma^-(0)\\[3pt]
\partial_x\gamma^-(-1/2)&=&e^{i\eta}\partial_x\gamma^+(+1/2) & \qquad\cos\theta\,\partial_x\gamma^+(0)&=&\sin\theta\,\partial_x\gamma^-(0).
\end{array}
\end{equation}
If $\psi$ is a function defined on $I$, we note by $\psi_\pm$ its restriction to $I_\pm$. Let us define the Hilbert space 
\begin{equation}\label{defHspace}
\mathcal{H}^1_\eta(I):=\{\psi\,|\,\psi_\pm\in\mH^1(I_\pm),\,\sin\theta\,\psi^+(0)=\cos\theta\,\psi^-(0)\mbox{ and }\psi^-(-1/2)=e^{i\eta}\psi^+(1/2)\}.
\end{equation}
Classically, one shows that $(\nu,\gamma)\in\Cplx\times\mathcal{H}^1_\eta(I)\setminus\{0\}$ is an eigenpair of (\ref{farfieldBis}) if and only if there holds
\begin{equation}\label{EigenPb0}
\int_{I_-\cup I_+}\partial_x \gamma \partial_x \overline{\psi}\,dx =\nu\,\int_{I_-\cup I_+} \gamma \overline{\psi}\,dx,\qquad \psi\in\mathcal{H}^1_\eta(I).
\end{equation}
Therefore, for all $\eta\in[0;2\pi)$, the spectrum of (\ref{farfield}), (\ref{Trans1}), (\ref{Trans2}) is real, made of non negative eigenvalues
\[
0 \le \nu_1(\eta) \le \nu_2(\eta) \le  \dots .
\]
Let us compute them explicitly. First, we consider the case where $\nu$ is a non zero eigenvalue. Exploiting (\ref{farfield}), we find 
\[
\begin{array}{rcl}
\gamma^+(x)&=&\phantom{e^{i\eta}\big(}A\,e^{i\sqrt{\nu}(x-1/2)}+B\,e^{-i\sqrt{\nu}(x-1/2)}\\[3pt]
\gamma^-(x)&=&e^{i\eta}\big(A\,e^{i\sqrt{\nu}(x+1/2)}+B\,e^{-i\sqrt{\nu}(x+1/2)}\big).
\end{array}
\]
Writing the transmission conditions (\ref{Trans1}), (\ref{Trans2}) at the origin, we obtain that $\nu\ne0$ is an eigenvalue if and only if the matrix
\[
\left(\begin{array}{cc}
\sin\theta\,e^{-i\sqrt{\nu}/2}-\cos\theta\,e^{i\sqrt{\nu}/2}e^{i\eta} & \sin\theta\,e^{i\sqrt{\nu}/2}-\cos\theta\,e^{-i\sqrt{\nu}/2}e^{i\eta} \\[4pt]
i\sqrt{\nu}(\cos\theta\,e^{-i\sqrt{\nu}/2}-\sin\theta\,e^{i\sqrt{\nu}/2}e^{i\eta}) & i\sqrt{\nu}(-\cos\theta\,e^{i\sqrt{\nu}/2}+\sin\theta\,e^{-i\sqrt{\nu}/2}e^{i\eta})
\end{array}
\right)
\]
has a non zero kernel. Computing its determinant, we obtain that this is equivalent to have
\begin{equation}\label{Dispersion1}
\sin(2\theta)\cos\eta=\cos\sqrt{\nu}.
\end{equation}
A similar calculus shows that $\nu=0$ is an eigenvalue if and only if there holds
\begin{equation}\label{Dispersion2}
\sin(2\theta)\cos\eta=1.
\end{equation}
For $p=N_{\bullet}+N_{\dagger}+m,\,m\in\N^{\ast}$, we have obtained 
\begin{equation}\label{ExpanConclu}
\Lambda^\eps_p(\eta)=\eps^{-2}\pi^2+\nu_m(\eta)+\dots
\end{equation}
where $\nu_m(\eta)$ satisfies (\ref{Dispersion1}) or (\ref{Dispersion2}). Since $\Upsilon^\eps_p=[a^\eps_{p-};a^\eps_{p+}]=\{\Lambda_{p}^\eps(\eta),\,\eta\in[0;2\pi)\}$, this analysis shows that we have the asymptotics
\[
a^\eps_{p\pm}=\eps^{-2}\pi^2+c_{p\pm}+\dots
\]
with 
\[
c_ {p-}=\inf_{\eta\in[0;2\pi)}\nu_m(\eta),\qquad\qquad c_ {p+}=\sup_{\eta\in[0;2\pi)}\nu_m(\eta).
\]
This establishes the item $ii)$ of (\ref{result_thm_3}).\\
\newline
Let us describe the different possibilities that one can meet according to the value of $\theta$ (which depends only on the geometry of $\Om$). Define the sets
\begin{equation}\label{SpectralBandsBis}
\Xi_m(\theta):=\{\nu_m(\eta),\,\eta\in[0;2\pi)\},\qquad\qquad\aleph(\theta):=\bigcup_{m=1}^{+\infty}\Xi_m(\theta).
\end{equation}
From (\ref{Dispersion1}), (\ref{Dispersion2}), we see that there holds $\Xi_m(\theta)=\Xi_m(\pi-\theta)$. Therefore it is sufficient to study $\Xi_m(\theta)$ for $\theta\in[0;\pi/2]$. Moreover, we observe that we also have $\Xi_m(\theta)=\Xi_m(\pi/2-\theta)$. As a consequence, it is enough to compute $\Xi_m(\theta)$ for $\theta\in[0;\pi/4]$. Solving (\ref{Dispersion1}),  (\ref{Dispersion2}), we find that for $\theta=0$, the dispersion curves are flat so that $\aleph(\theta)$ is simply an unbounded sequence of points (see Figure \ref{ImageDispersionCurves} left). This is a degenerate situation. For $\theta\in(0;\pi/4)$, $\aleph(\theta)$ coincides with a union of bands separated by some gaps (Figure \ref{ImageDispersionCurves} center). Finally, for $\theta=\pi/4$, the conditions (\ref{Trans1}), (\ref{Trans2}) become classical Kirchhoff  transmission conditions (continuity of the value and continuity of the flux), so that $\aleph(\theta)=[0;+\infty)$ (Figure \ref{ImageDispersionCurves} right). This is a second degenerate situation. Let us emphasize that though there holds $\aleph(\theta)=[0;+\infty)$ in this latter case, one may still have some gaps in the spectrum of $A^\eps$ above $\pi^2/\eps^2$. To study them, and possibly show their existence, it is necessary to work with higher order models. This will be done in Section \ref{SectionGapsUsualTR}.

\begin{figure}[!ht]
\centering
\includegraphics[width=16cm,trim={3cm 21cm 3cm 2.4cm},clip]{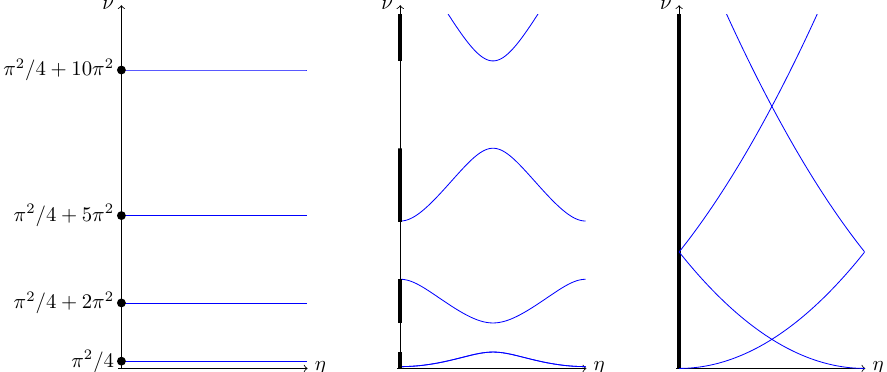}
\caption{Dispersion curves (\ref{Dispersion1}) for $\theta=0$ (left), $\theta=\pi/8$ (center), $\theta=\pi/4$ (right).\label{ImageDispersionCurves}}
\end{figure}

\begin{remark}\label{RmkRefStrip}
Let us discuss here the simple case where $\Om$ is the reference strip $\R\times(-1/2;1/2)$. In that situation, one observes that the functions $v_\pm$ introduced in (\ref{ScaSol}) are given by 
\[
v_+(z)=(x+i)\varphi(y),\qquad\qquad v_-(z)=(-x+i)\varphi(y).
\]
As a consequence, we have $s_{++}=s_{--}=0$, $s_{\pm\mp}=-1$ and so 
\[
\mathbb{S}:=\left(
\begin{array}{cc}
0 & -1 \\[2pt]
-1 & 0
\end{array}
\right).
\]
The eigenvalues of $\mathbb{S}$ are $\pm1$ and we can take $\boldsymbol{b}_1=(1/\sqrt{2},1/\sqrt{2})^{\top}$, $\boldsymbol{b}_2=(-1/\sqrt{2},1/\sqrt{2})^{\top}$. We deduce that the transmission conditions (\ref{Trans1}) and (\ref{Trans2}) write
\[
\begin{array}{|rcl}
\gamma^+(0)&=&\gamma^-(0)\\[3pt]
\partial_x\gamma^+(0)&=&\partial_x\gamma^-(0).
\end{array}
\]
This was expected because in this simple geometry, the origin plays no particular role for the spectral problem (\ref{PbSpectralCell}). We find that the angle $\theta$ in (\ref{defTheta}) is equal to $\pi/4$. Therefore, the quantity $\aleph(\theta)$ in (\ref{SpectralBandsBis}) is equal to $[0;+\infty)$, which was also expected. Indeed, in this situation, the periodic domain $\Pi^\eps$ coincides with the thin strip $\R\times(-\eps/2;\eps/2)$ and we know that the spectrum of the Dirichlet Laplacian in this geometry is $[\pi^2/\eps^2;+\infty)$ (observe that there is no gap).
\end{remark}

\subsection{Asymptotics of the spectral bands when $\dim\,\mX_{\dagger}=2$}

Finally, we consider the hypothetical case where the geometry $\Om$ is such that $\dim\,\mX_{\dagger}=2$ (again, this is an open question to show the existence of such a geometry). According to Proposition \ref{PropoCharacterisation}, this is equivalent to assume that $\dim(\ker(\mrm{Id}+\mathbb{S}))=2$. In that situation, in accordance with \cite{Grie08}, we find that the system of ODEs (\ref{farfield}) must be complemented with homogeneous Neumann boundary conditions at the origin. Indeed, working as in (\ref{OnExplicite}), we obtain that we must impose $(\partial_x\boldsymbol{\gamma}(0),\boldsymbol{d})=0$ for all $\boldsymbol{d}\in\Cplx^2$, i.e. 
\begin{equation}\label{NeumannOrigin}
\partial_x\gamma^+(0)=\partial_x\gamma^-(0)=0.
\end{equation}
Solving (\ref{farfield}), (\ref{NeumannOrigin}), we obtain 
\begin{equation}\label{DefModel1D}
\nu=m^2\pi^2 \mbox{ for }m\in\N,\qquad\quad  \begin{array}{|l}
\gamma^+(x)=\cos(m\pi x)\\[2pt]
\gamma^-(x)=-e^{i\eta}\cos(m\pi x).
\end{array}
\end{equation}
Note that $\nu$ is independent of $\eta\in[0;2\pi)$. As in \S\ref{ParaCasi}, this first model is not enough for our needs because it does not provide estimates for the lengths of the spectral bands. But working as in \S\ref{ParaCasi}, we can improve it to establish the item $iii)$ of (\ref{result_thm_3}).

\section{A model problem for the breathing of spectral bands}\label{SectionModel} 
  
In this section, our goal is to propose a model which describes the change of the spectrum of $A^\eps$ when perturbing the near field geometry around a particular $\Om=\Om_\star$ where $\dim\,\mX_\dagger=1$. Let us emphasize that this leads to make a periodic perturbation of $\Pi^\eps$.

\subsection{Setting and main result}

Let us start from a geometry $\Om_\star$ such that $\dim\,\mX_\dagger=1$ (we will explain in Section \ref{SectionNFgeom} how to construct particular examples). To simplify the presentation, we assume that we also have $\mX_{\mrm{bo}}=\{0\}$. As in (\ref{DefMu1}), we denote by $N_{\bullet}$ the number of eigenvalues of $A^{\Om_\star}$ below the continuous spectrum. Now we perturb this waveguide. To proceed, consider $\Gamma$ a smooth bounded and connected part of $\partial\Om_\star$. Denote by $s$ the arc length on $\Gamma$ so that $\Gamma=\{P(s)\in\R^2\,|\,s\in J\}$ where $J$ is a given interval of $\R$. Let $h\in\mathscr{C}^{\infty}_0(J)$ be a smooth profile function such that $h\ge0$ and $h\not\equiv0$. For some given parameter $\rho\in\R$ that we can choose as we wish, we define the geometry $\Om^{\rho,\eps}$ such that outside of $\mathscr{V}$, $\partial\Om^{\rho,\eps}=\partial\Om$ and inside $\mathscr{V}$,  $\partial\Om^{\rho,\eps}$ coincides with 
\begin{equation}\label{DefGeom}
\Gamma^{\rho,\eps}:=\{P(s)+\eps \rho h(s)n(s)\,|\,s\in J\}.
\end{equation}
Here $n(s)$ is the outward unit normal vector to $\Gamma$ at point $P(s)$. Note that here the inner field geometry also depends on $\eps$. By (\ref{defCell}), this defines a family of unit cells and periodic waveguides that we denote respectively $\om^{\rho,\eps}$ and $\Pi^{\rho,\eps}$ to stress the dependence in $\rho$. Observe that with this notation, we have in particular $\om^{0,\eps}=\om^{\eps}$, $\Pi^{0,\eps}=\Pi^{\eps}$.
Our goal is to compute, for a fixed $\rho\in\R$, the asymptotics of the spectral bands $\Upsilon^{\rho,\eps}_p:=\{\Lambda_p^{\rho,\eps}(\eta),\,\eta\in[0;2\pi)\}$ of the Dirichlet Laplacian in $\Pi^{\rho,\eps}$. The study of the first $N_\bullet$ spectral bands has no particular interest compared to what has been done in Section \ref{SectionFirstBands} and so we focus our attention on the case $p>N_\bullet$. The next theorem constitutes the second important result of this article.
\begin{theorem}\label{MainThmPerioMigration}
Fix $\rho\in\R$. For $m\in\N^{\ast}$ and $p=N_{\bullet}+m$, let $\Upsilon^{\rho,\eps}_p=[a^{\rho,\eps}_{m-};a^{\rho,\eps}_{m+}]$, with $a^{\rho,\eps}_{m-}\le a^{\rho,\eps}_{m+}$, be the spectral band as defined above.  
There are some (real) constants $c^\rho_{m-}<c^\rho_{m+}$, $C_{m}>0$, $\delta_m>0$ and $\eps_m>0$ such that we have \\[6pt]
\begin{equation}\label{result_thm_migration}
\Big|a^{\rho,\eps}_{m\pm}-\Big(\eps^{-2}\pi^2+c^\rho_{m\pm}\Big)\Big|\le C_{m}\,\eps^{1+\delta_m},\qquad\forall\eps\in(0;\eps_m].
\end{equation}
Here the constants $C_m$, $\eps_m$ can be chosen independently of $\rho$ if we impose to $\rho$ to belong to a compact set of $\R$. Moreover, we have
\begin{equation}\label{result_thm_migration_2}
\lim_{\rho\to-\infty} c^\rho_{m\pm}=m^2\pi^2,\qquad\qquad 
 c^\rho_{1\pm}\underset{\rho\to+\infty}{\sim}-\cfrac{T^2}{4}\,\rho^2,\qquad\qquad 
\lim_{\rho\to+\infty} c^\rho_{(m+1)\pm}=m^2\pi^2
\end{equation}
with 
\[
c^\rho_{m+}-c^\rho_{m-}\underset{\rho\to-\infty}{=}O(1/\rho),\qquad c^\rho_{1+}-c^\rho_{1-}\underset{\rho\to+\infty}{=}O(e^{-\delta\rho}),\qquad c^\rho_{(m+1)+}-c^\rho_{(m+1)-}\underset{\rho\to+\infty}{=}O(1/\rho).
\]
Here $T>0$ which depends on $h$ is defined in (\ref{def_T}) and $\delta>0$ is a constant.
\end{theorem}
\noindent Let us comment this statement. Note first that all the asymptotic intervals $[c^\rho_{m-};c^\rho_{m+}]$ reduce to a singleton when $\pm\rho\to+\infty$. This describes the transition between cases $i)$ and $ii)$ in Theorem \ref{MainThmPerio}. Some illustrations of the dependence of the $[c^\rho_{m-};c^\rho_{m+}]$, that we denote by $\Xi^\rho_m(\theta)$ in the sequel, with respect to $\rho$ are given in Figures \ref{ImageAsymptoRho}, \ref{ImageAsymptoRho1} below. We can see that when $\rho$ runs from $-\infty$ to $+\infty$, which boils down to inflating the near field geometry around $\Om_\star$, the $[c^\rho_{m-};c^\rho_{m+}]$ expand and then shrink. This is what we call the breathing phenomenon of the spectrum of $A^\eps$.
In the process, in $\sigma(A^\eps)$ a band dives below $\pi^2/\eps^2$ and then stops breathing. It is approximated when $\eps\to0^+$ by the interval $[\eps^{-2}\pi^2+c^\rho_{1-};\eps^{-2}\pi^2+c^\rho_{1+}]$ which becomes extremely short when $\rho\to+\infty$. We emphasize that in order these results to be valid for $\pm\rho$ large, $\eps$ must be chosen small enough so that $\Om^{\rho,\eps}$ remains a small perturbation of $\Om_\star$. Typically, for a given $\eps>0$, we can take $\rho\in(-\eps^{-1/2};\eps^{-1/2})$.

\subsection{Proof of Theorem \ref{MainThmPerioMigration}}
\noindent Fix $\rho\in\R$ and pick $p=N_{\bullet}+m$ with $m\in\N^{\ast}$. We wish to compute an asymptotic expansion of $\Lambda_p^{\rho,\eps}(\eta)$ as $\eps\to0^+$. To simplify the presentation of the computations below, we remove the indices $\rho$, $\eta$ and $p$. Reproducing what has been done in Section \ref{SectionAsymptoHigher}, as an approximation when $\eps\to0^+$, we consider the expansions
\begin{equation}\label{ExpansionTR0_2Bis}
\Lambda^\eps=\eps^{-2}\pi^2+\nu+\dots,\qquad u^\eps(z)=v^\eps(z)+\dots
\end{equation}
with $v^\eps$ of the form
\begin{equation}\label{FFieldBis}
v^\eps(z)=\gamma^\pm(x)\,\varphi(y/\eps)\quad\mbox{ for }\pm x>0.
\end{equation}
We find that $\gamma^\pm$ must satisfy the equations (\ref{farfield}) that it remains to complement with conditions at the origin. As in (\ref{NearField}), we look for an expansion of $u^\eps$ in a neighbourhood of the origin of the form 
\[
u^\eps(z)=W(z/\eps)+\dots\qquad\mbox{ with }\qquad W=a_+v_++a_-v_-.
\]
Here $a_\pm\in\Cplx$ are constants to determine and $v_\pm$ are the functions introduced in (\ref{ScaSol}). Then the conditions (\ref{CompactForm1}), (\ref{CompactForm2}) write
\begin{equation}\label{CompactForm1Bis}
\hspace{0.7cm}\boldsymbol{\gamma}(0)=i(\mrm{Id}-\mathbb{S}^{\rho,\eps})\boldsymbol{a}
\end{equation}
\begin{equation}\label{CompactForm2Bis}
\eps\partial_x\boldsymbol{\gamma}(0)=(\mrm{Id}+\mathbb{S}^{\rho,\eps})\boldsymbol{a}
\end{equation}
where $\boldsymbol{a}=(a_+,a_-)^\top$. Above, to facilitate the understanding, we indicate the dependence of the scattering matrix with respect to the geometry. Denote by $\tau_1^\eps$, $\tau_2^\eps$ the two eigenvalues of $\mathbb{S}^{\rho,\eps}$ and introduce $\boldsymbol{b}^\eps_1$, $\boldsymbol{b}^\eps_2$ two corresponding normalized eigenvectors. As we have seen, $\boldsymbol{b}^\eps_1$, $\boldsymbol{b}^\eps_2$ are orthogonal to each other and can be chosen to be real valued (Lemma \ref{LemmaUnitary}). For $\boldsymbol{a}$, we have the decomposition 
\[
\boldsymbol{a}=(\boldsymbol{a},\boldsymbol{b}^\eps_1)\,\boldsymbol{b}^\eps_1+(\boldsymbol{a},\boldsymbol{b}^\eps_2)\,\boldsymbol{b}^\eps_2.
\] 
We deduce that there holds 
\begin{equation}\label{PropertySR}
\mathbb{S}^{\rho,\eps}\boldsymbol{a}=\tau_1^\eps(\boldsymbol{a},\boldsymbol{b}^\eps_1)\,\boldsymbol{b}^\eps_1+\tau_2^\eps(\boldsymbol{a},\boldsymbol{b}^\eps_2)\,\boldsymbol{b}^\eps_2.
\end{equation}
For the eigenvalues $\tau_1^\eps$, $\tau_2^\eps$, we will show in Proposition \ref{PropositionCounterclock} below the asymptotics 
\begin{equation}\label{asympto_beta}
\tau_1^\eps=-1-iT\eps\rho+O(\eps^2),\qquad\qquad \tau_2^\eps=\tau_2+O(\eps).
\end{equation}
Here
\begin{equation}\label{def_T}
T=\cfrac{1}{2}\,\int_{\Gamma} h|\partial_nv|^2\,ds>0
\end{equation}
($v\not\equiv0$ on $\Gamma$ is defined in (\ref{Decompov})) and $\tau_2$ is the eigenvalue of $\mathbb{S}$, the threshold scattering matrix in $\Om_\star$, which is different from $-1$. Taking the inner product of (\ref{CompactForm2Bis}) against $\boldsymbol{b}^\eps_1$ and using (\ref{PropertySR}), we get $(\boldsymbol{a},\boldsymbol{b}^\eps_1)=\eps(\partial_x\boldsymbol{\gamma}(0),\boldsymbol{b}^\eps_1)-\tau_1^\eps(\boldsymbol{a},\boldsymbol{b}^\eps_1)$ and so $(\boldsymbol{a},\boldsymbol{b}^\eps_1)=(1+\tau_1^\eps)^{-1}\eps(\partial_x\boldsymbol{\gamma}(0),\boldsymbol{b}^\eps_1)$. Exploiting (\ref{asympto_beta}) and identifying the terms of order $\eps^0$ in the expansions
\[
(\boldsymbol{a},\boldsymbol{b}^\eps_1)=(\boldsymbol{a},\boldsymbol{b}_1)+O(\eps),\qquad\qquad (1+\tau_1^\eps)^{-1}\eps(\partial_x\boldsymbol{\gamma}(0),\boldsymbol{b}^\eps_1)=\cfrac{i}{T\rho}(\partial_x\boldsymbol{\gamma}(0),\boldsymbol{b}_1)+O(\eps),
\]
we obtain
\[
(\boldsymbol{a},\boldsymbol{b}_1)=\cfrac{i}{T\rho}\,(\partial_x\boldsymbol{\gamma}(0),\boldsymbol{b}_1).
\]
Similarly, we find $(\boldsymbol{a},\boldsymbol{b}_2)=0$. Therefore, we get 
\begin{equation}\label{Expressiona}
\boldsymbol{a}=\cfrac{i}{T\rho}\,(\partial_x\boldsymbol{\gamma}(0),\boldsymbol{b}_1)\,\boldsymbol{b}_1.
\end{equation}
Finally, inserting (\ref{Expressiona}) into (\ref{CompactForm1Bis}), we find that this leads us to complement (\ref{farfield}) with the conditions
\begin{equation}\label{TconditionsSys}
\boldsymbol{\gamma}(0)=-\cfrac{2}{T\rho}\,(\partial_x\boldsymbol{\gamma}(0),\boldsymbol{b}_1)\,\boldsymbol{b}_1.
\end{equation}
Using as in (\ref{defTheta}) the notation $\boldsymbol{b}_1=(\cos\theta,\sin\theta)^{\top}$ with $\theta\in[0;\pi)$, this writes componentwise
\begin{equation}\label{Tconditions}
\begin{array}{|rcl}
\gamma^+(0)&=&-\cfrac{2}{T\rho}\,\big((\cos\theta)^2\,\partial_x\gamma^+(0)-\sin\theta\cos\theta\,\partial_x\gamma^-(0)\big)\\[3pt]
\gamma^-(0)&=&-\cfrac{2}{T\rho}\,\big(\sin\theta\cos\theta\,\partial_x\gamma^+(0)-(\sin\theta)^2\,\partial_x\gamma^-(0)\big).
\end{array}
\end{equation}
Recombining (\ref{Tconditions}), one finds that it is equivalent to the  weighted Kirchhoff transmission conditions 
\begin{equation}\label{TK}
\begin{array}{|rcl}
\sin\theta\,\gamma^+(0)-\cos\theta\,\gamma^-(0)&=&0\\[3pt]
\cos\theta\,\partial_x\gamma^+(0)-\sin\theta\,\partial_x\gamma^-(0)&=&-\cfrac{T\rho}{2}\,(\cos\theta\,\gamma^+(0)+\sin\theta\,\gamma^-(0)).
\end{array}
\end{equation}
From (\ref{TK}), when $\rho\to\pm\infty$, asymptotically, we find back the homogeneous Dirichlet conditions 
\[
\gamma^\pm(0)=0.
\]
On the other hand, from (\ref{TK}), when $\rho\to0$, we get 
\begin{equation}\label{TconditionsKirch}
\begin{array}{|rcl}
\sin\theta\,\gamma^+(0)&=&\cos\theta\,\gamma^-(0)\\[3pt]
\cos\theta\,\partial_x\gamma^+(0)&=&\sin\theta\,\partial_x\gamma^-(0)
\end{array}
\end{equation}
as in (\ref{Trans1}), (\ref{Trans2}). Let us study in more details the spectrum of (\ref{farfield}), (\ref{TK}) with respect to $\rho\in\R$. Classically, one finds that $(\nu,\gamma)\in\Cplx\times\mathcal{H}^1_\eta(I)\setminus\{0\}$ (see (\ref{defHspace}) for the definition of the Hilbert space $\mathcal{H}^1_\eta(I)$ which incorporates the first transmission condition of (\ref{TconditionsKirch})) is an eigenpair of (\ref{farfield}), (\ref{TK}) if and only if there holds
\begin{equation}\label{EigenPb1}
\begin{array}{l}
\dsp\int_{I_-\cup I_+}\partial_x \gamma \partial_x \overline{\psi}\,dx -\cfrac{T\rho}{2}\,\bigg(\cos\theta\,\gamma_+(0)+\sin\theta\,\gamma_-(0)\bigg)\bigg(\cos\theta\overline{\psi_+(0)}+\sin\theta\overline{\psi_-(0)}\bigg)\\[12pt]
\hspace{6cm}=\dsp\nu\,\int_{I_-\cup I_+} \gamma \overline{\psi}\,dx,\qquad \psi\in\mathcal{H}^1_\eta(I).
\end{array}
\end{equation}
Therefore, we see that the spectrum of (\ref{farfield}), (\ref{TK}) is made of a sequence of real eigenvalues 
\begin{equation}\label{defEigenval}
\nu^{\rho}_1(\eta) \le \nu^{\rho}_2(\eta) \le  \dots \le \nu^{\rho}_m(\eta) \le \dots
\end{equation}
which is bounded from below and which accumulates only at $+\infty$. Note that we indicate again the dependence with respect to $\rho\in\R$. For $\rho\le0$, from (\ref{EigenPb1}), we see that $\nu^\rho_1(\eta)\ge0$. However when $\rho\ge0$, one can have $\nu^\rho_1(\eta)\le0$ as we will see below. Working as in (\ref{Dispersion1}), we can characterize more explicitly the $\nu^\rho_p(\eta)$. First we find that $\nu$ is a non-zero eigenvalue of (\ref{farfield}), (\ref{TK}) if and only if there holds 
\begin{equation}\label{Dispersion1Bis}
\sin(2\theta)\cos\eta=\cos\sqrt{\nu}-\cfrac{T\rho}{2}\,\cfrac{\sin\sqrt{\nu}}{\sqrt{\nu}}\,.
\end{equation} 
On the other hand, we obtain that zero is an eigenvalue if and only if we have
\begin{equation}\label{Dispersion2Bis}
\sin(2\theta)\cos\eta=1-\cfrac{T\rho}{2}\,.
\end{equation} 
Remark that for $\rho=0$, these relations coincide with (\ref{Dispersion1}), (\ref{Dispersion2}), which was expected. For $p=N_{\bullet}+m,\,m\in\N^{\ast}$, we have obtained 
\begin{equation}\label{ExpanConcluBis}
\Lambda^{\rho,\eps}_p(\eta)=\eps^{-2}\pi^2+\nu^\rho_m(\eta)+\dots
\end{equation}
where $\nu^\rho_m(\eta)$ solves (\ref{Dispersion1Bis}) or (\ref{Dispersion2Bis}). This analysis shows that for the spectral bands $\Upsilon^{\rho,\eps}_p=[a^{\rho,\eps}_{m-};a^{\rho,\eps}_{m+}]=\{\Lambda_{p}^{\rho,\eps}(\eta),\,\eta\in[0;2\pi)\}$, we have the asymptotics
\[
a^{\rho,\eps}_{m\pm}=\eps^{-2}\pi^2+c^\rho_{m\pm}+\dots
\]
with 
\[
c^\rho_{m-}=\inf_{\eta\in[0;2\pi)}\nu^\rho_m(\eta),\qquad\qquad c^\rho_{m+}=\sup_{\eta\in[0;2\pi)}\nu^\rho_m(\eta).
\]
This gives (\ref{result_thm_migration}). Finally, we establish (\ref{result_thm_migration_2}). Let us detail how to proceed to obtain the asymptotics for $c^\rho_{1\pm}$. Let $\gamma\in\mathcal{H}^1_\eta(I)\setminus\{0\}$ be an eigenfunction of (\ref{EigenPb1}) associated with $\nu_1^\rho(\eta)$. 
Consider the ansatz, for $\rho>0$ large,
\[
\nu_1^\rho(\eta)=-\alpha^2\,\rho^2+\dots,\qquad\qquad
\begin{array}{|l}
\gamma^+(x)=A\,\cos\theta\,e^{-\alpha \rho x}+\dots,\quad x>0 \\[3pt]
\gamma^-(x)=A\,\sin\theta\,e^{+\alpha \rho x}+\dots,\quad x<0 ,
\end{array}
\]
where $\alpha$, $A$ are constants and the dots correspond to small remainders. Note that with this choice, the first two equations of (\ref{farfield}) and the first one of (\ref{TK}) are satisfied. Imposing the second condition of (\ref{TK}) yields
\[
-((\cos\theta)^2+(\sin\theta)^2)\,A \alpha \rho = -\cfrac{T\rho}{2}\,((\cos\theta)^2+(\sin\theta)^2)\,A.
\]
This leads us to take $\alpha=T/2$. Besides, since the quantities $\gamma^\pm(\pm1/2)$, $\partial_x\gamma^\pm(\pm1/2)$ decay exponentially as $\rho\to+\infty$, the discrepancy in the quasi-periodic conditions of (\ref{farfield}) is $O(e^{-\delta\rho})$ with $\delta>0$. This proves the formula $c^\rho_{1\pm}\underset{\rho\to+\infty}{\sim}-T^2\rho^2/4$ of (\ref{result_thm_migration_2}). Additionally, this guarantees that the segment $[c^\rho_{1-};c^\rho_{1+}]$ becomes exponentially short as $\rho\to+\infty$. The justification of these calculus is classical. Now in (\ref{result_thm_migration_2}) we turn our attention to the asymptotics of the $c^\rho_{m\pm}$, $m\in\N^\ast$, as $\rho\to-\infty$. For $(\nu_m^\rho(\eta),\gamma)$ an eigenpair of (\ref{EigenPb1}), one can show the expansions
\[
\nu_m^\rho(\eta)=m^2\pi^2+\dots,\qquad\qquad
\begin{array}{|l}
\gamma^+(x)=A\,\sin(m\pi x)+\dots,\quad x>0 \\[3pt]
\gamma^-(x)=-A\,e^{i\eta}\sin(m\pi x)+\dots,\quad x<0 ,
\end{array}
\] 
where $A$ is a constant and the dots stand for terms in $O(1/\rho)$. Observe that this implies that the length of the segment $[c^\rho_{m-};c^\rho_{m+}]$ behaves as $O(1/\rho)$ when $\rho\to-\infty$. As a consequence, it is larger than $[c^\rho_{1-};c^\rho_{1+}]$. One works similarly on the asymptotics of $c^\rho_{(m+1)\pm}$ as $\rho\to+\infty$, $m\in\N^\ast$.\\
\newline
Note that for the particular Floquet-Bloch parameter $\eta=\pi/2$, from the  the explicit dispersion relations (\ref{Dispersion1Bis}), (\ref{Dispersion2Bis}), wee that the $\nu^\rho_m(\pi/2)$ coincide with the real numbers $t$ such that 
\[
\cfrac{\rho}{\sqrt{t}}\,\tan \sqrt{t}=\cfrac{2}{T}
\]
when $\rho\ne0$. This clearly implies 
\[
\lim_{\rho\to-\infty}\nu^\rho_m(\pi/2)=m^2\pi^2,\qquad\qquad\mbox{ and }\qquad\qquad\lim_{\rho\to+\infty}\nu^\rho_ {m+1}(\pi/2)=m^2\pi^2.
\]

\subsection{Illustration of the results}

As in (\ref{SpectralBandsBis}), we introduce the sets
\[
\Xi^\rho_m(\theta):=[c^\rho_{m-};c^\rho_{m+}]=\{\nu^\rho_m(\eta),\,\eta\in[0;2\pi)\},\qquad\qquad\aleph^\rho(\theta):=\bigcup_{m=1}^{+\infty}\Xi^\rho_m(\theta).
\]
In Figures \ref{ImageAsymptoRho}, \ref{ImageAsymptoRho1}, we display the $\Xi^\rho_m(\theta)$ with respect to $\rho$. In Figure \ref{ImageAsymptoRho}, we take $\sin(2\theta)=0.7$ while we impose $\sin(2\theta)=1$ in Figure \ref{ImageAsymptoRho1}. On the other hand, $T$ is arbitrarily chosen equal to $2$. To obtain these pictures, we exploit (\ref{Dispersion1Bis}) in a simple way: for a given range of values of $\nu$, we check whether or not there holds 
\[
\bigg|\cos\sqrt{\nu}-\cfrac{T\rho}{2}\,\cfrac{\sin\sqrt{\nu}}{\sqrt{\nu}}\bigg| \le |\sin(2\theta)|\,.
\]
In accordance with Theorem \ref{MainThmPerioMigration}, we observe that the  $\Xi^\rho_m(\theta)$ tend to singletons when $\pm\rho\to+\infty$. More precisely, when $\rho$ increases, i.e. when the near field geometry is enlarged, the $\Xi^\rho_m(\theta)$ expand up to $\rho=0$. Then the band $\Xi^\rho_1(\theta)$ emerges in $(-\infty;0)$ while the other spectral bands shrink as $\rho$ continues to grow. Note also that the whole spectrum comes down, which is expected because with Dirichlet boundary conditions, the eigenvalues decrease when the domain inflates.

\begin{figure}[!ht]
\centering
\includegraphics[width=14cm]{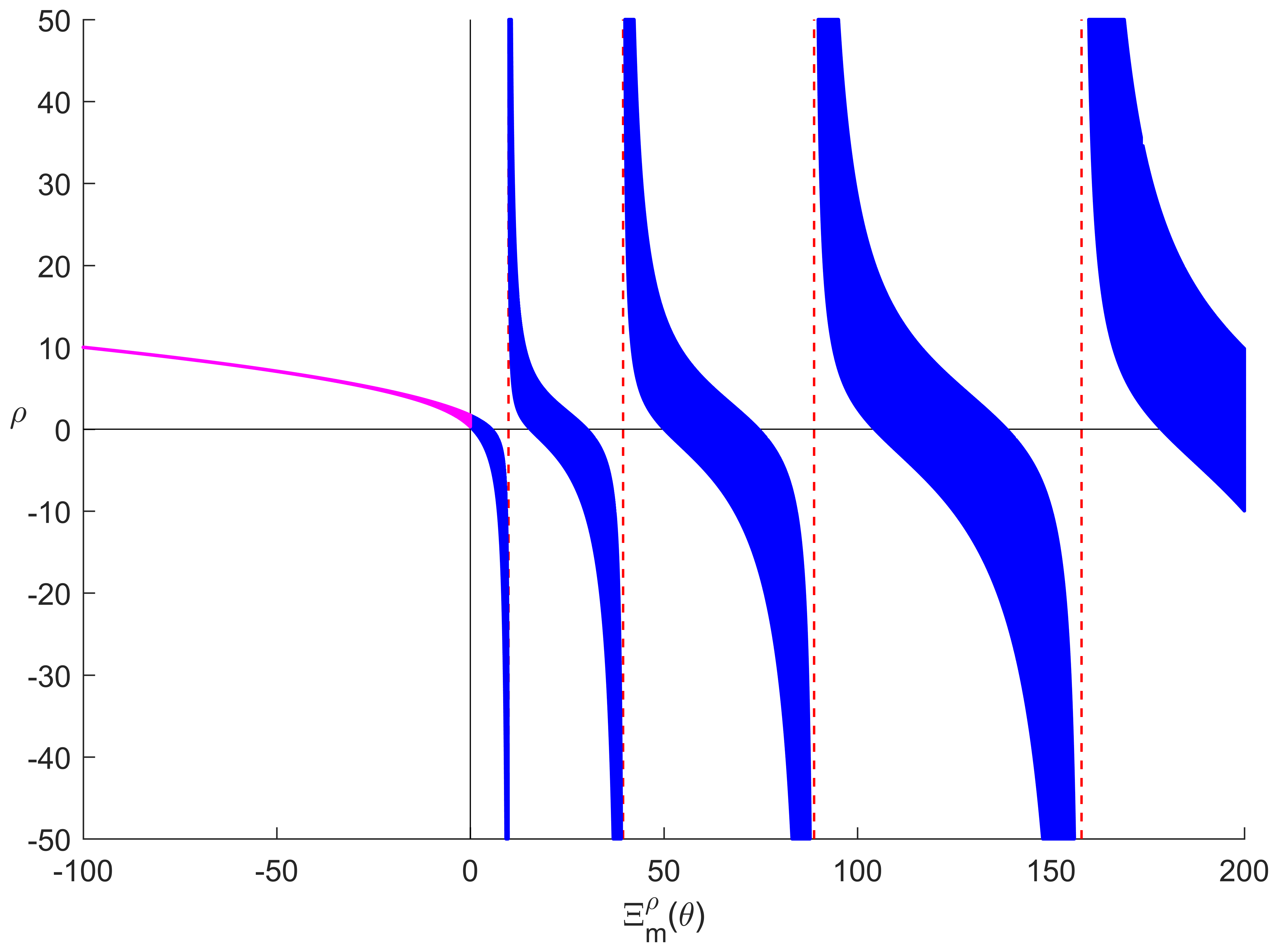}
\caption{Sets $\Xi^\rho_m(\theta)$ (in magenta and blue) with respect to $\rho$. We use the magenta colour to stress the situations where $\Xi^\rho_m(\theta)$ take negative values. The vertical red dashed lines correspond to the $m^2\pi^2$, $m\in\N^{\ast}$. Here we take $T=2$ (arbitrarily) and $\sin(2\theta)=0.7$.
\label{ImageAsymptoRho}}
\end{figure}

\begin{figure}[!ht]
\centering
\includegraphics[width=14cm]{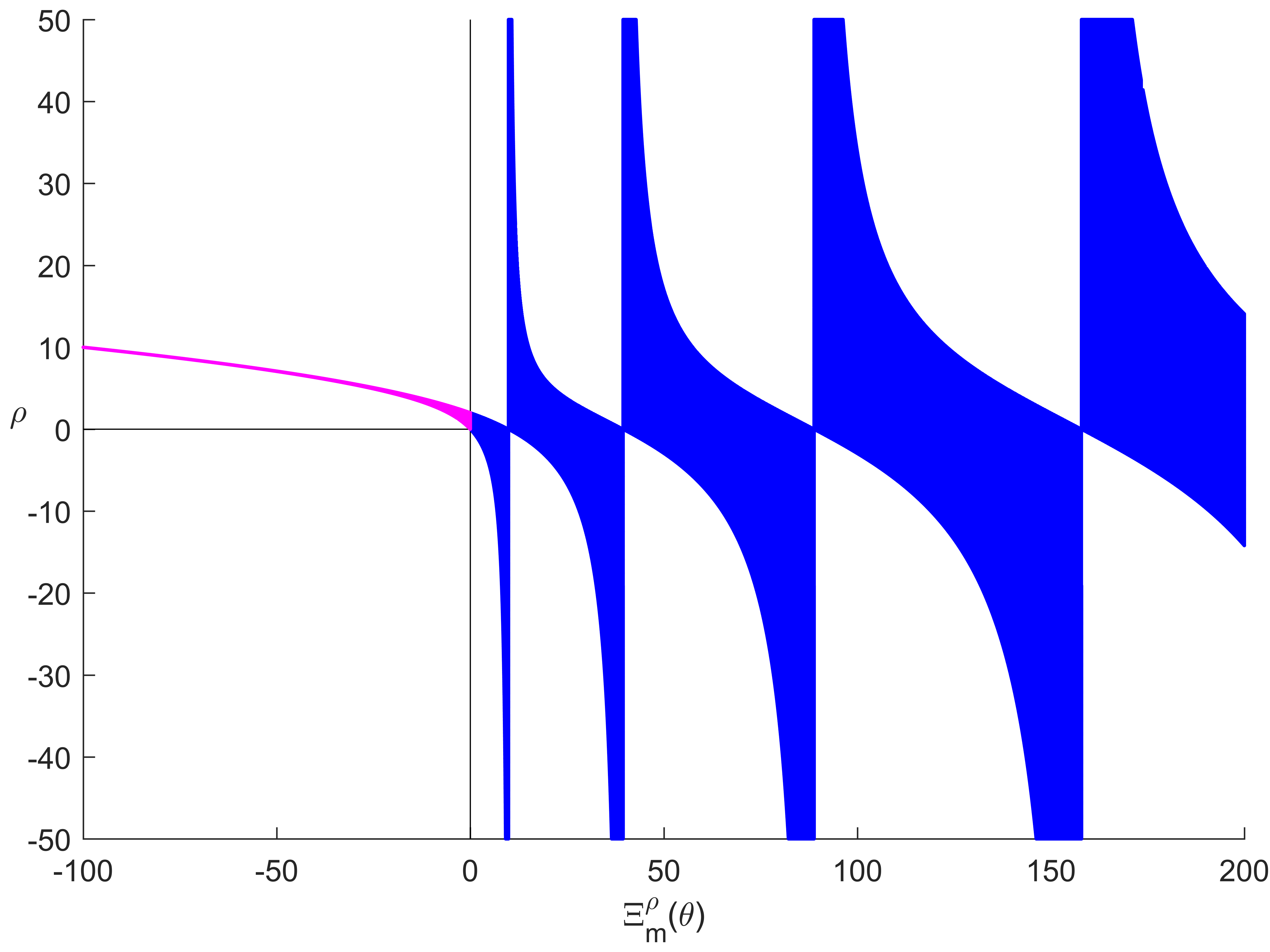}
\caption{Same quantities as in Figure \ref{ImageAsymptoRho}
but this time with $\sin(2\theta)=1$. 
\label{ImageAsymptoRho1}}
\end{figure}

\subsection{Behaviour of the eigenvalues of $\mathbb{S}$ under small perturbation of the geometry}
\noindent In this section, we show a result concerning the asymptotics of the eigenvalues of the threshold scattering matrix $\mathbb{S}$ (see  (\ref{DefTRScaMa}) for its definition) which has been used in the proof of Theorem \ref{MainThmPerioMigration} (for a close study, see also \cite{NaRU21}). Let $\Om$ be a waveguide which coincides with the strip $\R\times(-1/2;1/2)$ outside of a bounded region as described in Section \ref{SectionSetting}. Note that our result below is general and we do not assume here that $\dim\,\mX_\dagger=1$. Consider $\tau$ a simple eigenvalue of $\mathbb{S}$. We want to understand how this eigenvalue evolves when the geometry is slightly perturbed.\\
\newline
Define $\Om^\eps$ from $\Om$ as $\Om^{\rho,\eps}$ was constructed from $\Om_\star$ in (\ref{DefGeom}) with $\rho=1$. In particular, outside of $\mathscr{V}$, we have $\partial\Om^{\eps}=\partial\Om$ and inside $\mathscr{V}$,  $\partial\Om^{\eps}$ coincides with 
\begin{equation}\label{DefGeom}
\Gamma^{\eps}:=\{P(s)+\eps h(s)n(s)\,|\,s\in J\}.
\end{equation} 
We denote by $\mathbb{S}^\eps$ the threshold scattering matrix in $\Om^\eps$ and by $\tau^\eps$ its eigenvalue corresponding to the perturbation of $\tau$.

\begin{proposition}\label{PropositionCounterclock}
There are $\eps_0>0$, $C>0$ such that we have 
\begin{equation}\label{MainResultAsy}
\Big|\tau^\eps-\Big(\tau+\cfrac{i\tau\eps}{2}\,\int_{\Gamma} h|\partial_nv|^2\,ds\Big)\Big| \le C\,\eps^2,\qquad\forall \eps\in(0;\eps_0].
\end{equation}
Here $v\not\equiv0$ on $\Gamma$ is a function defined in (\ref{Decompov}).
\end{proposition}
\begin{remark}
This result also holds without sign assumption for $h$. However, we see that when we take $h\ge0$ with $h\not\equiv0$, i.e. when the geometry inflates, the eigenvalues of the threshold scattering matrix rotates counter-clockwise on the unit circle. On the contrary, when $\Om$ deflates, the eigenvalues of $\mathbb{S}$ rotate clockwise.
\end{remark}
\begin{proof}
To obtain an asymptotic expansion of $\tau^\eps$ when $\eps\to0^+$, we have to compute an expansion of the functions $v_{\pm}^\eps$ defined as $v_{\pm}$ in (\ref{ScaSol}) but in the geometry $\Om^\eps$. Let us focus our attention on $v_{+}^\eps$. We consider the simplest ansatz
\begin{equation}\label{Decompo}
v_{+}^\eps=v_{+}+\eps v_+'+\dots
\end{equation}
where $v_+'$ is to determine and the dots stand for higher-order terms. Inserting (\ref{Decompo}) in the problem (\ref{NearFieldPb}) posed in $\Om^\eps$, taking the limit $\eps\to0^+$ and identifying the powers in $\eps$, we find that $v_{+}$ coincides with the function introduced in (\ref{ScaSol}). Moreover, exploiting in particular the Taylor expansion, for $s\in J$, 
\[
\begin{array}{rcl}
0=v_{+}^\eps(P(s)+\eps h(s)n(s))&=& v_{+}^\eps(P(s))+\eps h(s)n(s)\cdot\nabla v^\eps_{+}(P(s))+O(\eps^2)\\[3pt]
&=&v_{+}(P(s))+\eps\,(v'_{+}(P(s))+h(s)n(s)\cdot\nabla v_{+}(P(s)))+O(\eps^2),
\end{array}
\]
we obtain that $v'_{+}$ satisfies the problem
\[
\begin{array}{|rcll}
\Delta v'_{+}+\pi^2v'_{+}&=&0&\mbox{ in }\Om\\[3pt]
v'_{+}&=&0 &\mbox{ on }\partial\Om\setminus\overline{\Gamma}\\[3pt]
v'_{+}&=& -hn\cdot\nabla v_{+} &\mbox{ on }\Gamma.
\end{array}
\]
Note that above and below, we naturally define $h(z)$, $n(z)$ for $z\in\Gamma$ by taking the values of $h(s)$, $n(s)$ for $s$ such that $z=P(s)$. On the other hand, since the incident field in the problem defining $v_{+}^\eps$ is independent of $\eps$, we find that $v'_{+}$ must admit the expansion
\begin{equation}\label{DecompoPerturb}
v'_+  =\psi_-s'_{+-}w^{\mrm{out}}+\psi_+ s'_{++}w^{\mrm{out}}+\tilde{v}'_+ 
\end{equation}
where $s'_{+\pm}$ are complex numbers and $\tilde{v}'_+$ decays exponentially as $|x|\to+\infty$. Now we wish to obtain expressions for the $s'_{+\pm}$. To proceed, we start from the identity
\begin{equation}\label{IntegralIPP}
0=\int_{\Om_\kappa }(\Delta v'_{+}+\pi^2v'_{+})v_\pm-v'_{+}(\Delta v_\pm+\pi^2v_\pm)\,dz
\end{equation}
where for $\kappa>0$, $\Om_\kappa:=\{z\in\Om\,|\,|x|\le \kappa\}$. Integrating by parts in (\ref{IntegralIPP}), we obtain for large $\kappa$
\begin{equation}\label{IntegralIPP2}
0=\int_{x=\kappa}\partial_x v'_{+}v_\pm-v'_{+}\partial_xv_\pm\,dy-\int_{x=-\kappa}\partial_x v'_{+}v_\pm-v'_{+}\partial_xv_\pm\,dy+\int_{\Gamma} h\partial_nv_+\partial_nv_\pm\,ds.
\end{equation}
Taking the limit $\kappa\to+\infty$ in (\ref{IntegralIPP2}) and working with the decompositions (\ref{ScaSol}), (\ref{DecompoPerturb}), we find
\[
0=2is'_{+\pm}+\int_{\Gamma} h\partial_nv_+\partial_nv_\pm\,ds\qquad\mbox{ and  so }\qquad s'_{+\pm}=\cfrac{i}{2}\,\int_{\Gamma} h\partial_nv_+\partial_nv_\pm\,ds.
\]
We get similar results when working with $v_-^\eps$. Finally, this analysis shows that
\[
\mathbb{S}^\eps=\mathbb{S}+\eps\mathbb{S}'+O(\eps^2)
\]
where $\mathbb{S}$ is the threshold scattering matrix appearing in (\ref{DefTRScaMa}) and 
\begin{equation}\label{TermsSP}
\mathbb{S}'=\left(\begin{array}{cc}
s'_{++} & s'_{+-}\\[3pt]
s'_{-+} & s'_{--}\\[3pt]\end{array}
\right)\qquad\mbox{ with }\qquad
\begin{array}{|rcl}
s'_{\pm\pm}&=&\dsp\cfrac{i}{2}\,\int_{\Gamma} h\partial_nv_\pm\partial_nv_\pm\,ds\\[6pt]
s'_{\pm\mp}&=&\dsp\cfrac{i}{2}\,\int_{\Gamma} h\partial_nv_\pm\partial_nv_\mp\,ds.
\end{array}
\end{equation}
Now we come to the asymptotics of $\tau^\eps$ (which is by assumption a simple eigenvalue of $\mathbb{S}^\eps$). Introduce $\boldsymbol{b}^\eps\in\R^2\setminus\{0\}$ a corresponding eigenvector. For $\tau^\eps$, $\boldsymbol{b}^\eps$, we work with the simplest ansatz
\begin{equation}\label{ExpansionEigen}
\tau^\eps=\tau+\eps\tau'+\dots,\qquad\quad\boldsymbol{b}^\eps=\boldsymbol{b}+\eps\boldsymbol{b}'+\dots,
\end{equation}
where again the dots denote unessential higher order terms. Inserting the expansions (\ref{ExpansionEigen}) in the relation $\mathbb{S}^\eps\boldsymbol{b}^\eps=\tau^\eps\boldsymbol{b}^\eps$ and identifying the powers in $\eps$, we get
\begin{equation}\label{IdenPowers}
\mathbb{S}\boldsymbol{b}=\tau\boldsymbol{b},\qquad\qquad \mathbb{S}\boldsymbol{b}'-\tau\boldsymbol{b}'=-(\mathbb{S}'\boldsymbol{b}-\tau'\boldsymbol{b}).
\end{equation}
We deduce first that $(\tau,\boldsymbol{b})$ is an eigenpair of $\mathbb{S}$. Besides, taking the inner product of the second equality of (\ref{IdenPowers}) with $\boldsymbol{b}$ (which is real), we find
\[
\begin{array}{ll}
&\phantom{-}(\mathbb{S}\boldsymbol{b}'-\tau\boldsymbol{b}',\boldsymbol{b})=(\boldsymbol{b}',\overline{\mathbb{S}}^\top\boldsymbol{b}-\overline{\tau}\boldsymbol{b})=(\boldsymbol{b}',\overline{\mathbb{S}}\boldsymbol{b}-\overline{\tau}\boldsymbol{b})=0\\[5pt]
=&-(\mathbb{S}'\boldsymbol{b}-\tau'\boldsymbol{b},\boldsymbol{b}),
\end{array}
\]
which gives
\begin{equation}\label{FormulaTauP}
\tau'=\cfrac{(\mathbb{S}'\boldsymbol{b},\boldsymbol{b})}{(\boldsymbol{b},\boldsymbol{b})}\,.
\end{equation}
Note that as expected, $\tau'$ is independent of the choice of the normalization made for $\boldsymbol{b}^\eps$. Since $\tau$ belongs to the unit circle in the complex plane, there is $\varsigma\in[0;2\pi)$ such that $\tau=e^{i\varsigma}$. Introduce $b_+$, $b_-\in\R$ such that $\boldsymbol{b}=(b_+,b_-)^\top$. Additionally, we assume that we have $(\boldsymbol{b},\boldsymbol{b})=1$. Combining (\ref{TermsSP}) and (\ref{FormulaTauP}), we obtain 
\begin{equation}\label{ScdFormula}
\tau'=\cfrac{i}{2}\,\int_{\Gamma} h(\partial_nv)^2\,ds\qquad\mbox{ with }\qquad v:=b_+\,v_++b_-\,v_-.
\end{equation}
But the function $v$ admits the expansion
\begin{equation}\label{Decompov}
\begin{array}{rcl}
v&=&\psi_-(b_-w^{\mrm{in}}+(b_+s_{+-}+b_-s_{--})w^{\mrm{out}})+\psi_+(b_+w^{\mrm{in}}+(b_+s_{++}+b_-s_{-+})w^{\mrm{out}})+\tilde{v} \\[3pt]
&=&b_-\psi_-(w^{\mrm{in}}+e^{i\varsigma}w^{\mrm{out}})+b_+\psi_+(w^{\mrm{in}}+e^{i\varsigma}w^{\mrm{out}})+\tilde{v}
\end{array}
\end{equation}
with $\tilde{v}$ which decays exponentially as $|x|\to+\infty$. From (\ref{Decompov}), we infer that $e^{-i\varsigma/2}v$ is valued in $\R$ because $e^{-i\varsigma/2}v-\overline{e^{-i\varsigma/2}v}$ is exponentially decaying at infinity. Therefore we have
\[
\tau'=\cfrac{i\tau}{2}\,\int_{\Gamma} h|\partial_nv|^2\,ds.
\]
On the other hand, the unique continuation principle ensures that $\partial_nv$ can not vanish identically on $\Gamma$. With these two properties, together with (\ref{ScdFormula}), finally we find the desired result (\ref{MainResultAsy}). Our approach above is rather formal but rigorous justification can be obtained classically by converting the perturbation of the geometry into a perturbation in an equation set in a fixed domain. For these techniques, we refer the reader for example to \cite{DeZo11,HePi18} and \cite[Chap.\,7,\,\S6.5]{Kato95}.
\end{proof}

\section{Simple examples of near field geometries where $\dim\,\mX_\dagger=1$}\label{SectionNFgeom}
 
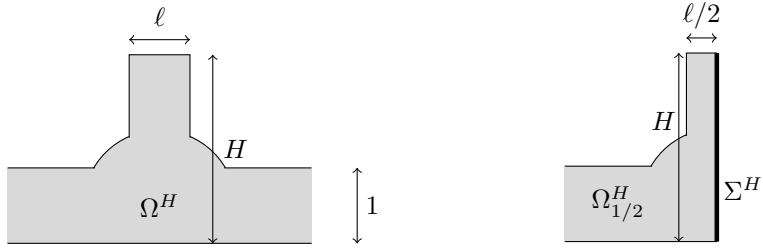
\begin{figure}[!ht]
\centering
\begin{tikzpicture}[scale=1]
\begin{scope}
\clip(-1,-1/2) rectangle (1,2);
\draw[fill=gray!30] (0,0) circle [radius=1];
\end{scope}
\draw[fill=gray!30,draw=none](-2,-1/2) rectangle (2,1/2);
\draw[fill=gray!30,draw=none](-0.4,-1/2) rectangle (0.4,2);
\draw[black] (-2,-1/2)--(2,-1/2);
\draw[black] (-2,1/2)--(-0.85,1/2);
\draw[black] (-0.4,0.9)--(-0.4,2)--(0.4,2)--(0.4,0.9);
\draw[black] (0.85,1/2)--(2,1/2);
\draw[black,<->] (1/2+0.2,-1/2)--(1/2+0.2,2);
\node at (1/2+0.5,0.75) {\small $ H$};
\draw[black,<->] (-0.4,2.2)--(0.4,2.2);
\node at (0,2.5) {\small $ \ell$};
\draw[black,<->] (2+0.6,-1/2)--(2+0.6,1/2);
\node at (2+0.8,0) {\small $1 $};
\node at (0,0) {\small $\Om^H$};
\end{tikzpicture}\qquad\qquad\qquad\begin{tikzpicture}[scale=1]
\begin{scope}
\clip(-1,-1/2) rectangle (0,2);
\draw[fill=gray!30] (0,0) circle [radius=1];
\end{scope}
\draw[fill=gray!30,draw=none](-2,-1/2) rectangle (0,1/2);
\draw[fill=gray!30,draw=none](-0.4,-1/2) rectangle (0,2);
\draw[black] (-2,-1/2)--(0,-1/2);
\draw[black] (-2,1/2)--(-0.85,1/2);
\draw[black] (-0.4,0.9)--(-0.4,2)--(0,2);
\draw[black,line width=0.6mm] (0,-1/2)--(0,2);
\draw[black,<->] (0-0.5,-1/2)--(0-0.5,2);
\node at (0-0.7,1.1) {\small $ H$};
\node at (0+0.35,0.2) {\small $\Sigma^H$};
\draw[black,<->] (-0.4,2.2)--(0,2.2);
\node at (-0.2,2.5) {\small $ \ell/2$};
\node at (-1.3,0) {\small $\Om^H_{1/2}$};
\end{tikzpicture}
\caption{Geometries of $\Om^H$ (left) and $\Om^H_{1/2}$ (right). 
\label{GeometryPartH}} 
\end{figure}

\noindent In this section, we explain how to exhibit simple geometries for which $\dim\,\mX_\dagger=1$. To proceed, we adapt ideas used in \cite{ChNP18,ChPaSu}. For $\ell$ fixed in $(1;2)$ and $H>1$, we assume that $\Om$, that we denote now $\Om^H$, coincides with 
\[
\mathcal{S}\cup \mathcal{R}^H,\qquad\mbox{ where }\qquad\mathcal{S}=\R\times(-1/2;1/2),\ \mathcal{R}^H:=(-\ell/2;\ell/2)\times(0;H), 
\]
outside of $B(O,r_0)$, the ball centred at $O$ and of radius $r_0>\ell$. Moreover, we assume that $\Om^H$ is symmetric with respect to the $(Oy)$ axis, i.e. that there holds
\[
\Om^H=\{(-x,y)\,|\,(x,y)\in\Om^H\}
\]
(see an example of such $\Om^H$ in Figure \ref{GeometryPartH} left). Below, we use the objects introduced in Section \ref{SectionNFPb}. We add a superscript ${}^H$ to indicate the dependence in $H$. As seen in (\ref{ScaSol}), Problem (\ref{NearFieldPb}) admits the solutions
\begin{equation}\label{ScaSolH}
\begin{array}{rcl}
v_+^H & =& \psi_-s_{+-}^Hw^{\mrm{out}}+\psi_+(w^{\mrm{in}}+s_{++}^Hw^{\mrm{out}})+\tilde{v}_+^H \\[5pt]
v_-^H & =& \psi_-(w^{\mrm{in}}+s_{--}^Hw^{\mrm{out}})+\psi_+s_{-+}^Hw^{\mrm{out}}+\tilde{v}_-^H
\end{array}
\end{equation}
where $s_{\pm\pm}^H$, $s_{\mp\pm}^H\in\Cplx$ and where $\tilde{v}_\pm^H$ decay exponentially as $|x|\to+\infty$. By observing that the function $(x,y)\mapsto v_-^H(-x,y)$ has the same expansion as $v_+^H$ at infinity, by uniqueness of the definition of the scattering matrix, we deduce that $\mathbb{S}^H$ has the simple form
\begin{equation}\label{defMatScaSym}
\mathbb{S}^H=\left(
\begin{array}{cc}
\mathcal{R}^H &  \mathcal{T}^H\\[2pt]
\mathcal{T}^H & \mathcal{R}^H
\end{array}
\right)\qquad\mbox{ with }\qquad \mathcal{R}^H=s^H_{++}=s^H_{--}\quad\mbox{ and }\quad\mathcal{T}^H=s^H_{+-}=s^H_{-+}.
\end{equation}
Let us continue to exploit the property of symmetry of $\Om^H$. Classically, define the half-waveguide  
\[
\Om^H_{1/2}:=\{(x,y)\in\Om^H\,|\,x<0\}
\]
and consider the problem with Dirichlet boundary conditions 
\begin{equation}\label{DemiPbD}
\begin{array}{|rcll}
\Delta U+\pi^2U&=&0&\mbox{ in }\Om^H_{1/2}\\[3pt]
U&=&0 &\mbox{ on }\partial\Om^H_{1/2}
\end{array}
\end{equation}
as well as the one with mixed boundary conditions 
\begin{equation}\label{DemiPbN}
\begin{array}{|rcll}
\Delta u+\pi^2u&=&0&\mbox{ in }\Om^H_{1/2}\\[3pt]
\partial_n u&=&0 &\mbox{ on }\Sigma^H:=\partial\Om^H_{1/2}\cap\big(\{0\}\times(0;H)\big)\\[3pt]
u&=&0 &\mbox{ on }\partial\Om^H_{1/2}\setminus\overline{\Sigma^H}.
\end{array}
\end{equation}
Here $\partial_n=\partial_x$ on $\Sigma^H$. Problems (\ref{DemiPbD}) and (\ref{DemiPbN}) admit respectively the solutions
\begin{equation}\label{DefSolDemi}
\begin{array}{rcl}
U^H&=&w^{\mrm{in}}+r_D^H w^{\mrm{out}}+\tilde{U}^H\\[5pt]
u^H&=&w^{\mrm{in}}+r_N^H w^{\mrm{out}}+\tilde{u}^H
\end{array}
\end{equation}
where $r_D^H$, $r_N^H\in\Cplx$ are uniquely defined and $\tilde{U}^H,\,\tilde{u}^H\in\mH^1(\Om^H_{1/2})$. Due to conservation of energy, one has
\[
|r_D^H|=|r_N^H|=1.
\]
Direct inspection shows that if $W$ is a solution of Problem (\ref{NearFieldPb}) in $\Om^H$, then we have $W(x,y)=(u^H(x,y)+U^H(x,y))/2$ in $\Om^H_{1/2}$ and $W(x,y)=(u^H(-x,y)-U^H(-x,y))/2$ in $\Om^H\setminus\overline{\Om^H_{1/2}}$ (up possibly to a term which is exponentially decaying at $\pm\infty$ if there are trapped modes). We deduce that the scattering coefficients $\mathcal{R}^H$, $\mathcal{T}^H$ appearing in (\ref{defMatScaSym}) are such that
\begin{equation}\label{IdentitiesSym}
\mathcal{R}^H=\cfrac{r_N^H+r_D^H}{2}\qquad\mbox{ and }\qquad \mathcal{T}^H=\cfrac{r_N^H-r_D^H}{2}\,.
\end{equation}
Computing the characteristic polynomial of $\mathbb{S}^H$ and using the above relations, we find
\[
\mrm{det}(\mathbb{S}^H-\lambda\,\mrm{Id})=(\lambda-\mathcal{R}^H)^2-(\mathcal{T}^H)^2=(\lambda-(\mathcal{R}^H-\mathcal{T}^H))(\lambda-(\mathcal{R}^H+\mathcal{T}^H))=(\lambda-r_D^H)(\lambda-r_N^H).
\]
We deduce that the eigenvalues of $\mathbb{S}^H$ are exactly $r_D^H$ and $r_N^H$. Let us study their behaviour with respect to $H\to+\infty$. \\
\newline
We start with $r_D^H$. When $H\to+\infty$, we are led to study the problem
\begin{equation}\label{DemiPbDinf}
\begin{array}{|rcll}
\Delta U+\pi^2U&=&0&\mbox{ in }\Om^\infty_{1/2}\\[3pt]
U&=&0 &\mbox{ on }\partial\Om^\infty_{1/2}
\end{array}
\end{equation}
where $\Om^\infty_{1/2}:=\bigcup_{H\ge r_0}\Om^H_{1/2}$ ($\Om^\infty_{1/2}$ is unbounded in the $y$ direction). Due to the fact that $\ell\in(1;2)$, in the vertical branch of $\Om^\infty_{1/2}$, which is of width $\ell/2$, no mode of (\ref{DemiPbDinf}) can propagate. For this reason, Problem (\ref{DemiPbDinf}) admits a solution with the expansion 
\begin{equation}\label{DefZetal}
U^{\infty}=\zeta_l\,(w^{\mrm{in}}+r_D^\infty w^{\mrm{out}})+\tilde{U}^{\infty}
\end{equation}
where $r_D^\infty\in\Cplx$, $\tilde{U}^{\infty}\in\mH^1(\Om^\infty_{1/2})$. Here $\zeta_l$ is a smooth cut-off function such that $\zeta_l=1$ for $x<-2r_0$ and $\zeta_l=0$ for $x>-r_0$. Note that by conservation of energy, we have $|r_D^\infty|=1$. Then working as in \cite[Chap.\,5,\,\S5.6]{NaPl94} (see also \cite[Prop.\,8.1]{ChNP18}), one can show that the function $U^H$ introduced in (\ref{DefSolDemi}) is well-approximated by $U^{\infty}$ when $H$ tends to $+\infty$. More precisely, we can prove the estimate
\[
\|U^H-U^\infty\|_{\mH^1(\Om^H_{1/2})} \le C\,e^{-\pi\sqrt{4/\ell^2-1}H}
\]
where $C>0$ is independent of $H$. From this, we infer that $\lim_{H\to+\infty}r_D^H=r_D^\infty$. As a consequence, if $r_D^\infty\ne-1$, which is the case in general, we deduce that we have $r_D^H\ne-1$ for $H$ large enough.\\
\newline
Now let us study the behaviour of $r_N^H$ as $H\to+\infty$.  When $H\to+\infty$, we are led to study the problem
\begin{equation}\label{DemiPbNInf}
\begin{array}{|rcll}
\Delta u+\pi^2u&=&0&\mbox{ in }\Om^\infty_{1/2}\\[3pt]
\partial_n u&=&0 &\mbox{ on }\Sigma^\infty:=\partial\Om^\infty_{1/2}\cap\big(\{0\}\times(0;H)\big)\\[3pt]
u&=&0 &\mbox{ on }\partial\Om^\infty_{1/2}\setminus\overline{\Sigma^\infty}.
\end{array}
\end{equation}
It admits solutions of the form
\begin{equation}\label{DefScaSol}
\begin{array}{rcl}
u^{\infty}_l&=&\zeta_l\,(w^{\mrm{in}}+r_N^\infty\,w^{\mrm{out}})+\zeta_t \,t_N^\infty\,w^{\mrm{out}}+\tilde{u}^{\infty}_l \\[3pt]
u^{\infty}_t&=&\zeta_l\,t_N^\infty\,w^{\mrm{out}}+\zeta_t\,(w^{\mrm{in}}_t+\tilde{r}_N^\infty\, w^{\mrm{out}}_t)+\tilde{u}^{\infty}_t
\end{array}
\end{equation}
where $r_N^\infty$, $\tilde{r}_N^\infty$, $t_N^\infty\in\mathbb{C}$ and $\tilde{u}^{\infty}_l$, $\tilde{u}^{\infty}_t\in\mH^1(\Om^\infty_{1/2})$. Here 
\[
w^{\mrm{out}}_t(x,y)=\beta_\ell^{-1/2}\,e^{i\beta_\ell y}\sqrt{2/\ell}\cos(\pi x/\ell),\qquad w^{\mrm{in}}_t(x,y)=\beta_\ell^{-1/2}\,e^{-i\beta_\ell y}\sqrt{2/\ell}\cos(\pi x/\ell)
\]
with $\beta_\ell:=\pi\sqrt{1-1/\ell^2}>0$. Additionally, $\zeta_l$ is the one introduced in (\ref{DefZetal}) while $\zeta_t$ is a smooth cut-off function such that $\zeta_t=1$ for $y\ge 2r_0$ and $\zeta_t=0$ for $y\le r_0$. Let us look for an expansion of $u^H$ of the form 
\begin{equation}\label{expansionuH}
u^H=u^{\infty}_l+a(H)\,u^{\infty}_t+\dots
\end{equation}
where $a(H)$ is an unknown function and the dots correspond to a small remainder. Exploiting the condition $u^H=0$ at $y=H$, we deduce that we must impose
\[
t_N^\infty e^{i\beta_\ell H}+a(H)(e^{-i\beta_\ell H}+\tilde{r}_N^\infty e^{i\beta_\ell H})=0.
\]
Assuming that $|\tilde{r}_N^\infty|\ne1\Leftrightarrow t_N^\infty\ne0$, this gives 
\[
a(H)=\cfrac{-t_N^\infty }{\tilde{r}_N^\infty+e^{-2i\beta_\ell H}}\,.
\]
Inserting the expression of $a(H)$ in (\ref{expansionuH}) and looking at the behaviour as $x\to-\infty$, we get 
\[
r_N^H=r_N^\mrm{asy}(H)\qquad\mbox{ with }\qquad r_N^\mrm{asy}(H):=r_N^\infty-\cfrac{(t_N^\infty)^2}{\tilde{r}_N^\infty+e^{-2i\beta_\ell H}}+\dots.
\]
More precisely, one can establish the estimate $|r_N^H-r_N^\mrm{asy}(H)|\le C\,e^{-c\,H}$ for some constants $c$, $C>0$ independent of $H$. Exploiting that the scattering matrix associated with the solutions (\ref{DefScaSol}) is unitary and working with the M\"obius transform as in \cite[\S3.2]{ChNP18}, one can show that $H\mapsto r_N^\mrm{asy}(H)$ runs periodically on the unit circle as $H\to+\infty$. Since on the other hand one can prove that $H\mapsto r_N^H$ runs continuously on the unit circle\footnote{Observe that this is not directly a consequence of Proposition \ref{PropositionCounterclock} because $\partial\Om^H_{1/2}$ has some corner points. However this can be shown by using the tools presented in \cite[Chap.\,2,\,4,\,5]{MaNP00}.}, we deduce that for certain $H$, $r_N^H$ passes (exactly) through the point of affix $-1$. \\
\newline 
\noindent We summarize the above results in the following statement. 
\begin{proposition}\label{PropositionCaseSym}
Assume that $\Om^\infty_{1/2}$ is such that $r_D^\infty\ne-1$, $t_N^\infty\ne0$. Then there exists a sequence $(H_n)_{n}$ satisfying
\[
\lim_{n\to+\infty}H_n=+\infty\qquad\mbox{ and }\qquad\lim_{n\to+\infty}H_{n+1}-H_{n}=\cfrac{\pi}{\beta_\ell}=\cfrac{1}{\sqrt{1-1/\ell^2}}
\]
(the sequence is unbounded and almost periodic) such that the scattering matrix $\mathbb{S}^{H_n}$ in the geometry $\Om^{H_n}$ has an eigenvalue equal to $-1$.
\end{proposition}
\begin{remark}
Note that by adapting the above procedure, we could consider cases where $\ell>2$. It would also provide geometries where $-1$ is an eigenvalue of $\mathbb{S}^{H}$.
\end{remark}

\section{Gaps in the particular case of classical Kirchhoff and anti-Kirchhoff transmission conditions}\label{SectionGapsUsualTR}

In this section, we wish to come back to the possible situation where $\Om$ is such $\dim\,\mX_\dagger=1$ and where we must impose classical Kirchhoff or anti-Kirchhoff transmission conditions in the model problem (\ref{farfieldBis}) of order $\eps^0$. With the notation of (\ref{defTheta}), this happens when $\theta=\pi/4$ or $\theta=3\pi/4$. Note that showing the existence of such geometries is not obvious but might be possible. As explained after (\ref{SpectralBandsBis}), in this case, after a shift of $\pi^2/\eps^2$, the bands of the model problem cover the interval $[\pi^2/\eps^2;+\infty)$. Our goal is to show that in general however, there are (short) gaps in the spectrum of the exact operator $A^\eps$ above $\pi^2/\eps^2$. In other words, we want to prove that in general lungs, that is, spectral bands, do not touch in the breathing process when perturbing the near field geometry around such a $\Om$. To establish this result, we need to construct a refined model involving terms of order $\eps$. This will be done by adapting some techniques presented in \cite{Naza10,BoPa13} (see also \cite{Naza24} for an application to a problem involving the Neumann Laplacian).\\
\newline
To present the analysis, we first introduce a few particular functions defined in $\Om$. To make things simple, we assume that $\Om$ is symmetric with respect to the $(Oy)$ axis, i.e. such that $\Om=\{(-x,y)\,|\,(x,y)\in\Om\}$, and that there holds $\mX_{\mrm{tr}}=\{0\}$. First we focus our attention on the case $\dim\,\mX_\dagger=1$ and $\theta=\pi/4$ (Kirchhoff transmission conditions). By definition of these two parameters, then the vector $(1,1)^{\top}$ is an eigenvector of the threshold scattering matrix $\mathbb{S}$ associated with the eigenvalue $-1$. Exploiting this property, we find that the functions $v_+$, $v_-$ introduced in (\ref{ScaSol}) satisfy
\[
\begin{array}{rcl}
v_++v_- &=& \psi_-(w^{\mrm{in}}+(s_{+-}+s_{--})w^{\mrm{out}})+\psi_+(w^{\mrm{in}}+(s_{++}+s_{-+})w^{\mrm{out}})+\tilde{v}_++\tilde{v}_- \\[3pt]
&=& \psi_-(w^{\mrm{in}}-w^{\mrm{out}})+\psi_+(w^{\mrm{in}}-w^{\mrm{out}})+\tilde{v}_++\tilde{v}_-\\[3pt]
&=& \psi_-\,w_0+\psi_+\,w_0+\tilde{v}_++\tilde{v}_-.
\end{array}
\]
This shows that $\mX_{\mrm{bo}}$ contains a function admitting the expansion
\[
U_0=\sum_{\pm}\psi_\pm w_0+\tilde{U}_0
\]
with $\tilde{U}_0\in\mH^1_0(\Om)$. Observing that $(x,y)\mapsto U_0(x,y)-U_0(-x,y)$ belongs to $\mX_{\mrm{tr}}$, since $\mX_{\mrm{tr}}=\{0\}$, we infer that $U_0$ is even with respect to $(Oy)$. Besides, the general results presented for example in \cite[Chap.\,5]{NaPl94} guarantee that the space 
$\mX$ introduced in (\ref{DefSpaceX}) of solutions of (\ref{NearFieldPb}) which are at most affine at infinity is of dimension two. Let 
\[
U=\sum_{\pm}\psi_\pm (C^1_\pm w_1+C^0_\pm w_0)+\tilde{U},
\]
with $C^1_\pm$, $C^0_\pm\in\Cplx$ and $\tilde{U}\in\mH^1_0(\Om)$, be an element of this space. Integrating by parts in the identity
\[
0=\int_{\Om^R}(\Delta U+\pi^2U)U_0-U(\Delta U_0+\pi^2U_0)\,dz,
\]
where $\Om^R=\{(x,y)\in\Om\,|\,|x|<R\}$, and taking the limit $R\to+\infty$, we get $C^1_+=-C^1_-$. Since $\mX_{\mrm{bo}}\subset\mX$ is of dimension one, we deduce that $\mX$ contains exactly one function with the expansion
\begin{equation}\label{DefU1}
U_1=\sum_{\pm}\pm\psi_\pm (w_1+m_{\Om}w_0)+\tilde{U}_1
\end{equation}
where $m_{\Om}\in\R$ and $\tilde{U}_1\in\mH^1_0(\Om)$. Classically $m_{\Om}$ is called the polarization coefficient. Observing that $(x,y)\mapsto U_1(x,y)+U_1(-x,y)$ belongs to $\mX_{\mrm{tr}}$, since $\mX_{\mrm{tr}}=\{0\}$, we infer that $U_1$ is odd with respect to $(Oy)$. Finally, simple manipulations allow one to show that there is a function $U_2$ satisfying
\[
\begin{array}{|rcll}
-\Delta U_2-\pi^2U_2&=&U_0&\mbox{ in }\Om\\[3pt]
U_2&=&0&\mbox{ on }\partial\Om
\end{array}
\]
and admitting the expansion 
\begin{equation}\label{DefU2}
U_2=\sum_{\pm}\psi_\pm (-\frac{w^2_1}{2}+M_\Om\,w_1+C_{\Om}\,w_0)+\tilde{U}_2
\end{equation}
with $M_\Om$, $C_{\Om}\in\R$ and $\tilde{U}_2\in\mH^1_0(\Om)$. Note that $(x,y)\mapsto U_2(x,y)-U_2(-x,y)$ belongs to $\mX_{\mrm{tr}}$ because $U_0$ is even with respect to the $(Oy)$ axis. Therefore $U_2$ is also even with respect to the $(Oy)$ axis.\\
\newline
Now pick $\eta\in[0;2\pi)$ and introduce $u^\eps(\cdot,\eta)$ an eigenfunction of (\ref{PbSpectralCell}) associated with some eigenvalue $\Lambda^\eps(\eta)$. In the sequel, to simplify, we do not indicate the dependence on $\eta$. As a refined approximation of (\ref{ExpansionTR0_2})
when $\eps\to0^+$, we consider the expansion
\begin{equation}\label{ExpansionTR0_2_bis}
\Lambda^\eps=\eps^{-2}\pi^2+\nu+\eps\Lambda'+\dots,\quad u^\eps(z)=(\gamma(x)+\eps\gamma'(x)+\eps^2(\gamma''(x)+V(x,y/\eps)))\,\varphi(y/\eps)+\dots.
\end{equation}
From what has been obtained in \S\ref{ParaDim1}, we know that $(\nu,\gamma)$ must be an eigenpair of (\ref{farfieldBis}). Let us observe that for $\theta=\pi/4$, Problem (\ref{farfieldBis}) simplifies into
\begin{equation}\label{farfieldTer}
\begin{array}{|rclrcl}
\partial^2_x\gamma+\nu\gamma&=&0\quad\mbox{ in }I=(-1/2;1/2)\\[3pt]
\gamma(-1/2)&=&e^{i\eta}\gamma(+1/2) \\[3pt]
\partial_x\gamma(-1/2)&=&e^{i\eta}\partial_x\gamma(+1/2).
\end{array}
\end{equation}
Solving it, we get 
\begin{equation}\label{CalculExact}
\nu=(\eta+2\pi k)^2\mbox{ for some }\,k\in\Z\qquad\mbox{ and }\qquad\gamma(x)=e^{i(-\eta-2\pi k)x}.
\end{equation}
Next we have to identify the terms of order $\eps$ in (\ref{ExpansionTR0_2_bis}). First we find that we must have
\begin{equation}\label{farfieldQua}
\begin{array}{|rclrcl}
\partial^2_x\gamma'+\nu\gamma'&=&-\Lambda'\gamma\quad\mbox{ in }I_-\cup I_+\\[3pt]
\gamma'(-1/2)&=&e^{i\eta}\gamma'(+1/2) \\[3pt]
\partial_x\gamma'(-1/2)&=&e^{i\eta}\partial_x\gamma'(+1/2).
\end{array}
\end{equation}
We still have to impose transmission conditions at the origin. Let us find them. In a neighbourhood of $O$, we look for an expansion of $u^\eps$ of the form 
\begin{equation}\label{NearFieldBis}
u^\eps(z)=U_0(z/\eps)+\eps(a_1U_1(z/\eps))+\eps^2(a_2U_2(z/\eps))+\dots
\end{equation}
where $a_1$, $a_2$ are some constants to determine. As $x\to0^\pm$, we have the Taylor expansions 
\[
\gamma(x)=1+x\partial_x\gamma(0)+\cfrac{x^2}{2}\,\partial^2_x\gamma(0)+O(x^3),\qquad \gamma'(x)=\gamma'(0^\pm)+x\partial_x\gamma'(0^\pm)+O(x^2)
\]
so that we can write, as $x\to0^\pm$,
\[
\gamma(x)+\eps\gamma'(x)+\eps^2\gamma''(x)=1+\eps\,(\gamma'(0^\pm)+\cfrac{x}{\eps}\,\partial_x\gamma(0))+\eps^2\,(\gamma''(0^\pm)+\cfrac{x}{\eps}\,\partial_x\gamma'(0^\pm)+\cfrac{x^2}{\eps^2}\,\partial^2_x\gamma(0))+\dots.
\]
By matching this far field expansion with the inner field expansion (\ref{NearFieldBis}), from the definition of $U_1$ and $U_2$ (see (\ref{DefU1}), (\ref{DefU2})), we see that we must take $a_1=\partial_x\gamma(0)$ and $a_2=-\partial^2_x\gamma(0)$ in (\ref{NearFieldBis}). Using in particular that $-\partial^2_x\gamma(0)=\nu\gamma(0)$, we deduce that we must complement (\ref{farfieldQua}) with the conditions
\begin{equation}\label{TRModelEps}
\begin{array}{|rcl}
\gamma'(0^+)-\gamma'(0^-)&=&2m_{\Om}\partial_x\gamma(0)\\[3pt]
\partial_x\gamma'(0^+)-\partial_x\gamma'(0^-)&=&-2M_{\Om}\partial^2_x\gamma(0)=2\nu M_{\Om}\gamma(0).
\end{array}
\end{equation}
If $\nu=\nu(\eta)$ is a simple eigenvalue of (\ref{farfieldTer}), which happens if and only if $\eta\notin\{0,\pi\}$ then $\Lambda'$ in (\ref{farfieldQua}) is computed through the compatibility condition. More precisely, multiplying (\ref{farfieldQua}) by $\gamma$, integrating twice by parts and exploiting (\ref{TRModelEps}), we obtain 
\begin{equation}\label{MainTermDeriv}
\Lambda'=2 m_{\Om}|\partial\gamma(0)|^2-2\nu M_{\Om}|\gamma(0)|^2=2\nu(m_{\Om}-M_{\Om}).
\end{equation}
Thus we see that by taking into account the terms of order $\eps$, compared  to the model of order $\eps^0$, the dispersion curves move up or down depending on the value of $m_{\Om}-M_{\Om}$ (note that this quantity depends only on the inner field geometry $\Om$).\\

\begin{figure}[!ht]
\centering
\includegraphics[width=16cm,trim={3cm 20.8cm 3cm 2.4cm},clip]{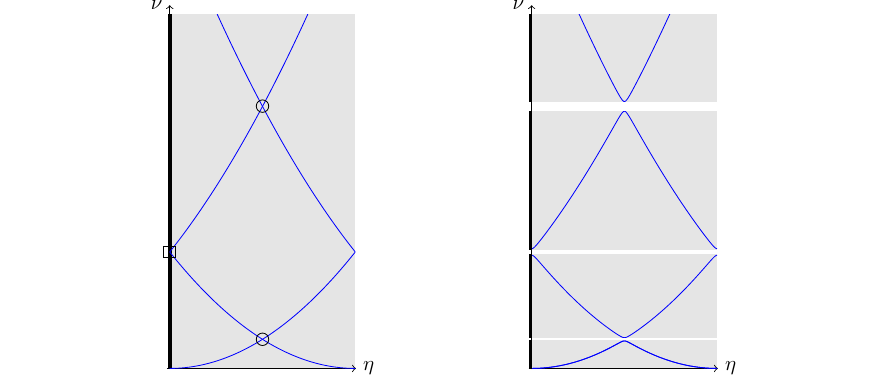}
\caption{Schematic picture of the dispersion curves of the models of orders $\eps^0$ (left) and $\eps$ (right). On the right picture, the projection of the white rectangles on the vertical axis corresponds to the gaps of the model of order $\eps$.\label{ImageDispersionCurvesModel}}
\end{figure}

\noindent Now we turn our attention to the case where $\nu=\nu(\eta)$ is a double eigenvalue of (\ref{farfieldTer}). To set ideas, we assume that $\eta=\pi$ so that according to (\ref{CalculExact}), there holds $\nu(\eta)=(2k+1)^2\pi^2$ for some $q\in\N$ (see the circles in Figure \ref{ImageDispersionCurvesModel} left). Our study above (see in particular formula (\ref{MainTermDeriv})) does not allow us to understand precisely enough the behaviour of the exact dispersion curves at the point $(\pi,\nu(\pi))$. Actually here we expect that they present a rapid variation. To capture it, following \cite{Naza10}\footnote{For another approach, one may consult the proof of \cite[Theorem 2]{LeWZ19} where the authors follows a Lyapunov-Schmidt reduction strategy.}, let us impose some dependence of the Floquet parameter with respect to $\eps$. More precisely, for $t\in\R$, let us work with $\eta=\eta^\eps=\pi+\eps t$. Denote by $\Lambda^\eps_\pm(\eta^\eps)$, with $\Lambda^\eps_-(\eta^\eps)\le\Lambda^\eps_+(\eta^\eps)$, the two eigenvalues of (\ref{PbSpectralCell}) closest to $\Lambda^0:=(2k+1)^2\pi^2$ and $u^\eps_\pm(z,\eta^\eps)$ corresponding eigenfunctions. Our goal is to obtain an expansion of $\Lambda^\eps_\pm(\eta^\eps)$, $u^\eps_\pm(z,\eta^\eps)$ as $\eps$ tends to zero (note that both the geometry and the Floquet parameter depend on $\eps$). We consider the ansatz
\[
\Lambda^\eps_\pm(\eta^\eps)=\Lambda^0+\eps\Lambda'_\pm(t)+\dots
\] 
\[
u^\eps_\pm(z,\eta^\eps)=(\gamma_\pm(x,t)+\eps\gamma'_\pm(x,t))\,\varphi(y/\eps)+\dots
\]
with $\gamma_\pm(x,t)=a_\pm\,e^{-i\sqrt{\Lambda^0} x}+b_\pm\,e^{i\sqrt{\Lambda^0} x}$, the constants $a_\pm$, $b_\pm$ being to determine. As in (\ref{farfieldQua}), (\ref{TRModelEps}), first we find that $\Lambda'_\pm(t)$, $\gamma'_\pm(x,t)$ must satisfy
\begin{equation}\label{farfieldCinq}
\begin{array}{|rclrcl}
\partial^2_x\gamma'_\pm+\Lambda^0\gamma'_\pm&=&-\Lambda'_\pm(t)\gamma_\pm\quad\mbox{ in }I_-\cup I_+\\[3pt]
\gamma'_\pm(0^+)-\gamma'_\pm(0^-)&=&2m_{\Om}\partial_x\gamma_\pm(0) \\[3pt]
\partial_x\gamma'_\pm(0^+)-\partial_x\gamma'_\pm(0^-)&=&2\Lambda^0 M_{\Om}\gamma_\pm(0).
\end{array}
\end{equation}
Concerning the quasi-periodic boundary conditions, exploiting that $\eta^\eps=\pi+\eps t$ and writing a Taylor expansion ($\eps t$ is small), identifying the terms of order $\eps$, we obtain
\begin{equation}\label{QuasiPeriot}
\begin{array}{|rcl}
\gamma'_\pm(1/2,t)+\gamma'_\pm(-1/2,t) &=& -it\gamma_\pm(1/2,t)\\[3pt]
\partial_x\gamma'_\pm(1/2,t)+\partial_x\gamma'_\pm(-1/2,t) &=& -it\partial_x\gamma_\pm(1/2,t).
\end{array}
\end{equation}
In order the problem (\ref{farfieldCinq})--(\ref{QuasiPeriot}) for $\gamma'_+$ to have a solution, two compatibility conditions must be satisfied. They are obtained by multiplying (\ref{farfieldCinq}) by $e^{\pm i\sqrt{\Lambda^0} x}$ and integrating by parts. At the end, we find that the quantities $\Lambda'_+(t)$, $\mathfrak{u}_+:=(a_+,b_+)^{\top}$ solve
\[
\mathscr{M}(t)\mathfrak{u}_+=\Lambda'_+(t)\mathfrak{u}_+\quad\mbox{ with }\quad \mathscr{M}(t):=2\left(\begin{array}{cc}
t\sqrt{\Lambda^0}+\Lambda^0(M_\Om-m_\Om) & \Lambda^0(M_\Om+m_\Om) \\
\Lambda^0(M_\Om+m_\Om) & -t\sqrt{\Lambda^0}+\Lambda^0(M_\Om-m_\Om)
\end{array}\right).
\]
Similarly, in order the problem (\ref{farfieldCinq})--(\ref{QuasiPeriot}) for $\gamma'_-$ to have a solution, we get that 
$\Lambda'_-(t)$, $\mathfrak{u}_-:=(a_-,b_-)^{\top}$ must satisfy
\[
\mathscr{M}(t)\mathfrak{u}_-=\Lambda'_-(t)\mathfrak{u}_-.
\]
Here, implicitly we have defined $\Lambda'_+(t)$ (resp. $\Lambda'_-(t)$) as the largest (resp. smallest) eigenvalue of $\mathscr{M}(t)$. A direct calculus yields
\[
\Lambda'_\pm(t)=2\Lambda^0(M_\Om-m_\Om)\pm 2\Lambda^0\sqrt{t^2/\Lambda^0+(M_\Om+m_\Om)^2}.
\]
Therefore we have
\[
\Lambda'_+(t)-\Lambda'_-(t)=4\Lambda^0\sqrt{t^2/\Lambda^0+(M_\Om+m_\Om)^2}.
\]
Thus if $M_\Om+m_\Om\ne0$, we see that there holds $\Lambda'_+(t)-\Lambda'_-(t)>0$ for all $t\in\R$. We infer that the two dispersion curves of the model of order $\eps^0$ which form a ``cross''  at the point $(\pi,\Lambda^0)$  (see the circles on Figure \ref{ImageDispersionCurvesModel} left) move away when considering the model of order $\eps$. As a consequence a gap of width $O(\eps)$ opens as in Figure  \ref{ImageDispersionCurvesModel} right. We summarize this result in the following statement. 
\begin{proposition}\label{PropositionGaps}
Assume that the domain $\Om$ is symmetric with respect to the $(Oy)$ axis. Assume also that $\dim\,\mX_\dagger=1$ and $\theta=\pi/4$ in (\ref{defTheta}). If additionally the constants 
$m_\Om$, $M_\Om$ introduced in (\ref{DefU1}), (\ref{DefU2}) are such that $M_\Om+m_\Om\ne0$, then for $\eps$ small enough the operator $A^\eps$ has some gaps above $\pi^2/\eps^2$.
\end{proposition}

\noindent If $M_\Om+m_\Om=0$, we can not conclude from this study and it is necessary to compute higher order terms in the asymptotics. Roughly speaking, there is no gap above $\pi^2/\eps^2$ in the spectrum of the exact operator $A^\eps$
if for any order of the model, the corresponding dispersion curves cross. This happens when $\Om$ is the straight strip (see the discussion of Remark \ref{RmkRefStrip}) but seems very rare.\\
\newline
Above we focused our attention on the case of classical Kirchhoff transmission conditions at the origin ($\theta=\pi/4$ in (\ref{farfieldBis})). The situation $\theta=3\pi/4$, which corresponds to anti-Kirchhoff transmission conditions, is also interesting. Then  (\ref{farfieldBis}) writes
\[
\begin{array}{|rclrcl}
\partial^2_x\gamma^\pm+\nu\gamma^\pm&=&0\quad\mbox{ in }I_\pm\\[3pt]
\gamma^-(-1/2)&=&e^{i\eta}\gamma^+(+1/2) & \gamma^+(0)&=&-\gamma^-(0)\\[3pt]
\partial_x\gamma^-(-1/2)&=&e^{i\eta}\partial_x\gamma^+(+1/2) & \qquad\partial_x\gamma^+(0)&=&-\partial_x\gamma^-(0).
\end{array}
\]
Solving this problem, we get $\nu=(\eta+\pi(2k+1))^2$ for some $k\in\Z$ with $\gamma_\pm(x)=\pm e^{i(-\eta-\pi(2k+1))x}$. As a consequence, the spectral bands of the model of order $\eps^0$ cover the interval $[\pi^2/\eps^2;+\infty)$ (after the usual shift by $\pi^2/\eps^2$). However again one can show that in general there are some gaps in the spectrum of the exact operator $A^\eps$. To proceed, one works as above when $\theta=\pi/4$. The first step consists in showing the existence of functions $V_0$, $V_1$ solving (\ref{NearFieldPb}) and admitting the expansions 
\[
V_0=\sum_{\pm}\pm\psi_\pm w_0+\tilde{V}_0\qquad V_1=\sum_{\pm}\psi_\pm (w_1+m_{\Om}w_0)+\tilde{V}_1
\]
with a new $m_{\Om}\in\R$ and $\tilde{V}_0$, $\tilde{V}_1\in\mH^1_0(\Om)$. Then we establish that there is $V_2$ satisfying
\[
\begin{array}{|rcll}
-\Delta V_2-\pi^2V_2&=&V_0&\mbox{ in }\Om\\[3pt]
V_2&=&0&\mbox{ on }\partial\Om
\end{array}
\]
which decomposes as
\[
V_2=\sum_{\pm}\pm\psi_\pm (-\frac{w^2_1}{2}+M_\Om\,w_1+C_{\Om}\,w_0)+\tilde{V}_2,
\]
with some new $M_{\Om}$, $C_\Om\in\R$ and $\tilde{V}_2\in\mH^1_0(\Om)$. With these three functions, then we can adapt the analysis which led to Proposition \ref{PropositionGaps}.

\section{Numerics}\label{SectionNumerics}

In this section, we give numerical illustrations of some of the results above. First, we use what has been done in Section \ref{SectionNFgeom} to provide  examples of geometries where $\mX_\dagger\ne\{0\}$. We work in the waveguide
\begin{equation}\label{DefGeomH}
\Om^H=\mathcal{S}\cup \mathcal{R}^H\qquad\mbox{ with }\qquad\mathcal{S}=\R\times(-1/2;1/2)\qquad\mbox{and}\qquad\mathcal{R}^H=(-\ell/2;\ell/2)\times(0;H).
\end{equation}
Numerically, we approximate the functions $v_\pm^H$ introduced in (\ref{ScaSolH}). To proceed, we truncate the domain $\Om^H$ at $x=\pm L$ with $L=2$ and impose the complex Robin conditions
\begin{equation}\label{RobinConditions}
\pm \cfrac{\partial v}{\partial x}=\cfrac{1}{L-i}\,v\quad\mbox{ at }x=\pm L.
\end{equation}
We emphasize that (\ref{RobinConditions}) is an approximated transparent condition. Indeed the function $w^{\mrm{out}}$ defined in (\ref{defModeswinout}) satisfies it exactly but the $v_\pm^H$ only up to an error which decays when $L$ increases. For more details concerning this approach, we refer the reader to \cite[\S6.2]{ChNa23}. Then we work with a standard P2 finite element method by using the library \texttt{FreeFem++} \cite{Hech12}. This allows us to get an approximation of the matrix $\mathbb{S}^H$ introduced in (\ref{defMatScaSym}).\\
\newline
In Figure \ref{TR_matrix_2D} left, we take $\ell=1.6$ and make $H$ vary in $[1.5;6]$. Since $\mathbb{S}^H$ is unitary, its two eigenvalues are located on the unit circle. Therefore, we only display the phase of the two eigenvalues. In accordance with what has been obtained in Section \ref{SectionNFgeom} and in particular the result of Proposition \ref{PropositionCaseSym}, we observe that one eigenvalue tends to a constant on the unit circle while the other rotates continuously and counter-clockwise (which is coherent with Proposition \ref{PropositionCounterclock}) when $H$ increases. Additionally, we indeed find that one eigenvalue passes through $-1$ almost periodically, the period being equal to 
\begin{equation}\label{Period}
\cfrac{\pi}{\beta_\ell}=\cfrac{1}{\sqrt{1-1/\ell^2}}\approx 1.281.
\end{equation}
Numerically, we obtain that $-1$ is an eigenvalue of $\mathbb{S}^H$ for 
\begin{equation}\label{defHstar}
H\in\{1.764,3.047,4.329,5.612\}.
\end{equation}
Note that the period (\ref{Period}) is well respected. In Figure \ref{TR_matrix_2D} right, we display the same quantities but this time we work in the geometry $\Om^H$ defined in (\ref{DefGeomH}) with  $\ell=1.9$. As expected, we observe that the rotating eigenvalue passes more frequently through $-1$.

\begin{figure}[!ht]
\centering
\includegraphics[width=0.47\textwidth]{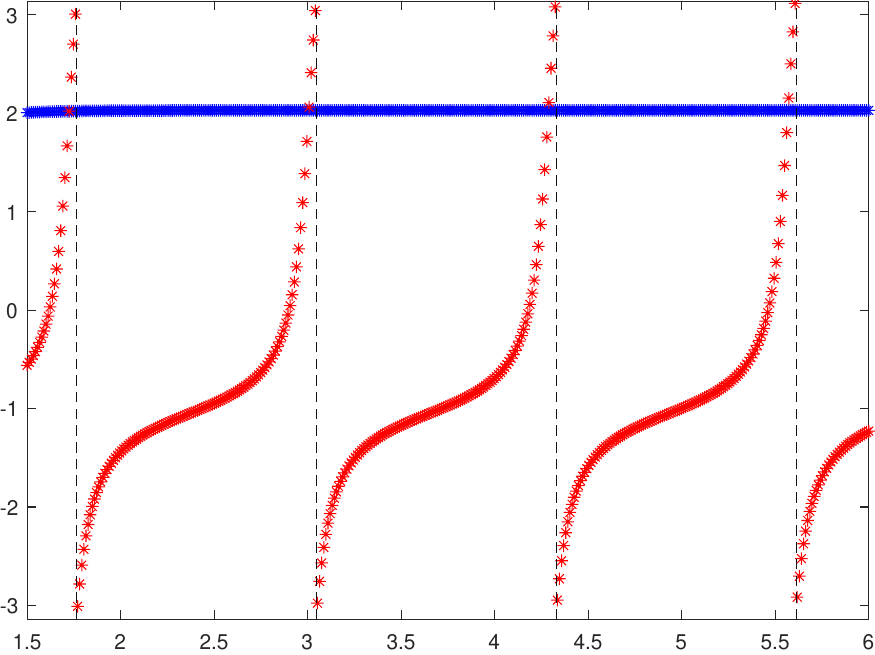}\quad\includegraphics[width=0.47\textwidth]{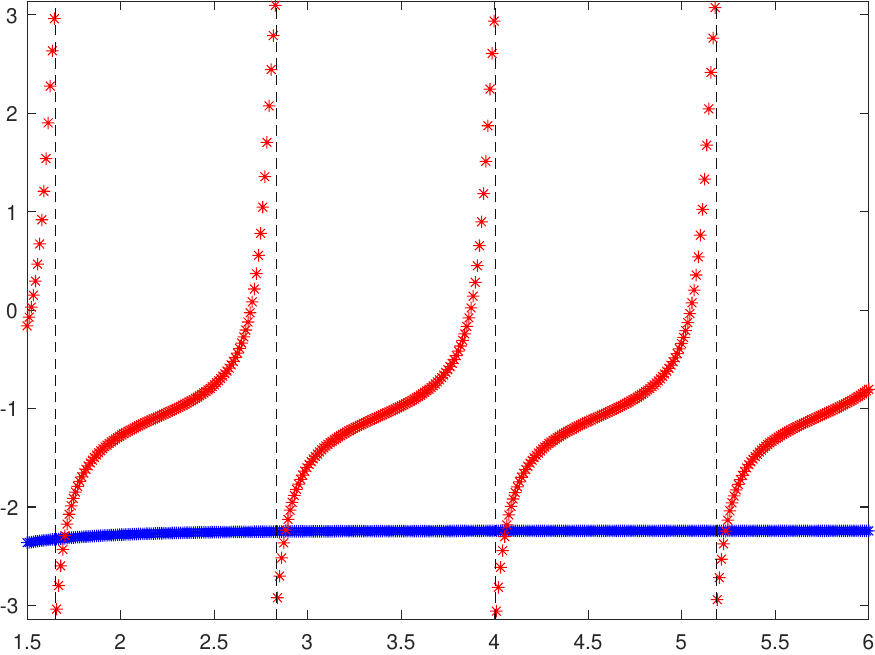}
\caption{Phase of the two eigenvalues of the threshold scattering matrix $\mathbb{S}^H$ defined in (\ref{defMatScaSym}) with respect to $H\in[1.5;6]$. The vertical dashed lines correspond to the values of $H$ for which the phase of one eigenvalue is equal to $\pi$ (then the corresponding eigenvalue is equal to $-1$). Left and right results are obtained respectively for $\ell=1.6$ and $\ell=1.9$. }
\label{TR_matrix_2D}
\end{figure}

\noindent In the next series of numerical experiments, we compute the spectrum of Problem (\ref{PbSpectralCell}) set in the thin periodicity cell $\om^\eps$. The $\om^\eps$ we consider is constructed from $\Om^H$ as in (\ref{DefGeomH}) with $\ell=1.6$ and for different values of $H$. We remind the reader that the way $\om^\eps$ is defined from $\Om^H$ appears in (\ref{defCell}). To solve  (\ref{PbSpectralCell}), classically we rewrite it as a problem with periodic boundary conditions at $\partial\om^\eps_\pm$. Again we work with a P2 finite element method. The matrices are constructed with \texttt{FreeFem++} while the resolution of the eigenvalue problem is made with \texttt{Matlab}\footnote{\url{https://mathworks.com}}. Since the spectrum of (\ref{PbSpectralCell}) is $2\pi$-periodic in $\eta$, we compute it for $\eta\in[0;2\pi]$.\\
\newline
In Figure \ref{Bands_eps0pt1}, we represent the first six eigenvalues of (\ref{PbSpectralCell}) with respect to $\eta\in[0;2\pi]$ for $H=2.5$ and $\eps=0.1$. In other words, we display the dispersion curves $\eta\mapsto \Lambda^\eps_p(\eta)$, for $p=1,\dots,6$. The red dashed lines indicate the value of the normalized threshold $\pi^2/\eps^2$. In Figure \ref{Bands_eps0pt02}, we compute the same quantities but this time with $\eps=0.02$. In accordance with the result of Theorem \ref{MainThmPerio}, we observe that the dispersion curves below $\pi^2/\eps^2$ (in magenta) are extremely flat, generating extreminly short spectral bands $\Upsilon^\eps_p$. Concerning, the bands located above $\pi^2/\eps^2$ (in blue), also in agreement with Theorem \ref{MainThmPerio}, we note that they get shorter when $\eps$ gets smaller. As a consequence, the size of the gaps between the bands increases as $\eps$ tends to zero. If one looks closely at the size of the gap between $\Upsilon^\eps_{N_\bullet+m}$ and $\Upsilon^\eps_{N_\bullet+m+1}$ (here $N_\bullet=2$), one finds a value which corroborates the $(2m+1)\pi^2$ predicted by the theory (see the discussion after Theorem \ref{MainThmPerio}).

\begin{figure}[!ht]
\centering
\includegraphics[width=0.46\textwidth]{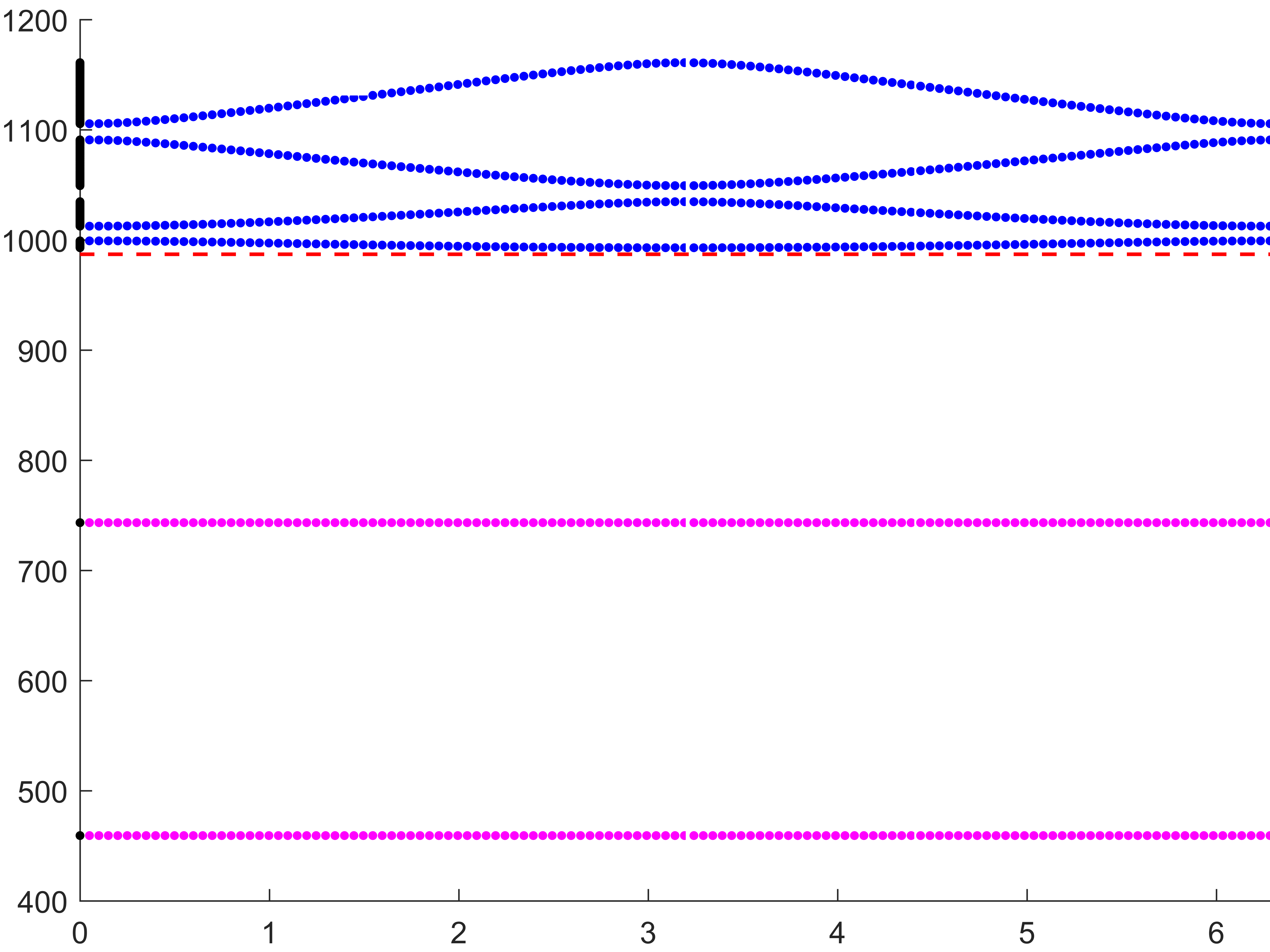}\quad\includegraphics[width=0.46\textwidth]{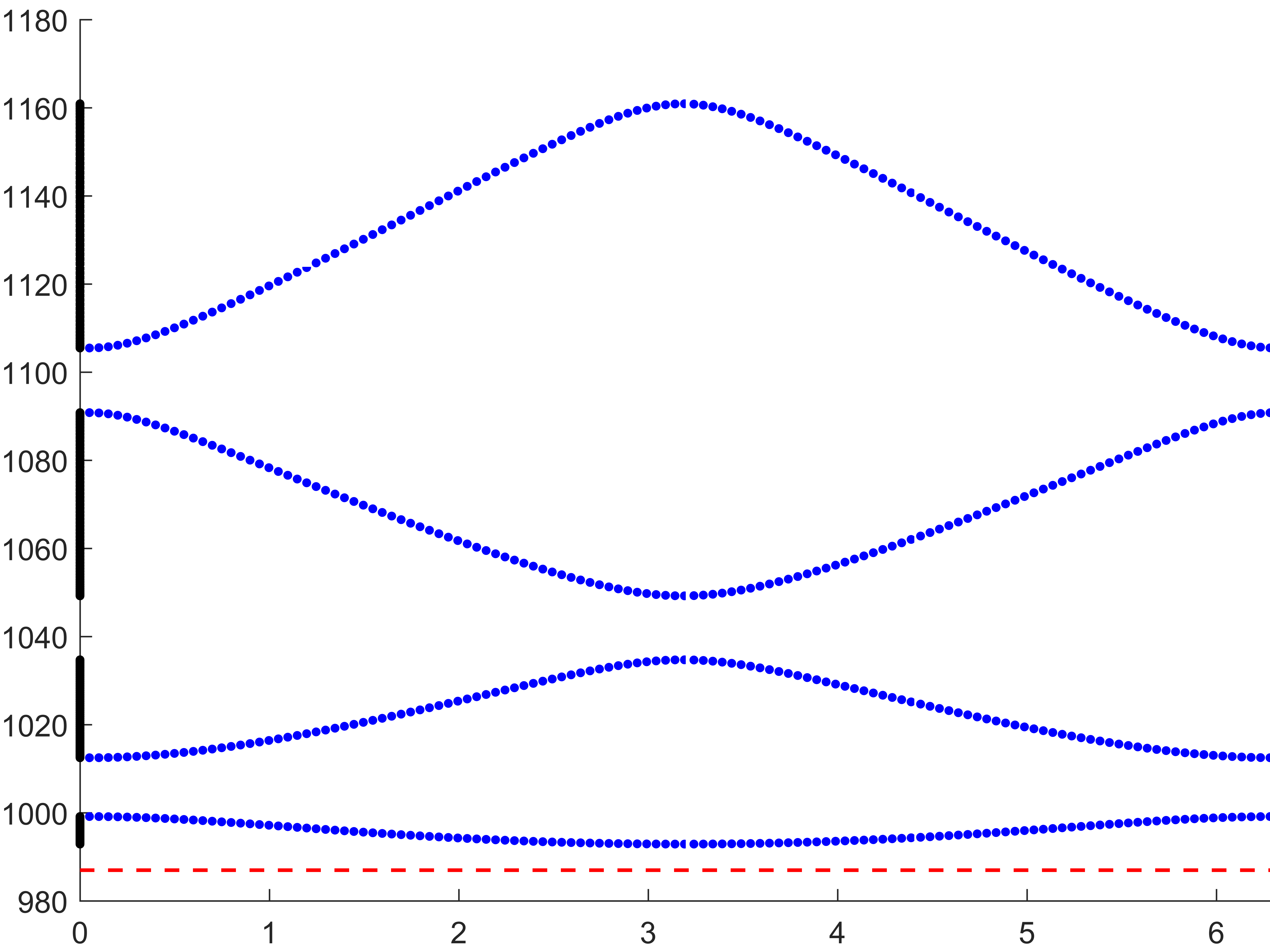}
\caption{Left: first six eigenvalues of (\ref{PbSpectralCell}) with respect to $\eta\in[0;2\pi]$. Right: zoom on the first eigenvalues larger than $\pi^2/\eps^2$. The red dashed lines indicate the value of the normalized threshold $\pi^2/\eps^2$. Here $H=2.5$ and $\eps=0.1$.}
\label{Bands_eps0pt1}
\end{figure}

\begin{figure}[!ht]
\centering
\includegraphics[width=0.46\textwidth]{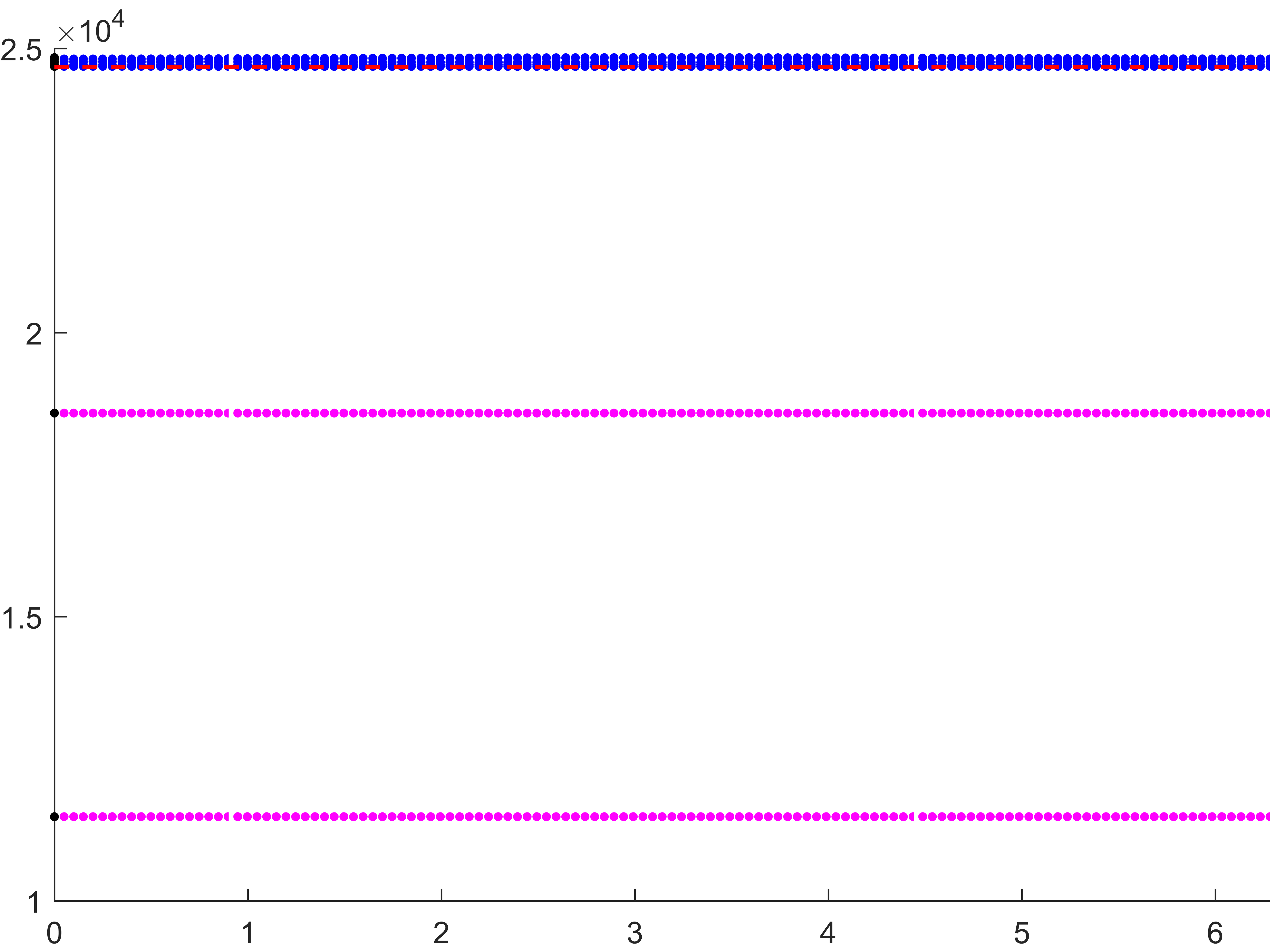}\quad\includegraphics[width=0.46\textwidth]{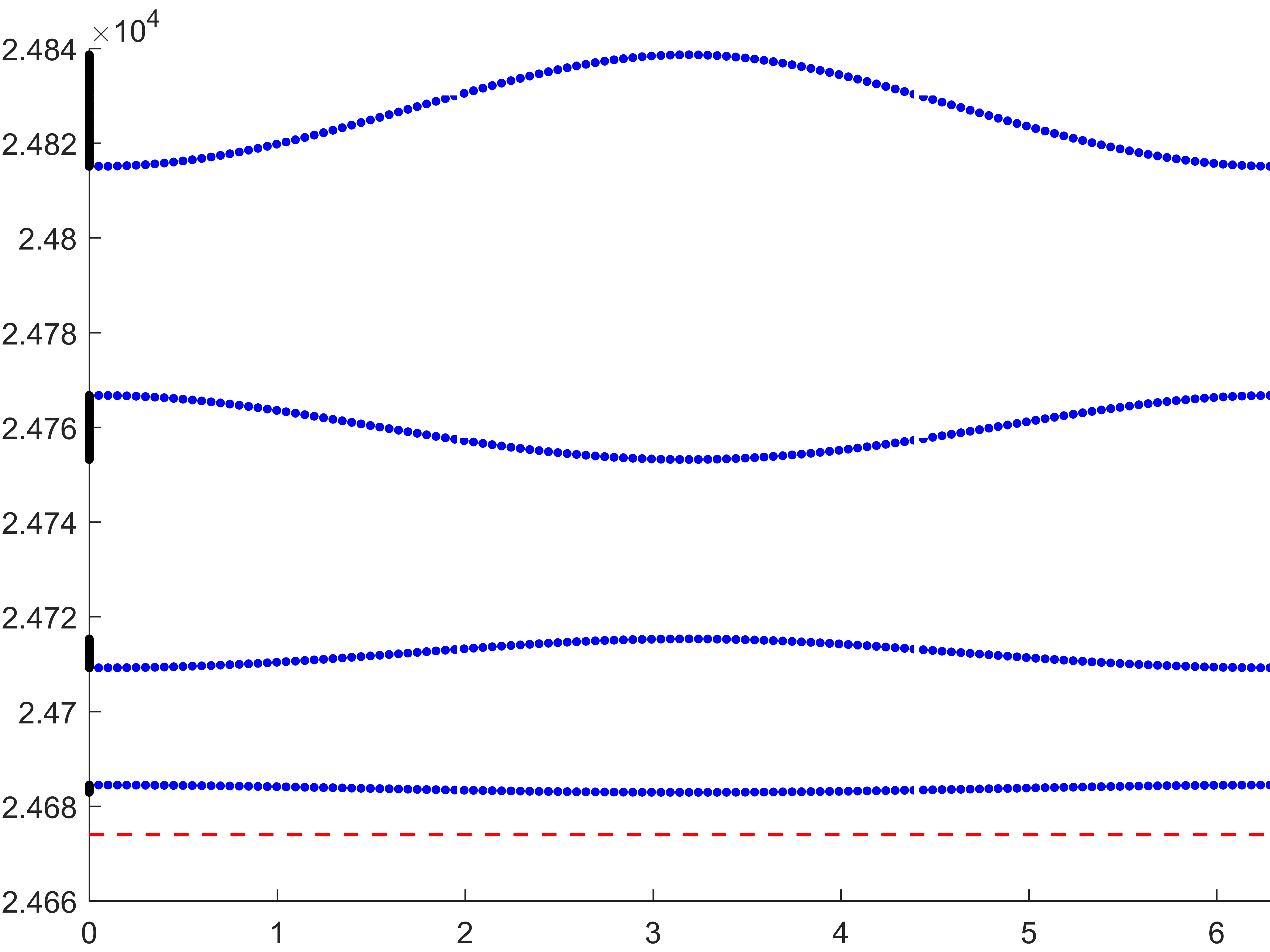}
\caption{Same quantities as in Figure \ref{Bands_eps0pt1} but with $\eps=0.02$.}
\label{Bands_eps0pt02}
\end{figure}

\noindent In Figure \ref{Bands_eps0pt02_H3}, we represent the same quantities as in Figure \ref{Bands_eps0pt02} but this time with $H=3.05$, i.e. with some $H$ slightly larger than one of the particular $H_\star$ appearing in (\ref{defHstar}) such that $\mathbb{S}^{H_\star}$ has an eigenvalue equal to $-1$. In agreement with the result of Theorem \ref{MainThmPerioMigration}, we note that the spectral bands above $\pi^2/\eps^2$ are larger than in the case $H=2.5$. Moreover, we observe that one thin spectral band (in magenta) has appeared below $\pi^2/\eps^2$. By further increasing $H$, this spectral band will dive below $\pi^2/\eps^2$, stops breathing and so dies to become extremely thin.

\begin{figure}[!ht]
\centering
\includegraphics[width=0.46\textwidth]{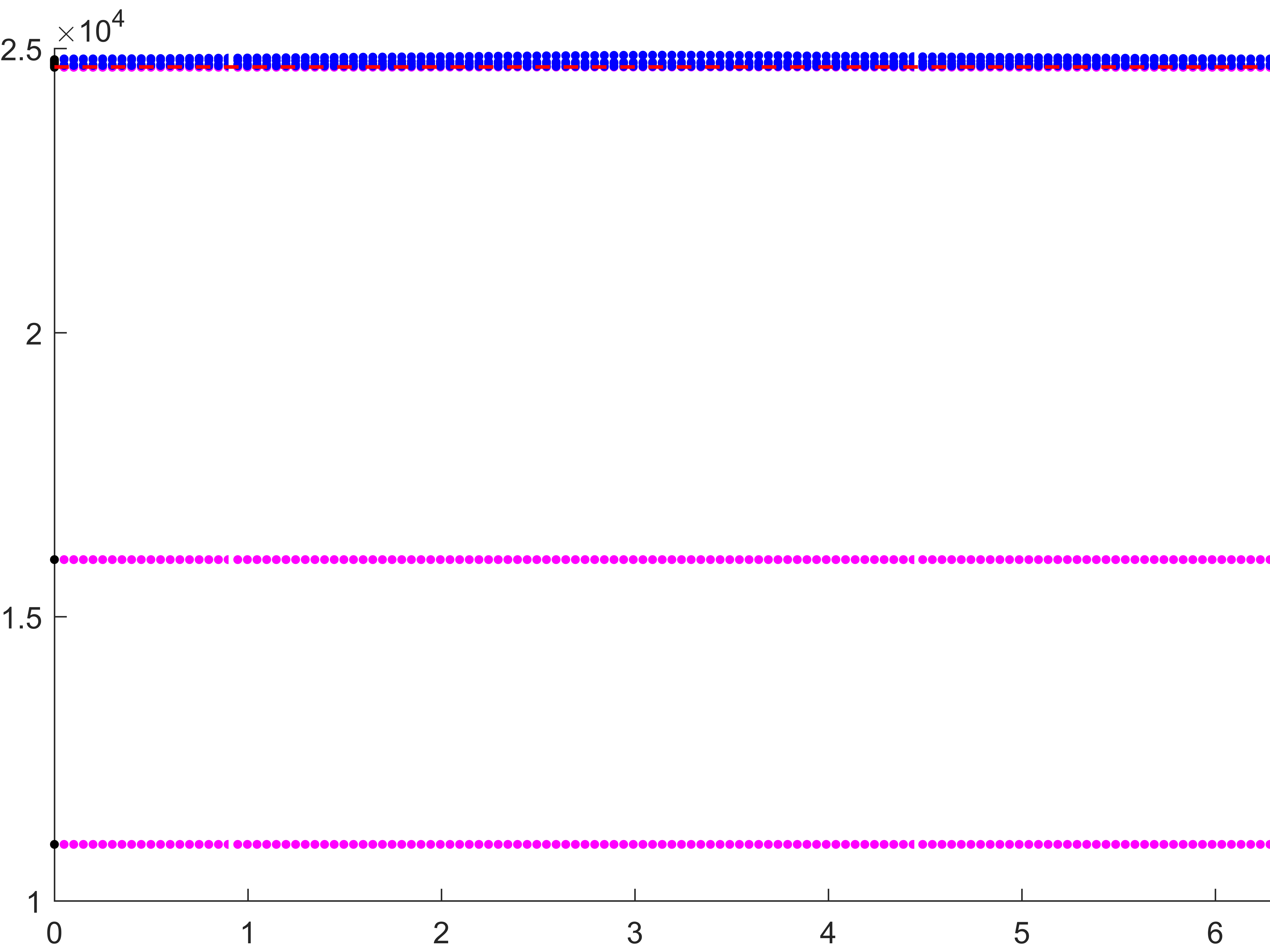}\quad\includegraphics[width=0.46\textwidth]{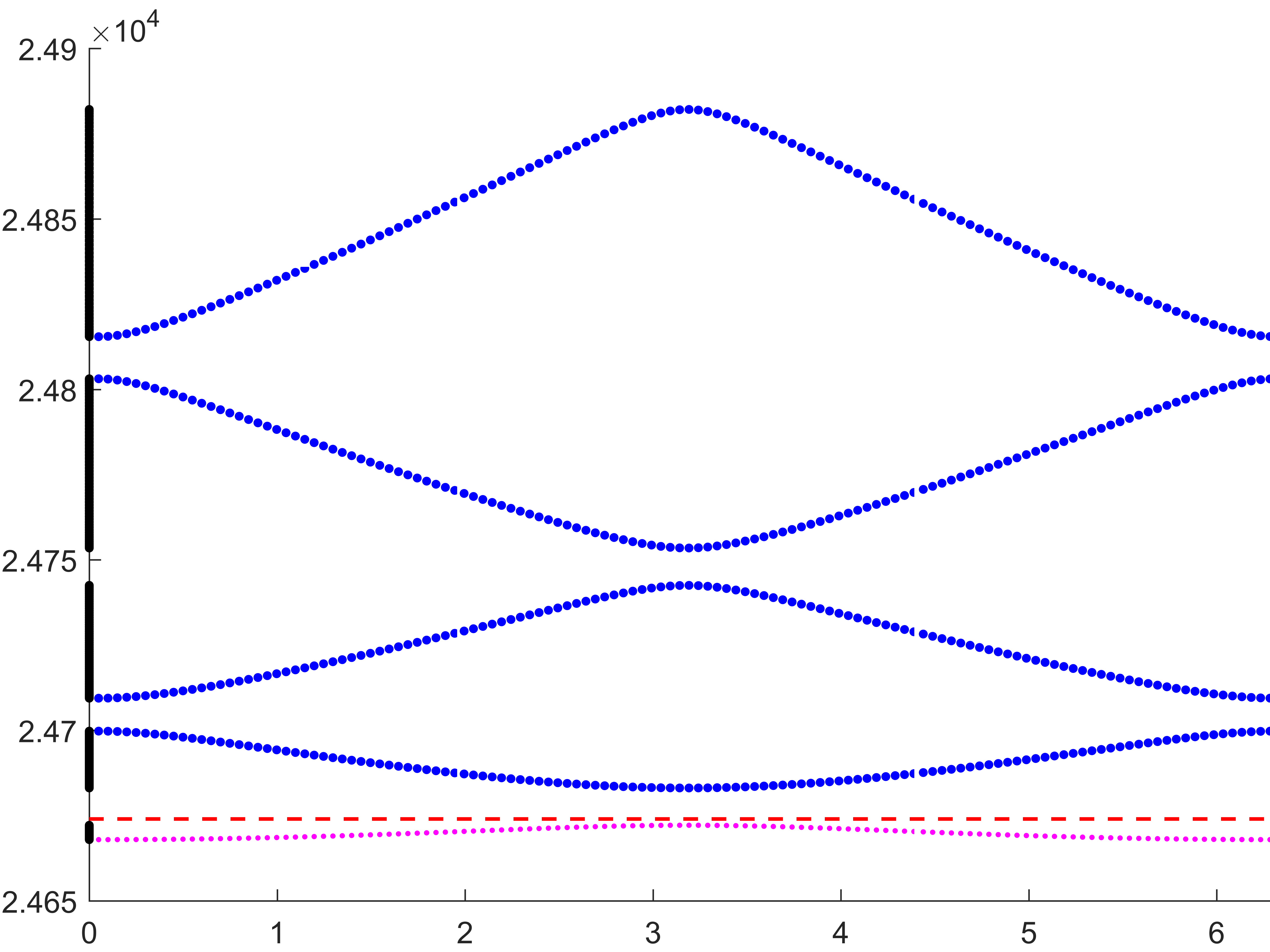}
\caption{Same quantities as in Figure \ref{Bands_eps0pt02} but with $H=3.05$.}
\label{Bands_eps0pt02_H3}
\end{figure}

\noindent Finally in Figures \ref{SpectrumH}, \ref{SpectrumHZoom}, we display the spectrum of $A^\eps$ with respect to $H\in[1.5;3.5]$. Practically, we select fifty values of $H\in[1.5;3.5]$ and for each of them, compute numerically the spectrum of (\ref{PbSpectralCell}) with respect to $\eta\in[0;2\pi]$. This gives us the spectral bands $\Upsilon^\eps_p$ introduced in (\ref{SpectralBands}) that we display on the vertical line corresponding to $H$. On these figures, the vertical dashed lines marks the $H_\star$ appearing in (\ref{defHstar}) such that $\mathbb{S}^{H_\star}$ has an eigenvalue equal to $-1$. In Figure \ref{SpectrumH}, we clearly see the extremely thin spectral bands (the ones in magenta) appearing at the $H_\star$. Figure \ref{SpectrumHZoom} on the other hand illustrates the phenomenon of breathing of the spectrum when inflating the near field geometry: when increasing $H$ around $H_\star$, the spectral bands above $\pi^2/\eps^2$ expand and then shrink. Note that the results we get are in agreement with the model obtained in Section \ref{SectionModel} (see in particular the numerics of Figure \ref{ImageAsymptoRho}). We observe some small differences between the results of Figures \ref{ImageAsymptoRho} and \ref{SpectrumHZoom} around the $H_\star$ however. We do not really know if they are due to the fact that $\eps$ is not small enough (numerically, it is difficult to work with tiny $\eps$ because computations become heavy) or are the consequences of numerical errors (we already work at relatively high frequency).

\begin{figure}[!ht]
\centering
\includegraphics[width=0.89\textwidth]{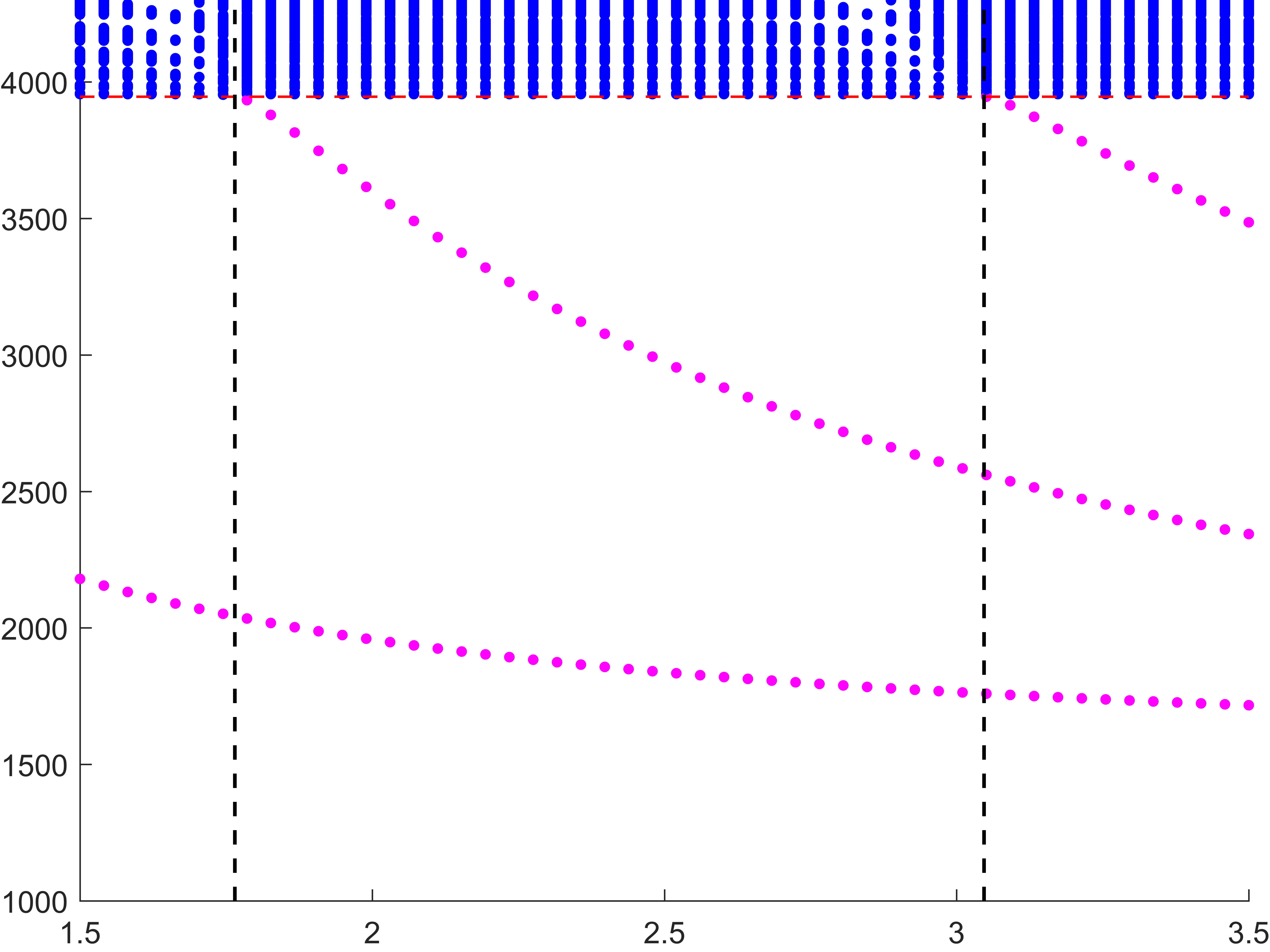}
\caption{Spectrum of $A^\eps$ with respect to $H\in[1.5;3.5]$. The horizontal red dashed line corresponds to the normalized threshold $\pi^2/\eps^2$. The vertical dashed lines marks the values of $H_\star$ appearing in (\ref{defHstar}) such that $\mathbb{S}^{H_\star}$ has an eigenvalue equal to $-1$.}
\label{SpectrumH}
\end{figure}

\begin{figure}[!ht]
\centering
\includegraphics[width=0.89\textwidth]{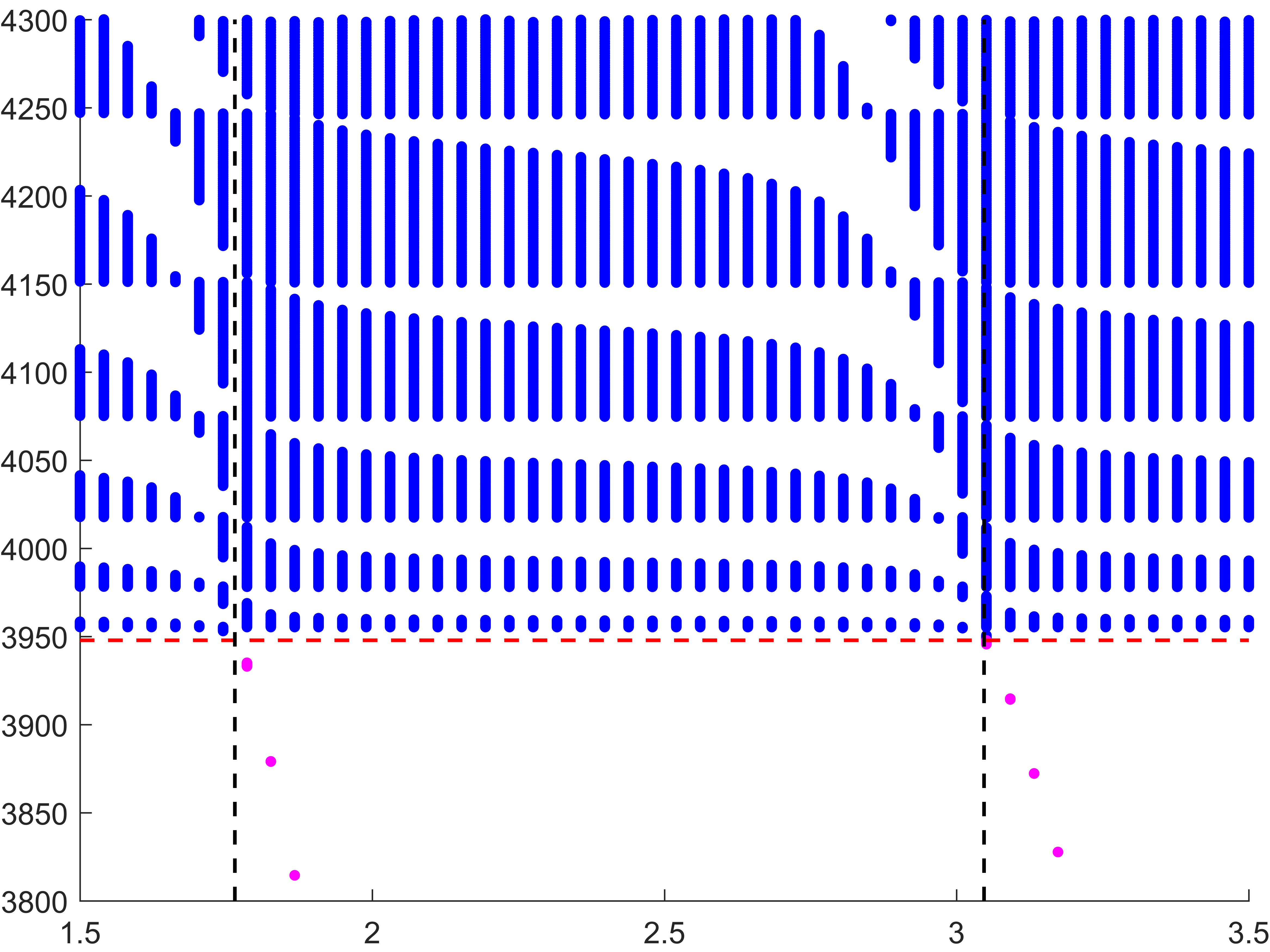}
\caption{Same quantities as in Figure \ref{SpectrumH} with a zoom on the spectral bands closed to $\pi^2/\eps^2$.}
\label{SpectrumHZoom}
\end{figure}

\clearpage
\newpage

\section*{Acknowledgements}
The work of the second author was supported by the Russian Science Foundation, project 22-11-00046.

\bibliography{Bibli}

\def\cprime{$'$}
\begin{thebibliography}{10}

\bibitem{BaMaNa17}
{F.L.} Bakharev, {S.G.} Matveenko, and {S.A.} Nazarov.
\newblock The discrete spectrum of cross-shaped waveguides.
\newblock {\em St. Petersburg Math. J.}, 28(2):171--180, 2017.

\bibitem{BeKu13}
G.~Berkolaiko and P.~Kuchment.
\newblock {\em Introduction to quantum graphs}, volume 186.
\newblock Providence, RI: American Mathematical Society (AMS), 2013.

\bibitem{BiSo87}
{M.Sh.} Birman and {M.Z.} Solomjak.
\newblock {\em Spectral theory of selfadjoint operators in {H}ilbert space}.
\newblock Mathematics and its Applications (Soviet Series). D. Reidel
  Publishing Co., Dordrecht, 1987.

\bibitem{BoPa13}
D.~Borisov and K.~Pankrashkin.
\newblock Quantum waveguides with small periodic perturbations: gaps and edges
  of brillouin zones.
\newblock {\em J. Phys. A Math. Theor.}, 46(23):235203, 2013.

\bibitem{CaEx07}
C.~Cacciapuoti and P.~Exner.
\newblock Nontrivial edge coupling from a {Dirichlet} network squeezing: the
  case of a bent waveguide.
\newblock {\em J. Phys. A, Math. Theor.}, 40(26):f511--g523, 2007.

\bibitem{ChNa23}
L.~Chesnel and {S.A.} Nazarov.
\newblock {Spectrum of the Dirichlet Laplacian in a thin cubic lattice}.
\newblock {\em Math. Model. Numer. Anal.}, 57:3251--3273, 2023.

\bibitem{ChNP18}
L.~Chesnel, {S.A.} Nazarov, and V.~Pagneux.
\newblock Invisibility and perfect reflectivity in waveguides with finite
  length branches.
\newblock {\em SIAM J. Appl. Math.}, 78(4):2176--2199, 2018.

\bibitem{ChPaSu}
L.~Chesnel and V.~Pagneux.
\newblock Simple examples of perfectly invisible and trapped modes in
  waveguides.
\newblock {\em Quart. J. Mech. Appl. Math.}, 71(3):297--315, 2018.

\bibitem{DeZo11}
M.~C. Delfour and J.-P. Zol{\'e}sio.
\newblock {\em Shapes and geometries. {Metrics}, analysis, differential
  calculus, and optimization}, volume~22 of {\em Adv. Des. Control}.
\newblock Philadelphia, PA: Society for Industrial {and} Applied Mathematics
  (SIAM), 2nd ed. edition, 2011.

\bibitem{ExKo15}
P.~Exner and H.~Kova{\v{r}}{\'\i}k.
\newblock {\em Quantum waveguides}.
\newblock Springer, 2015.

\bibitem{Geim09}
{A.K.} Geim.
\newblock Graphene: status and prospects.
\newblock {\em Science}, 324(5934):1530--1534, 2009.

\bibitem{GeNo07}
{A.K.} Geim and {K.S.} Novoselov.
\newblock The rise of graphene.
\newblock {\em Nature materials}, 6(3):183--191, 2007.

\bibitem{Gelf50}
{I.M.} Gelfand.
\newblock Expansion in characteristic functions of an equation with periodic
  coefficients.
\newblock In {\em Dokl. Akad. Nauk SSSR}, volume~73, pages 1117--1120, 1950.

\bibitem{Grie08}
D.~Grieser.
\newblock Spectra of graph neighborhoods and scattering.
\newblock {\em Proc. Lond. Math. Soc.}, 97(3):718--752, 2008.

\bibitem{Grie08INI}
D.~Grieser.
\newblock Thin tubes in mathematical physics, global analysis and spectral
  geometry.
\newblock In {\em Analysis on graphs and its applications. Selected papers
  based on the INI programme, Cambridge, UK, January 8--June 29, 2007}, pages
  565--593. 2008.

\bibitem{Hech12}
F.~Hecht.
\newblock New development in freefem++.
\newblock {\em J. Numer. Math.}, 20(3-4):251--265, 2012.
\newblock \url{http://www3.freefem.org/}.

\bibitem{Henr06}
A.~Henrot.
\newblock {\em Extremum problems for eigenvalues of elliptic operators}.
\newblock {Birkh\"{a}user}, 2006.

\bibitem{HePi18}
A.~Henrot and M.~Pierre.
\newblock {\em Shape variation and optimization. {A} geometrical analysis},
  volume~28 of {\em EMS Tracts Math.}
\newblock Z{\"u}rich: European Mathematical Society (EMS), 2018.

\bibitem{Kato95}
T.~Kato.
\newblock {\em {Perturbation theory for linear operators.}}
\newblock Springer-Verlag, Berlin, reprint of the corr. print. of the 2nd ed.
  1980 edition, 1995.

\bibitem{KoNS16}
{A.I.} Korolkov, {S.A.} Nazarov, and {A.V.} Shanin.
\newblock Stabilizing solutions at thresholds of the continuous spectrum and
  anomalous transmission of waves.
\newblock {\em Z. Angew. Math. Mech.}, 96(10):1245--1260, 2016.

\bibitem{Kuch02}
P.~Kuchment.
\newblock Graph models for waves in thin structures.
\newblock {\em Waves in random media}, 12(4):R1, 2002.

\bibitem{KuchINI}
P.~Kuchment.
\newblock Quantum graphs: an introduction and a brief survey.
\newblock In {\em Analysis on graphs and its applications. Selected papers
  based on the INI programme, Cambridge, UK, January 8--June 29, 2007}, pages
  291--312. Providence, RI: American Mathematical Society (AMS), 2008.

\bibitem{KuPo07}
P.~Kuchment and O.~Post.
\newblock On the spectra of carbon nano-structures.
\newblock {\em Commun. Math. Phys.}, 275(3):805--826, 2007.

\bibitem{KuZe03}
P.~Kuchment and H.~Zeng.
\newblock Asymptotics of spectra of {Neumann} laplacians in thin domains.
\newblock {\em Contemp. Math.}, 327:199--214, 2003.

\bibitem{Kuch82}
{P.A.} Kuchment.
\newblock Floquet theory for partial differential equations.
\newblock {\em Russ. Math. Surv.}, 37(4):1, 1982.

\bibitem{Kuch93}
{P.A.} Kuchment.
\newblock {\em Floquet theory for partial differential equations}, volume~60.
\newblock Springer Science \& Business Media, 1993.

\bibitem{LeWZ19}
{J.P.} Lee-Thorp, {M.I.} Weinstein, and Y.~Zhu.
\newblock Elliptic operators with honeycomb symmetry: Dirac points, edge states
  and applications to photonic graphene.
\newblock {\em Arch. Ration. Mech. Anal.}, 232:1--63, 2019.

\bibitem{MaNP00}
{V.G.} {Maz'ya}, {S.A.} Nazarov, and {B.A.} Plamenevski{\u\i}.
\newblock {\em {Asymptotic theory of elliptic boundary value problems in
  singularly perturbed domains, Vol. 1, 2}}.
\newblock {Birkh\"{a}user}, Basel, 2000.

\bibitem{MoVa07}
S.~Molchanov and B.~Vainberg.
\newblock Scattering solutions in networks of thin fibers: small diameter
  asymptotics.
\newblock {\em Commun. Math. Phys.}, 273(2):533--559, 2007.

\bibitem{MoVa08}
S.~Molchanov and B.~Vainberg.
\newblock Laplace operator in networks of thin fibers: spectrum near the
  threshold.
\newblock In {\em Stochastic analysing in mathematical physics. Proceedings of
  a satellite conference of ICM 2006, Lisbon, Portugal, September 4--8, 2006.
  Selected papers.}, pages 69--94. 2008.

\bibitem{Naza99a}
{S.A.} Nazarov.
\newblock Properties of spectra of boundary value problems in cylindrical and
  quasicylindrical domains.
\newblock In {\em Sobolev spaces in mathematics II}, pages 261--309. Springer,
  2009.

\bibitem{Naza10}
{S.A.} Nazarov.
\newblock Opening of a gap in the continuous spectrum of a periodically
  perturbed waveguide.
\newblock {\em Mathematical Notes}, 87(5-6):738--756, 2010.

\bibitem{Naza14}
{S.A.} Nazarov.
\newblock Bounded solutions in a {T}-shaped waveguide and the spectral
  properties of the dirichlet ladder.
\newblock {\em Comput. Math. Math. Phys.}, 54(8):1261--1279, 2014.

\bibitem{Naza17bis}
{S.A.} Nazarov.
\newblock Almost standing waves in a periodic waveguide with resonator, and
  near-threshold eigenvalues.
\newblock {\em Algebra i Analiz}, 28(3):110--160, 2016.
\newblock (English transl.: Sb. Math. J. 2017. V. 28, N 3. P. 377--410).

\bibitem{Naza16}
{S.A.} Nazarov.
\newblock Transmission conditions in one-dimensional model of a rectangular
  lattice of thin quantum waveguides.
\newblock {\em J. Math. Sci.}, 219(6):994--1015, 2016.

\bibitem{Naza17}
{S.A.} Nazarov.
\newblock The spectra of rectangular lattices of quantum waveguides.
\newblock {\em Izv. Math.}, 81(1):29, 2017.

\bibitem{Naza18}
{S.A.} Nazarov.
\newblock Breakdown of cycles and the possibility of opening spectral gaps in a
  square lattice of thin acoustic waveguides.
\newblock {\em Izv. Math.}, 82(6):1148--1195, 2018.

\bibitem{Naza20Threshold}
{S.A.} Nazarov.
\newblock Threshold resonances and virtual levels in the spectrum of
  cylindrical and periodic waveguides.
\newblock {\em Izv. Math.}, 84(6):1105, 2020.

\bibitem{NazaDoubleTR}
{S.A.} Nazarov.
\newblock Waveguide with double threshold resonance at a simple threshold.
\newblock {\em Sb. Math.}, 211(8):1080, 2020.

\bibitem{Naza23}
{S.A} Nazarov.
\newblock {On the one-dimensional asymptotic models of thin Neumann lattices}.
\newblock {\em Sib. Math. J.}, 64(2):356--373, 2023.

\bibitem{Naza24}
{S.A.} Nazarov.
\newblock Gaps in the spectrum of thin waveguides with periodically locally
  deformed walls.
\newblock {\em Comput. Math. Math. Phys.}, 64(1):99--117, 2024.

\bibitem{NaPl94}
{S.A.} Nazarov and {B.A.} Plamenevski{\u\i}.
\newblock {\em Elliptic problems in domains with piecewise smooth boundaries},
  volume~13 of {\em Expositions in Mathematics}.
\newblock De Gruyter, Berlin, Germany, 1994.

\bibitem{NaRU21}
{S.A.} Nazarov, K.~Ruotsalainen, and {P.J.} Uusitalo.
\newblock Scattering coefficients and threshold resonances in a waveguide with
  inflating resonator.
\newblock {\em Zap. N. S. POMI}, 506:175--209, 2021.

\bibitem{NGPNG09}
C.~Neto, F.~Guinea, N.~Peres, {K.S.} Novoselov, and {A.K.} Geim.
\newblock The electronic properties of graphene.
\newblock {\em Rev. Mod. Phys.}, 81(1):109, 2009.

\bibitem{Pank17}
K.~Pankrashkin.
\newblock Eigenvalue inequalities and absence of threshold resonances for
  waveguide junctions.
\newblock {\em J. Math. Anal. Appl.}, 449(1):907--925, 2017.

\bibitem{Paul36}
L.~Pauling.
\newblock The diamagnetic anisotropy of aromatic molecules.
\newblock {\em J. Chem. Phys.}, 4(10):673--677, 1936.

\bibitem{Post05}
O.~Post.
\newblock {Branched quantum wave guides with Dirichlet boundary conditions: the
  decoupling case}.
\newblock {\em J. Phys. A Math. Theor.}, 38(22):4917, 2005.

\bibitem{Post12}
O.~Post.
\newblock {\em Spectral analysis on graph-like spaces}, volume 2039.
\newblock Springer Science \& Business Media, 2012.

\bibitem{Poyn84}
{J.H.} Poynting.
\newblock Xv. on the transfer of energy in the electromagnetic field.
\newblock {\em Philos. Trans. R. Soc. Lond.}, (175):343--361, 1884.

\bibitem{Skri87}
{M.M.} Skriganov.
\newblock {\em Geometric and arithmetic methods in the spectral theory of
  multidimensional periodic operators}, volume 171.
\newblock 1987.

\end{thebibliography}
\bibliographystyle{plain}
\end{document}